\documentclass{amsart}

\usepackage{amssymb,amsxtra,latexsym}

\usepackage[full]{textcomp}
\usepackage[T1]{fontenc}
\usepackage[osf]{newtxtext}
\usepackage[bigdelims]{newtxmath}
\usepackage[scaled=0.94,scr=rsfs]{mathalfa}
\usepackage[scaled=0.93]{cabin}
\usepackage[varqu,varl]{inconsolata}

\usepackage{multicol}

\makeatletter

\def\theglossary{\@restonecoltrue\if@twocolumn\@restonecolfalse\fi
\columnseprule\z@ \columnsep 35\p@
\let\@makessectionhead\indexsec
\@xp\section\@xp*\@xp{\glossaryname}%
\let\item\@idxitem
\parindent\z@  \parskip\z@\@plus.3\p@\relax
\footnotesize}

\def\glossaryname{Notation Index}

\makeatother

\makeglossary

\usepackage[pagebackref]{hyperref,xcolor}
\definecolor{ceruleanblue}{rgb}{0.16, 0.32, 0.75}
\hypersetup{colorlinks=true,allcolors=ceruleanblue}

\usepackage{tikz-cd}
\tikzcdset{arrow style=math font}
\tikzcdset{diagrams={nodes={inner sep=3pt}}}

\makeatletter \let\@wraptoccontribs\wraptoccontribs \makeatother

\renewcommand{\theenumi}{\roman{enumi}}
\renewcommand{\labelenumi}{(\theenumi)}

\newcommand\note[1]%
                {$^\dagger$\marginpar{\footnotesize{$^\dagger${#1}}}}

\newcounter{num}

\newenvironment{numerate}%
{\begin{list}{\hskip\labelsep\hskip\parindent(\alph{num})}%
{\usecounter{num}%
\setlength\leftmargin{0em}\setlength\topsep{0em}%
\setlength\parsep{0em}\setlength\partopsep{0em}%
\setlength\itemsep{0em}\setlength\labelwidth{0em}}}%
{\end{list}}

\numberwithin{equation}{subsection}

\swapnumbers

\newtheorem{theorem}[equation]{Theorem}
\newtheorem{lemma}[equation]{Lemma}
\newtheorem{proposition}[equation]{Proposition}
\newtheorem{corollary}[equation]{Corollary}

\theoremstyle{definition}
\newtheorem{definition}[equation]{Definition}
\newtheorem{example}[equation]{Example}

\newtheorem{remark}[equation]{Remark}

\newcommand\lie{\mathfrak}

\newcommand{\g}{\lie{g}} 
\newcommand{\h}{\lie{h}}
 
\newcommand{\nnn}{\lie{n}}

\newcommand{\ttt}{\lie{t}} 
\newcommand{\kk}{\lie{k}}
\newcommand{\z}{\lie{z}}

\newcommand\bb{\mathbb}

\newcommand\Z{\bb{Z}} 
\newcommand\Q{\bb{Q}}
\newcommand\R{\bb{R}} 
\newcommand\C{\bb{C}} 

\newcommand\T{\bb{T}}

\newcommand\ca{\mathscr}

\DeclareMathOperator\Aut{Aut}
\DeclareMathOperator\Ad{Ad}
 
\DeclareMathOperator\Alg{Alg}
\DeclareMathOperator\ann{ann}
\DeclareMathOperator\Bad{\text{\bf Ad}}
\DeclareMathOperator\codim{codim}
\DeclareMathOperator\coker{coker}

\DeclareMathOperator\Der{Der}

\DeclareMathOperator\Diff{Diff}

\DeclareMathOperator\Hom{Hom}
\DeclareMathOperator\Hol{Hol}
\DeclareMathOperator\hol{hol}
\DeclareMathOperator\id{id}

\DeclareMathOperator\Iso{Iso}
\DeclareMathOperator\Lie{Lie}
\DeclareMathOperator\Mon{Mon}

\DeclareMathOperator\pr{pr}
\DeclareMathOperator\rank{rank}

\DeclareMathOperator\Vect{Vect}
\DeclareMathOperator\Bvect{\text{\bf Vect}}

\newcommand\group[1]{{\text{\bf#1}}}

\newcommand\U{\group{U}}

\newcommand\B{\group{B}}

\newcommand\G{\group{G}}

\newcommand\abs[1]{\lvert#1\rvert}

\newcommand\inner[1]{\langle#1\rangle}

\newcommand\qu[1][\kern.3ex]{/\kern-.7ex/_{\kern-.4ex#1}}
\newcommand\bigqu[1][\,\,]{\big/\kern-.85ex\big/_{\!\!#1}}

\newcommand\powl{[\kern-.3ex[} \newcommand\powr{]\kern-.3ex]}
\newcommand\bigpowl{\bigl[\kern-.6ex\bigl[}
    \newcommand\bigpowr{\bigr]\kern-.6ex\bigr]}

\newcommand\inj{\hookrightarrow}
\newcommand\sur{\mathrel{\to\kern-1.8ex\to}}

\newcommand\longto{\longrightarrow} 
\newcommand\To{\Rightarrow}
 
\newcommand\Longto{\Longrightarrow}
\newcommand\longhookrightarrow{\lhook\joinrel\longrightarrow}

\newcommand\longinj{\longhookrightarrow}
\newcommand\longsur{\mathrel{\longrightarrow\kern-1.8ex\to}}

\newcommand{\n}{^{-1}}

\newcommand\eps{\varepsilon}
\newcommand\bu{{\scriptscriptstyle\bullet}}

\newcommand\bas{{\mathrm{bas}}}
\newcommand\loc{{\mathrm{loc}}}
\newcommand\mult{{\mathrm{mult}}}

\newcommand\red{{\mathrm{red}}}

\newcommand\Liegpd{\group{LieGpd}}
\newcommand\Stack{\group{Stack}}
\newcommand\Diffstack{\group{DiffStack}}

\newcommand\stack{\group}

\newcommand\X{\stack{X}} 
\newcommand\Y{\stack{Y}}
\newcommand\HH{\stack{H}}
\newcommand\KK{\stack{K}}
\newcommand\TT{\stack{T}}

\newcommand\stackmorphism{\boldsymbol}
\newcommand\bo{\boldsymbol}

\newcommand\bunit{{\stackmorphism1}}
\newcommand\bzero{{\stackmorphism0}}
\newcommand\ba{{\stackmorphism{a}}}
\newcommand\bm{{\stackmorphism{m}}}
\newcommand\bp{{\stackmorphism{p}}}
\newcommand\bi{{\stackmorphism{i}}}

\newcommand\balpha{{\stackmorphism{\alpha}}}
\newcommand\bbeta{{\stackmorphism{\beta}}}
\newcommand\bgamma{{\stackmorphism{\gamma}}}

\newcommand\bzeta{{\stackmorphism{\zeta}}}

\newcommand\blambda{{\stackmorphism{\lambda}}}
\newcommand\bmu{{\stackmorphism{\mu}}}

\newcommand\bxi{{\stackmorphism{\xi}}}

\newcommand\bphi{{\stackmorphism{\phi}}}
\newcommand\bchi{{\stackmorphism{\chi}}}
\newcommand\bpsi{{\stackmorphism{\psi}}}
\newcommand\bomega{{\stackmorphism{\omega}}}

\newcommand\Bomega{{\stackmorphism{\Omega}}}

\newcommand\F{{\ca{F}}}

\begin{document}


\title{Stacky Hamiltonian actions and symplectic reduction}

\author{Benjamin Hoffman}

\address{Department of Mathematics, Cornell University, Ithaca, NY
  14853-4201, USA}

\email{bsh68@cornell.edu}

\author{Reyer Sjamaar}

\address{Department of Mathematics, Cornell University, Ithaca, NY
  14853-4201, USA}

\email{sjamaar@math.cornell.edu}

\contrib[with an appendix by]{Chenchang Zhu}

\address{Mathematisches Institut, Georg-August-Universit\"at,
  Bunsenstra{\ss}e 3--5, D-37073 G\"ottingen, Germany}

\email{zhu@uni-math.gwdg.de}




\begin{abstract}
We introduce the notion of a Hamiltonian action of an \'etale Lie
group stack on an \'etale symplectic stack and establish versions of
the Kirwan convexity theorem, the Meyer-Marsden-Weinstein symplectic
reduction theorem, and the Duistermaat-Heckman theorem in this
context.
\end{abstract}


\maketitle

\tableofcontents


\section{Introduction}

The leaf space of a foliation on a smooth manifold can be interpreted
as a differentiable stack.  Many invariants of the foliation, such as
its K-theory and cyclic homology, depend only on this stack.  In this
paper we are concerned with transversely symplectic foliations and
with some properties of their associated stacks.

For instance, work of He~\cite{he;odd-dimensional},
Ishida~\cite{ishida;transverse-kaehler}, Ratiu and
Zung~\cite{ratiu-zung;presymplectic}, and Lin and
Sjamaar~\cite{lin-sjamaar;presymplectic} shows that a version of
Kirwan's convexity theorem~\cite{kirwan;convexity-III} for the moment
map in symplectic geometry holds for certain transversely symplectic
foliations.  Our first main result, Theorem~\ref{theorem;convex},
upgrades their results to a convexity theorem for Hamiltonian actions
of \'etale Lie group stacks on \'etale symplectic stacks.  To avoid
technicalities, let us here state our theorem in the case where the
Lie group stack is a stacky torus $\stack{T}$.  Our assertion is that
if $\stack{T}$ acts on a symplectic stack $\X$ and if the action
admits a moment map, then under a certain ``cleanness'' hypothesis the
image of the moment map is a convex polytope. This statement extends
the Atiyah-Guillemin-Sternberg convexity theorem, which holds in the
case where $\X$ is a symplectic manifold and $\stack{T}$ an ordinary
torus, but the stacky situation has an interesting new feature, namely
that the normal fan of the moment polytope is not necessarily
rational.  More precisely, the normal vectors to the polytope are
cocharacters of $\stack{T}$, and the cocharacter group of a stacky
torus is not a lattice, but a quasi-lattice in the Lie algebra.  We
call a pair consisting of a stacky torus and a convex polytope in the
dual of its Lie algebra a \emph{stacky polytope}.

The second main result of this paper is
Theorem~\ref{theorem;reduction}, which generalizes the
Meyer-Marsden-Weinstein symplectic reduction theorem to the setting of
group stack actions on symplectic stacks.  We give a necessary and
sufficient condition for the reduction of a symplectic stack by a
Hamiltonian action of an \'etale Lie group stack to be again a
symplectic stack.  The theorem holds under a regularity hypothesis on
the moment map, but we make no assumption on the compactness of the
group stack or the properness of the action.  This generalizes a
theorem of Lerman and
Malkin~\cite[Theorem~3.13]{lerman-malkin;deligne-mumford}, who
considered the case of a Hamiltonian action of a compact Lie group on
a separated symplectic stack.  In view of work of
Calaque~\cite{calaque;derived-symplectic-topological-field},
Pecharich~\cite{pecharich;derived-marsden-weinstein}, and
Safronov~\cite{safronov;quasi-hamiltonian-chern-simons} we expect that
in the absence of any regularity assumptions the reduction of a
symplectic stack by the action of a Lie group stack is a derived
symplectic stack.

Our third main result is Theorem~\ref{theorem;DH}, which is an
extension of the Duistermaat-Heckman theorem to Hamiltonian actions of
stacky tori. The Duistermaat-Heck\-man theorem has two parts: (1)~the
variation of the reduced symplectic form is linear, and (2)~the moment
map image of the Liouville measure is piecewise polynomial.  It is
only the first part that we generalize here, leaving the second part
for later.

These results require a basic theory of Hamiltonian actions of Lie
group stacks, which we outline in
Sections~\ref{section;lie-stack}--\ref{section;hamiltonian-groupoid}
and in Appendices~\ref{appendix;groupobjects}-\ref{appendix;zhu}.  We
include some elementary, but apparently new, material concerning
\'etale stacks and stacky Lie groups, such as the fact that the Lie
$2$-algebra of vector fields of a differentiable stack is equivalent
to a Lie algebra if the stack is \'etale
(Proposition~\ref{proposition;vector-stack}), a structure theorem for
Lie $2$-groups of compact type
(Proposition~\ref{proposition;compact-connected}), and a
strictification theorem for stacky actions
(Theorem~\ref{theorem;strictaction}).  Our starting point is the
theory developed by Lerman and
Malkin~\cite{lerman-malkin;deligne-mumford}, which we extend in two
respects: the stacks that we deal with are \'etale, but usually not
separated, and the groups that act on them are themselves \'etale
stacks.

What we call a symplectic stack is a \emph{$0$-shifted symplectic
  $1$-stack} in the terminology of Pantev et
al.~\cite{pantev-toen-vaquie-vezzosi;shifted-symplectic}, except that
our stacks are defined over the category of differentiable manifolds
instead of the category of algebraic schemes.  See Getzler's lecture
notes~\cite{getzler;differential-stacks} for an introduction to higher
symplectic stacks over manifolds and for an explanation of how
Weinstein's symplectic
groupoids~\cite{weinstein;symplectic-groupoids-poisson-manifolds} and
Xu's quasi-symplectic groupoids~\cite{xu;momentum-morita} are
presentations of $1$-shifted symplectic $1$-stacks.  We expect that
some of our results, especially the reduction theorem, can be extended
to higher symplectic stacks.

An illuminating example of stacky symplectic reduction is Prato's
construction of toric quasifolds~\cite{prato;non-rational-symplectic},
which predates many of these developments.  We present (a slight
extension of) her construction as a running example in order to show
that every simple stacky polytope is the moment polytope of what we
call a \emph{toric} symplectic stack.  Toric symplectic stacks are a
$C^\infty$ counterpart of the toric stacks of algebraic geometry, a
comprehensive treatment of which was given by Gerashchenko and
Satriano~\cite{geraschenko-satriano;toric}.  However, the two theories
are very different.  The correspondence between toric symplectic
manifolds and nonsingular complex projective toric varieties
established by Delzant~\cite{delzant;hamiltoniens-periodiques} breaks
down in the world of stacks, because $C^\infty$ stacky tori, which
include such objects as the quotient of a two-dimensional torus by a
dense line, are seldom algebraic.  Toric symplectic stacks are
classified by the first author in~\cite{hoffman;toric-stacks}.

Two further instances of stacky symplectic reduction are the space of
geodesics of a Riemannian manifold (see
Example~\ref{example;geodesic}) and the quotient of a contact manifold
by the Reeb flow (see Example~\ref{example;contact}).

We are grateful to Chenchang Zhu for her contribution to this project,
which appears in Appendix~\ref{appendix;zhu}.  We also thank the
referees for their careful reading and many useful suggestions.

\section{Notation and conventions}\label{section;notation}

All manifolds are required to be $C^\infty$ and second countable, but
not necessarily Hausdorff.  Manifolds and smooth maps form the
category $\group{Diff}$.
\glossary{Diff@$\group{Diff}$, category of manifolds}%
Lie groups are group objects in $\group{Diff}$, in other words are
required to have countably many components. The functor taking a Lie
group to its Lie algebra is denoted $\Lie$. The space of smooth global
sections of a vector bundle $E$ is denoted by $\Gamma(E)$.  
\glossary{Lie@$\Lie$, Lie functor}%
\glossary{Gamma@$\Gamma(E)$, space of smooth sections of $E$}%
The Lie algebra of vector fields on a manifold $X$ is denoted by
$\Vect(X)=\Gamma(TX)$.
\glossary{Vect@$\Vect(X)$, vector fields on $X$}%
Given a smooth map $f\colon X\to Y$, two vector fields $v\in\Vect(X)$
and $w\in\Vect(Y)$ are \emph{$f$-related} (notation: $v\sim_fw$) if
$T_xf(v_x)=w_{f(x)}$ for all $x\in X$.
\glossary{*@$\sim_\phi$, $\phi$-relatedness of vector fields or
  sections}%
A Lie groupoid with object manifold $X_0$ and arrow manifold $X_1$ is
denoted by $X_1\rightrightarrows X_0$ or $X_\bu$.  Morphisms of Lie
groupoids $\phi_\bu\colon X_\bu\to Y_\bu$ are also decorated with a
subscript ``$\bu$'', as are basic differential forms
$\zeta_\bu\in\Omega_\bas^\bu(X_\bu)$ and basic vector fields
$v_\bu\in\Vect_\bas(X_\bu)$.  We write weak (Morita) equivalence of
groupoids as $X_\bu\simeq Y_\bu$.  Stacks over $\group{Diff}$ are
written in boldface, $\X$, $\Y$, and so are their $1$-morphisms
$\bphi\colon\X\to\Y$, $2$-morphisms $\balpha\colon\bphi\To\bpsi$, and
differential forms $\bzeta\in\Omega^\bu(\X)$, etc.  The classifying
stack of a Lie groupoid $X_\bu$ is denoted by $\B X_\bu$.  We write
equivalences of stacks as $\X\simeq\Y$.  We denote by $\star$ a
terminal object in the $2$-category of stacks.
\glossary{*@$\simeq$, Morita equivalence of groupoids or equivalence
  of stacks}
\glossary{*@$\star$, terminal object in $\Stack$}

See also the notation index at the end.

\section{Hamiltonian actions on presymplectic manifolds}
\label{section;presymplectic}

\numberwithin{equation}{section}

This section is a brief exposition of the presymplectic convexity
theorem~\cite[Theorem~2.2]{lin-sjamaar;presymplectic}, which is the
prototype of our stacky convexity theorem,
Theorem~\ref{theorem;convex}.

A \emph{presymplectic manifold} is a Hausdorff manifold $X$ equipped
with a closed $2$-form $\omega$ of constant rank.  The kernel
$\ker(\omega)$ defines an involutive distribution on $X$.  The
corresponding foliation of $X$ is called the \emph{null foliation},
which we denote by $\F$.  For $x\in X$, we write $\F(x)$ for the leaf
of $\F$ containing $x$.
\glossary{F@$\F$, foliation}%
\glossary{Fx@$\F(x)$, $\F$-leaf of $x$}%
\glossary{omega@$\omega$, (pre)symplectic form on manifold}%

Consider a left action of a connected Lie group $G$ on a presymplectic
manifold $(X,\omega)$.  For $\xi\in\g=\Lie(G)$, denote by $\xi_X$ the
fundamental vector field of $\xi$ on $X$; then the assignment
$\xi\mapsto \xi_X$ is a Lie algebra anti-homomorphism.  Let
\glossary{xix@$\xi_X$, fundamental vector field on $G$-manifold $X$
  induced by $\xi\in\Lie(G)$}%
\[
\nnn(\F)=\{\,\xi\in\g\mid\text{$(\xi_X)_x\in T_x\F$ for all $x\in
  X$}\,\}.
\]
The subspace $\nnn(\F)\subseteq\g$ is an ideal in $\g$, which we call
the \emph{null ideal} of $\F$,
following~\cite{lin-sjamaar;presymplectic}.
\glossary{nF@$\nnn(\F)$, null ideal of action on foliated manifold}%
Let $N(\F)\subseteq G$ be the connected immersed Lie subgroup of $G$
with Lie algebra $\nnn(\F)$, which we will call the $\emph{null
  subgroup}$.  The action of $G$ on $(X,\omega)$ is \emph{clean} if
\[T_x(N(\F)\cdot x)=T_x(G\cdot x)\cap T_x\F\]
for all $x\in X$.  For a Lie subalgebra $\h$ of $\g$, let
\[\ann(\h)=\{\,\eta\in\g^*\mid\inner{\eta,\h}=0\,\}\cong(\g/\h)^*\]
\glossary{ann@$\ann$, annihilator of subspace}%
be the annihilator of $\h$, with
$\inner{{\cdot},{\cdot}}\colon\g^*\times\g\to\R$ being the natural
pairing.  The action of $G$ on $(X,\omega)$ is
\emph{Hamiltonian} if there is a \emph{moment map}, i.e.\ a map
$\mu\colon X\to\g^*$ that satisfies the following conditions:
\glossary{mu@$\mu$, moment map for Lie group action on (pre)symplectic
  manifold}%
\begin{enumerate}
\item\label{item;hamilton}
$d\mu^\xi=\iota_{\xi_X}\omega$ for all $\xi\in\g$, where
  $\mu^\xi(x)=\inner{\mu(x),\xi}$ denotes the component of $\mu$ along
  $\xi$;
\item\label{item;coadjoint}
$\mu$ intertwines the $G$ action on $X$ and the coadjoint action of
  $G$ on $\g^*$;
\item\label{item;null}
$\mu(X)\subseteq\ann(\nnn(\F))$.
\end{enumerate}
(Since $G$ is connected, these assumptions imply that $\omega$ is
$G$-invariant.)  The tuple $(X,\omega,G,\mu)$ is then called a
\emph{presymplectic Hamiltonian $G$-manifold}.  An \emph{isomorphism}
between two Hamiltonian $G$-manifolds $(X,\omega,G,\mu)$ and
$(X',\omega',G,\mu')$ is a $G$-equivariant diffeomorphism $\phi\colon
X\to X'$ which preserves the presymplectic structure and the moment
map.

The conditions~\eqref{item;hamilton}--\eqref{item;null} on the moment
map are not independent.
by~\cite[proposition~2.9.1]{lin-sjamaar;presymplectic}, if $G$ is
compact and if $\mu\colon X\to\g^*$ satisfies~\eqref{item;hamilton}
and~\eqref{item;coadjoint}, there exists $\lambda\in\g^*$ fixed by the
coadjoint action of $G$ such that $\mu+\lambda$
satisfies~\eqref{item;hamilton}--\eqref{item;null}.

\begin{example}\label{example1}
This example is drawn from~\cite{delzant;hamiltoniens-periodiques}
and~\cite{prato;non-rational-symplectic}.  Let $\T$ be the circle
$\R/\Z$, $G$ the $n$-dimensional torus $\T^n$, and $N\subseteq G$ an
immersed Lie subgroup.  Consider the Hamiltonian $G$-manifold
$(X,\omega,G,\mu)$, where
\glossary{T@$\T$, circle $\R/\Z$}%
\[
X=\C^n,\qquad\omega=\frac{1}{2\pi i}\sum_jdz_j\wedge
d\bar{z}_j,\qquad\mu(z)=\sum_j\abs{z_j}^2e^*_j+\lambda,
\]
and $G$ acts on $\C^n$ in the standard way.  Here $e^*_j$ is the dual
of the standard basis of $\R^n=\g$, and $\lambda\in\g^*$ is in the
open negative orthant.  Let $\iota\colon \lie{n}\to\g$ be the
inclusion of Lie algebras and $\iota^*\colon \g^*\to\lie{n}^*$ the
dual projection.  Then $(X_0,\omega_0,G,\mu_0)$ is a presymplectic
Hamiltonian $G$-manifold, where
\[
X_0=(\iota^*\circ\mu)^{-1}(0),\quad\omega_0=\omega|_{X_0},\quad
G=\T^n,\quad\mu_0=\mu|_{X_0}.
\]
Note that $\mu_0$ takes values in $\ann(\lie{n})$.  We will return to
this example throughout the text.
\end{example}

The cleanness condition is essential for the following to be true.

\begin{theorem}[Lin and Sjamaar~\cite{lin-sjamaar;presymplectic}]
\label{theorem;LiSj}
Let $(X,\omega,G,\mu)$ be a Hamiltonian presymplectic $G$-manifold,
where $X$ is connected, and $G$ is compact and connected.  Assume that
the $G$-action is clean, and the moment map $\mu\colon X\to\g^*$ is
proper.  Choose a maximal torus $T$ of $G$ and a closed Weyl chamber
$C$ in $\ttt^*$, where $\ttt=\Lie(T)$, and define
$\Delta(X)=\mu(X)\cap C$.
\glossary{DeltaX@$\Delta(X)$, moment body of $X$}%
\begin{enumerate}
\item
The fibres of $\mu$ are connected and $\mu\colon X\to\mu(X)$ is an
open map.
\item
$\Delta(X)$ is a closed convex polyhedral set.
\item
$\Delta(X)$ is rational if and only if the null subgroup $N(\F)$ of
  $G$ is closed.
\end{enumerate}
\end{theorem}

\section{Lie groupoids and differentiable stacks}
\label{section;groupoid-stack}

\numberwithin{equation}{subsection}

This section is a summary of definitions, conventions, and well-known
facts.  For more about Lie groupoids see e.g.\ Moerdijk and
Mr\v{c}un~\cite{moerdijk-mrcun;foliations-groupoids} or Crainic and
Moerdijk~\cite{crainic-moerdijk;foliation-cyclic}.  For the
relationship between Lie groupoids and differentiable stacks see
e.g.\ Behrend and Xu~\cite{behrend-xu;stacks-gerbes},
Blohmann~\cite{blohmann;stacky-lie-groups},
Carchedi~\cite{carchedi;topological-differentiable-stacks},%
~\cite{carchedi;etale-stacks},
Lerman~\cite{lerman;orbifolds-as-stacks},
Metzler~\cite{metzler;topological-smooth-stack},
Noohi~\cite{noohi;foundations-topological-stacks}, or
Villatoro~\cite{villatoro;stacks-poisson}.  See also the notation
index at the end.

\subsection{Lie groupoids}\label{subsection;lie-groupoid}

A Lie groupoid $X_\bu=(X_1\rightrightarrows X_0)$ has structure maps
$s$, $t$, $m$, $(\cdot)\n$, and $u$ which are called \emph{source,
  target, multiplication, inversion}, and the \emph{identity
  bisection}, respectively.  When two arrows $f$, $g\in X_1$ have
$s(f)=t(g)$, they are \emph{composable}.  We typically write $f\circ
g$ for the multiplication $m(f,g)$ of two composable arrows.  The
\emph{object manifold} $X_0$ and the \emph{arrow manifold} $X_1$ are
(not necessarily Hausdorff) manifolds, and the maps $s$ and $t$ are
required to be surjective submersions.  If all source fibres
$s^{-1}(x)$ are connected (resp.\ simply connected), then $X_\bu$ is
\emph{source-connected} (resp.\ \emph{source-simply connected}).  For
$x\in X_0$, the \emph{orbit} of $x$ is $X_\bu\cdot x=t(s^{-1}(x))$,
and the \emph{isotropy group} of $x$ is
\[\Iso(x)=\Iso_{X_\bu}(x)=s\n(x)\cap t\n(x).\]
It is known (see
e.g.~\cite[Theorem~5.4]{moerdijk-mrcun;foliations-groupoids}) that
$X_\bu\cdot x$ is an immersed submanifold of $X_0$, that $\Iso(x)$ is
a closed submanifold of $X_1$, and that $\Iso(x)$ is a Lie group.  The
set of orbits equipped with the quotient topology is the \emph{orbit
  space} or \emph{coarse quotient space} $X_0/X_1$ of the Lie
groupoid.
\glossary{X*@$X_\bu$, Lie groupoid}%
\glossary{X1X0@$X_1\rightrightarrows X_0$, Lie groupoid}%
\glossary{X0@$X_0$, objects of $X_\bu$}%
\glossary{X1@$X_1$, arrows of $X_\bu$}%
\glossary{Isox@$\Iso_{X_\bu}(x)$, $X_\bu$-isotropy group of $x$}%
\glossary{X*dotx@$X_\bu\cdot x$, $X_\bu$-orbit of $x$}%
\glossary{X0/X1@$X_0/X_1$, orbit space of $X_\bu$}%
\glossary{s@$s$, source map of groupoid}%
\glossary{t@$t$, target map of groupoid}%
\glossary{m@$m$, multiplication map of groupoid}%
\glossary{u@$u$, identity bisection of groupoid}%
\glossary{*@$(\cdot)\n$, inversion law of groupoid or $2$-group}%

If a Lie group $G$ acts on a manifold $X$, we denote the action
groupoid by $G\ltimes X\rightrightarrows X$.  We sometimes denote a
Lie groupoid by its space of arrows; for instance in this notation the
action groupoid is $G\ltimes X$.  If $X$ is a smooth manifold, we
consider it as the identity Lie groupoid $X\rightrightarrows X$.  For
any Lie groupoid $X_\bu$ we have a natural inclusion $X_0\to X_\bu$ of
(the identity groupoid of) $X_0$ into $X_\bu$.
\glossary{G*X@$G\ltimes X$, action groupoid of $G$-action on manifold
  $X$}%

\begin{definition}\label{definition;Liealgebroid} 
The \emph{Lie algebroid} of a Lie groupoid $X_\bu$ is the vector
bundle $\Alg(X_\bu)$ over $X_0$ given by
\[\Alg(X_\bu)=\{\,w\in TX_1|_{u(X_0)}\mid Ts(w)=0\,\}.\]
The \emph{anchor map} is the vector bundle morphism $\rho\colon
\Alg(X_\bu)\to T X_0$ given by $\rho=Tt|_{\Alg(X_\bu)}$.
\glossary{Alg@$\Alg(X_\bu)$, Lie algebroid of $X_\bu$}%
\end{definition}

Sections of the Lie algebroid extend uniquely to right-invariant
vector fields on $X_1$, and so the space of sections of $\Alg(X_\bu)$
carries a natural Lie bracket.

\begin{definition}\label{definition;categoryofgroupoids}
A \emph{morphism of Lie groupoids} $\phi_\bu\colon X_\bu\to Y_\bu$ is
a smooth functor, i.e.\ a morphism of groupoids which is smooth on the
manifolds of objects $\phi_0\colon X_0\to Y_0$ and on the manifolds of
arrows $\phi_1\colon X_1\to Y_1$.  For two morphisms $\phi_\bu$,
$\psi_\bu\colon X_\bu\to Y_\bu$ of Lie groupoids, a
\emph{$2$-morphism} or \emph{natural transformation}
$\gamma\colon\phi_\bu\To\psi_\bu$ is a smooth map $\gamma\colon X_0\to
Y_1$ with the property that for each $x\in X_0$, $\gamma(x)$ is an
arrow from $\phi_0(x)$ to $\psi_0(x)$ in $Y_\bu$, and for every arrow
$f\colon x_1\to x_2$ in $X_\bu$, the following diagram commutes in
$Y_\bu$:
\[
\begin{tikzcd}[row sep=large]
\phi_0(x_1)\ar[r,"\gamma(x_1)"]\ar[d,"\phi_1(f)"']&
\psi_0(x_1)\ar[d,"\psi_1(f)"]
\\
\phi_0(x_2)\ar[r,"\gamma(x_2)"]&\psi_0(x_2)
\end{tikzcd}
\]
The $2$-category of Lie groupoids is denoted $\Liegpd$.
\glossary{Liegpd@$\Liegpd$, $2$-category of Lie groupoids}%
\end{definition}

A natural transformation of Lie groupoids is automatically a natural
isomorphism.

\begin{definition}\label{definition;moritaequivalence}
Let $\phi_\bu\colon X_\bu\to Y_\bu$ be a morphism of Lie groupoids.
Then $\phi_\bu$ is \emph{essentially surjective} if the map
$t\circ\pr_1\colon Y_1\times_{Y_0}X_0\to Y_0$ which sends $(g,x)$ to
$t(g)$ is a surjective submersion.  Here
$Y_1\times_{Y_0}X_0=Y_1\times_{s,Y_0,\phi_0}X_0$ means as usual the
fibred product of $Y_1$ and $X_0$, which consists of all $(g,x)\in
Y_1\times X_0$ satisfying $s(g)=\phi_0(x)$.  The morphism $\phi_\bu$
is \emph{fully faithful} if the square
\[
\begin{tikzcd}[row sep=large]
X_1\ar[r,"\phi_1"]\ar[d,"{(s,t)}"']&Y_1\ar[d,"{(s,t)}"]\\
X_0\times X_0\ar[r,"\phi_0\times\phi_0"]&Y_0\times Y_0
\end{tikzcd}
\]
is a fibred product of manifolds.  If $\phi_\bu$ is both essentially
surjective and fully faithful, we say that $\phi_\bu$ is a
\emph{Morita morphism} or \emph{weak equivalence}.  If there is a
zigzag of Morita morphisms $X_\bu\to Y_\bu$ and $X_\bu\to Z_\bu$, then
$Y_\bu$ and $Z_\bu$ are \emph{Morita equivalent}.
\glossary{*@$\simeq$, Morita equivalence of groupoids or equivalence
  of stacks}%
\end{definition}

\begin{definition}\label{definition;pullbackgroupoid}
Given a Lie groupoid $Y_\bu$ and a smooth map $\phi_0\colon X\to Y_0$,
the \emph{pullback groupoid} $X_\bu=\phi_0^*(Y_\bu)$ is defined by
$X_0=X$ and
\begin{align*}
X_1&=X\times_{\phi_0,Y_0,s} Y_1\times_{t,Y_0,\phi_0}X\\
&\hphantom{:}=\{\,(x,g,y)\in X\times Y_1\times
X\mid\phi_0(x)=s(g)\text{ and }t(g)=\phi_0(y)\,\}.
\end{align*}
\glossary{phiX*@$\phi_0^*X_\bu$, pullback groupoid}%
\end{definition}

The pullback groupoid is a Lie groupoid whenever
$(\phi_0,\phi_0)\colon X\times X\to Y_0\times Y_0$ is transverse to
$(s,t)\colon Y_1\to Y_0\times Y_0$.  The map $\phi_0$ then lifts to a
Lie groupoid homomorphism $\phi_\bu\colon\phi_0^*(Y_\bu)\to Y_\bu$,
which is fully faithful, and which is essentially surjective if and
only if $t\circ\pr_1\colon Y_1\times_{Y_0} X\to Y_0$ is surjective.

\begin{lemma}\phantomsection\label{lemma;hausdorff}
\begin{enumerate}
\item\label{item;hausdorff}
For every Lie groupoid $Y_\bu$ there is a Morita morphism $X_\bu\to
Y_\bu$ from a Lie groupoid $X_\bu$ with Hausdorff object manifold
$X_0$.
\item\label{item;source-hausdorff}
Let $X_\bu$ be a Lie groupoid with Hausdorff object manifold $X_0$.
Then for every $x\in X_0$ the source fibre $s^{-1}(x)$ is Hausdorff.
\end{enumerate}
\end{lemma}

\begin{proof}
\eqref{item;hausdorff}~Choose any surjective \'etale map $\phi_0\colon
X_0\to Y_0$ from a Hausdorff manifold $X_0$ to $Y_0$ (e.g.\ $X_0$ is
the disjoint union of charts in a countable atlas of $Y_0$); then
$X_\bu=\phi_0^*Y_\bu$ is Morita equivalent to $Y_\bu$ and has object
manifold $X_0$.

\eqref{item;source-hausdorff}~Let $x\in X_0$ and let $t_x$ be the
restriction of the target map to the source fibre $s^{-1}(x)$.
According
to~\cite[Theorem~5.4(iv)]{moerdijk-mrcun;foliations-groupoids},
$t_x\colon s^{-1}(x)\to X_\bu\cdot x$ is a locally trivial principal
bundle with structure group $\Iso(x)$.  The orbit $X_\bu\cdot
x\subseteq X_0$ is Hausdorff and so is the Lie group $\Iso(x)$, and
therefore the total space $s^{-1}(x)$ is Hausdorff.
\end{proof}

\begin{definition}\label{definition;fibreproductLiegpd} 
Let $\phi_\bu\colon X_\bu\to Z_\bu$ and $\psi_\bu\colon Y_\bu\to
Z_\bu$ be Lie groupoid morphisms.  The \emph{weak fibred product} is
the topological groupoid $X_\bu\times_{Z_\bu}^{(w)} Y_\bu$ whose space
of objects is
\[
X_0\times_{Z_0}Z_1\times_{Z_0}Y_0=\{\,(x,k,y)\in X_0\times Z_1\times
Y_0\mid\phi_0(x)=s(k),~\psi_0(y)=t(k)\,\},
\]
whose space of arrows is
\[
X_1\times_{Z_0}Z_1\times_{Z_0}Y_1=\{\,(f,k,g)\in X_1\times Z_1\times
Y_1\mid\phi_0(s(f))=s(k),~\psi_0(s(g))=t(k)\,\},
\]
and whose groupoid structure maps are as
in~\cite[\S\,5.3]{moerdijk-mrcun;foliations-groupoids}.
\glossary{X*Z*Y*@$X_\bu\times_{Z_\bu}^{(w)} Y_\bu$, weak fibred
  product}%
\end{definition}

If either $\phi_0$ or $\psi_0$ is a submersion, the weak fibred
product is a Lie groupoid, and it is a weak pullback in $\Liegpd$.

\subsection{Differentiable stacks}\label{subsection;stack}

We consider stacks over the site $\group{Diff}$ of ($C^\infty$ and
second countable, but not necessarily Hausdorff) manifolds, equipped
with the open cover Grothendieck topology.  

Stacks and their $1$-morphisms and $2$-morphisms form a strict
$(2,1)$-category $\Stack$.  An \emph{atlas} of a stack $\X$ is a
representable epimorphism $X\to\X$ from a manifold $X$ to $\X$.
A \emph{differentiable stack} is a stack that admits an atlas.  Every
differentiable stack $\X$ admits an atlas $X\to\X$ with $X$ Hausdorff
(namely starting from an arbitrary atlas $Y\to\X$, let $X$ be the
disjoint union of all charts in a countable atlas of $Y$).  We denote
the full subcategory of $\Stack$ formed by all differentiable stacks
by $\Diffstack$.
\glossary{Stack@$\Stack$, $2$-category of stacks over $\group{Diff}$}
\glossary{Xx@$\X$, stack over $\group{Diff}$}%
\glossary{Diffstack@$\Diffstack$, $2$-category of differentiable stacks}%

The \emph{classifying stack} $\B X_\bu$ of a Lie groupoid $X_\bu$ is
the stack of right torsors (right principal bundles) of $X_\bu$.
The \emph{classifying functor} is the $2$-functor
$\B\colon\Liegpd\to\Diffstack$ that takes $X_\bu$ to $\B X_\bu$.
(See~\cite[\S\,2]{blohmann;stacky-lie-groups},
\cite[\S\,3]{metzler;topological-smooth-stack},
\cite[\S\,I.2]{carchedi;topological-differentiable-stacks},
or~\cite[\S\,2]{carchedi;etale-stacks}.)  A \emph{presentation} of a
stack $\X$ is an equivalence $\B X_\bu\simeq\X$, where $X_\bu$ is a
Lie groupoid.
\glossary{B@$\B$, classifying functor}%

For a Lie groupoid $X_\bu$, the natural inclusion of Lie groupoids
$X_0\to X_\bu$ induces a morphism of stacks $X_0\to\B X_\bu$, which is
an atlas for $\B X_\bu$. Conversely, if $\X$ is a stack and
$\bphi\colon X_0\to\X$ is an atlas, then there exists an equivalence
$\B X_\bu\simeq\X$ such that $\bphi$ is $2$-isomorphic to the
composition
$\begin{tikzcd}[cramped,sep=small]X_0\ar[r]&\B
  X_\bu\ar[r,"\simeq"]&\X\end{tikzcd}$,
where $X_\bu$ is the Lie groupoid $X_0\times_\X X_0\rightrightarrows
X_0$.  (See
e.g.\ \cite[\S\,I.2.7]{carchedi;topological-differentiable-stacks}
or~\cite[\S\,3.2]{noohi;foundations-topological-stacks}.)  Thus a
stack is differentiable if and only if it admits a presentation.

The functor $\B$ takes essentially surjective Lie groupoid morphisms
to stack epimorphisms, and it takes fully faithful Lie groupoid
morphisms to stack monomorphisms. Thus $\B$ takes Morita morphisms to
equivalences of stacks. By the universal property of $2$-localization
(see~\cite[\S\,2]{pronk;etendues-fractions}), $\B$ extends to a
$2$-functor
\begin{equation}\label{equation;groupoid-stack}
\B^\loc\colon\Liegpd\bigl[\ca{M}\n\bigr]\stackrel\simeq\longto\Diffstack,
\glossary{LiegpdM@$\Liegpd\bigl[\ca{M}\n\bigr]$, $2$-localization of
  Lie groupoids at Morita morphisms}%
\end{equation}
from the bicategory of Lie groupoids, localized at the class $\ca{M}$
of all Morita morphisms, to the $2$-category $\Diffstack$.  The
functor $\B^\loc$ is an equivalence of bicategories; a result of this
type in a topological context was proved in
\cite{pronk;etendues-fractions} and a proof in our specific context is
given in \cite[Theorem 1.3.27]{villatoro;stacks-poisson}.  The
equivalence~\eqref{equation;groupoid-stack} implies that for all Lie
groupoids $X_\bu$ and $Y_\bu$ we have an equivalence of categories
$\B^\loc\colon\Hom(X_\bu,Y_\bu)\to\Hom(\B X_\bu,\B Y_\bu)$, where the
first Hom is taken in $\Liegpd[\ca{M}\n]$.

\begin{lemma}[Behrend-Xu {\cite[Theorem~2.2]{behrend-xu;stacks-gerbes}}]
\label{lemma;moritaequivandstacks} 
Two Lie groupoids are Morita equivalent if and only if their
classifying stacks are equivalent.
\end{lemma}

A convenient model for the bicategory $\Liegpd\bigl[\ca{M}\n\bigr]$ is
the bicategory of Lie groupoids where the $1$-morphisms are bibundles
and the $2$-morphisms are isomorphisms of bibundles as
in~\cite{blohmann;stacky-lie-groups}
or~\cite{lerman;orbifolds-as-stacks}.

The classifying functor has the following useful property.

\begin{lemma}\label{lemma;preservelimits}
The functor $\B$ preserves weak pullbacks.
\end{lemma}

\begin{proof}
This is shown in \cite[\S\,9]{noohi;foundations-topological-stacks}
for topological stacks, but the result carries over directly to
differentiable stacks.  See also
\cite[Theorem~12.6]{noohi;foundations-topological-stacks} for a
comparison of their quotient stack functor to our classifying stack
functor.
\end{proof}


\section{Foliation groupoids and \'etale stacks}
\label{section;foliation-etale}

In this section we review basic facts about foliation groupoids and
their classifying stacks, and draw some elementary consequences.  The
sources include Crainic and
Moerdijk~\cite{crainic-moerdijk;foliation-cyclic},
Hepworth~\cite{hepworth;vector-flow-stack}, and Berwick-Evans and
Lerman~\cite{berwick-evans-lerman;lie-2-algebras-vector}.  Our notion
of basic vector fields is adapted from Lerman and
Malkin~\cite{lerman-malkin;deligne-mumford}.  We show that the notions
of basic differential forms and basic vector fields are Morita
invariant (Proposition~\ref{proposition;vector-form-morita}) and that
the Lie $2$-algebra of vector fields of a stack is equivalent to a Lie
algebra if the stack is \'etale
(Proposition~\ref{proposition;vector-stack}).

\subsection{Foliation groupoids}\label{subsection;foliation}

A \emph{foliation groupoid} is a Lie groupoid
$X_\bu=(X_1\rightrightarrows X_0)$ whose object manifold $X_0$ is
Hausdorff and whose isotropy groups $\Iso(x)$ are discrete for all
$x\in X_0$.  An \emph{\'etale groupoid} is a Lie groupoid $X_\bu$
whose object manifold $X_0$ is Hausdorff and whose source map $s$ is
\'etale (i.e.\ a local diffeomorphism).

Clearly, \'etale groupoids are foliation groupoids.  If $X_\bu$ is a
source-connected foliation groupoid, then the orbits form a (constant
rank) foliation $\F=\F_{X_\bu}$ of $X_0$, the anchor map
$\rho\colon\Alg(X_\bu)\to TX_0$ is injective, and the image of $\rho$
is the tangent bundle $T\F$ of the foliation. There is no loss of
generality in the assumption that $X_0$ is Hausdorff; see
Lemma~\ref{lemma;hausdorff}.  If $X_\bu$ is a foliation groupoid,
every Lie groupoid with Hausdorff object manifold which is Morita
equivalent to $X_\bu$ is also a foliation groupoid.
\glossary{TF@$T\F$, tangent bundle of foliation $\F$}%

Let $X$ be a Hausdorff manifold equipped with a (regular) foliation
$\F$ and let $T\F$ be the tangent bundle of the foliation.  A Lie
groupoid $X_\bu$ over $X_0=X$ with the property that the anchor map
$\rho\colon\Alg(X_\bu)\to TX$ is injective and has image equal to
$T\F$ is said to \emph{integrate} the foliation $\F$.

The integrations of a foliation $\F$ form a category, the objects of
which are pairs $(X_\bu,\psi)$, where $X_\bu$ is a Lie groupoid
integrating $\F$ and $\psi\colon\Alg(X_\bu)\to T\F$ is an isomorphism
of Lie algebroids, and the arrows
$\phi_\bu\colon(X_\bu,\psi)\to(X_\bu',\psi')$ of which are morphisms
$\phi_\bu\colon X_\bu\to X_\bu'$ that respect the maps $\psi$ and
$\psi'$.

Given a foliated Hausdorff manifold $(X,\F)$, there are two important
source-connected foliation groupoids called the \emph{monodromy
  groupoid} $\Mon(X,\F)$, and the \emph{holonomy groupoid}
$\Hol(X,\F)$, both of which integrate $\F$.  (See
e.g.~\cite[\S\,5.2]{moerdijk-mrcun;foliations-groupoids}.)  There is a
Lie groupoid morphism $\hol\colon\Mon(X,\F)\to\Hol(X,\F)$ which is the
identity map on the object manifold $X$ and sends an arrow in
$\Mon(X,\F)$ to its holonomy action.  The following theorem says that
the category of source-connected integrations of $(X,\F)$ is a
preorder with the monodromy groupoid as a greatest element and the
holonomy groupoid as a least element.
\glossary{Mon@$\Mon(X,\F)$, monodromy groupoid}%
\glossary{Hol@$\Hol(X,\F)$, holonomy groupoid}%

\begin{theorem}%
[Crainic and
  Moerdijk~{\cite[Proposition~1]{crainic-moerdijk;foliation-cyclic}}]
\label{theorem;monhol}
Let $(X,\F)$ be a foliated Hausdorff manifold.  For every
source-connected Lie groupoid $X_\bu=(X_1\rightrightarrows X_0)$ over
$X_0=X$ integrating $\F$, there is a natural factorization of the
holonomy morphism $\Mon(X,\F)\to\Hol(X,\F)$ into morphisms of Lie
groupoids over $X$,
\[
\begin{tikzcd}
\Mon(X,\F)\ar[r,"\psi_{X_\bu}"]&X_\bu\ar[r,"\hol_{X_\bu}"]&\Hol(X,\F).
\end{tikzcd}
\]
The maps $\psi_{X_\bu}$ and $\hol_{X_\bu}$ are \'etale and surjective
on the manifolds of arrows, and $X_\bu$ is source-simply connected if
and only if $\psi_{X_\bu}$ is an isomorphism.
\end{theorem}

\begin{definition}
\label{definition;Liegroupbundle}
A \emph{Lie group bundle} is a Lie groupoid where every arrow $f$ has
$s(f)=t(f)$.  Let $X_\bu$ and $X_\bu'$ be Lie groupoids with the same
object manifold $X_0=X_0'$, and let $\psi_\bu\colon X_\bu\to X_\bu'$
be a Lie groupoid morphism which is the identity on $X_0$.  The
\emph{kernel} of $\psi_\bu$ is
\[
\ker(\psi_\bu)=\{\,f\in X_1\mid\psi_1(f)=u(s(f))=u(t(f))\,\}.
\]
\end{definition}

If $\psi_1$ is transverse to the identity bisection of $X_\bu'$, then
$\ker(\psi_\bu)$ is a Lie group bundle over $X_0$.  For instance, the
kernels of $\psi_{X_\bu}$ and $\hol_{X_\bu}$ in
Theorem~\ref{theorem;monhol} are Lie group bundles.

\begin{definition}
\label{definition;completetransversal} 
Let $(X,\F)$ be a foliated Hausdorff manifold.  A smooth map
$\phi\colon Y\to X$ is \emph{transverse to $\F$} if it is transverse
to each leaf of $\F$.  A transverse map $\phi\colon Y\to X$ is
\emph{complete} if $\phi(Y)$ intersects each leaf of $\F$ at least
once.  
\end{definition}

This extends the usual notion of a complete transversal, where $\phi$
is an injective immersion and $\dim Y=\codim\F$.

\begin{lemma}\label{lemma;transversal}
Let $X_\bu$ be a foliation groupoid integrating a foliation $\F$ on
$X_0$ and let $\phi_0\colon Y_0\to X_0$ be transverse to $\F$.  Then
the pullback groupoid $Y_\bu=\phi_0^*(X_\bu)$ is a foliation groupoid
which integrates the foliation $\phi_0^*\F$, and the induced morphism
$\phi_\bu\colon Y_\bu\to X_\bu$ is fully faithful.  If $\phi_0$ is
complete, then $\phi_\bu$ is a Morita morphism.
\end{lemma}

\begin{proof}[Sketch of proof]
This is well-known when $\phi_0$ is a transversal in the usual sense,
in which case $\phi_0^*\F$ is zero-dimensional and $Y_\bu$ is \'etale;
see e.g.\ Crainic and
Moerdijk~\cite[Lemma~2]{crainic-moerdijk;foliation-cyclic}.  The
general case is proved in a similar way: since $\phi_0$ is transverse
to $\F$, $(\phi_0,\phi_0)$ is transverse to $(s,t)$, so $Y_\bu$ is a
Lie groupoid, $\phi_1\colon Y_1\to X_1$ given by $\phi_1(x,f,y)=f$ is
smooth, $\phi_\bu$ is fully faithful and, if $\phi_0$ is complete,
essentially surjective.
\end{proof}

\subsection{The Bott connection}\label{subsection;bott}

Let $(X,\F)$ be a foliated Hausdorff manifold and let $N\F=TX/T\F\to
X$ be the normal bundle of the foliation.  The vector fields tangent
to $\F$ form a Lie subalgebra $\Gamma(T\F)$ of $\Vect(X)$.  Since $X$
is Hausdorff, we have $\Gamma(N\F)\cong\Vect(X)/\Gamma(T\F)$, so $N\F$
is a $\Gamma(T\F)$-module.  Let us write $\nabla_wv$ for the action of
a section $w\in\Gamma(T\F)$ on a section $v\in\Gamma(N\F)$.  The
operation
\begin{equation}\label{equation;bott-connection}
\nabla\colon\Gamma(T\F)\times\Gamma(N\F)\longto\Gamma(N\F)
\glossary{NF@$N\F$, normal bundle of foliation $\F$}%
\glossary{*@$\nabla$, Bott connection of foliation}%
\end{equation}
is the \emph{Bott connection} or \emph{partial connection} of $\F$.  A
section $v\in\Gamma(N\F)$ is \emph{$\Gamma(T\F)$-invariant} if
$\nabla_wv=0$ for all $w\in\Gamma(T\F)$.  We denote by
\begin{equation}\label{equation;lie-foliation}
\Vect_0(X,\F)=\Gamma(N\F)^{\Gamma(T\F)}
\glossary{Vect0@$\Vect_0(X,\F)$, transverse vector fields of foliation
  $\F$}%
\end{equation}
the space of all $\Gamma(T\F)$-invariant sections of $N\F$.  Let
$\lie{N}$ be the normalizer of the Lie subalgebra $\Gamma(T\F)$ of
$\Vect(X)$.  Elements of $\lie{N}$ are vector fields $v$ satisfying
$[v,w]\in\Gamma(T\F)$ for all $w\in\Gamma(T\F)$, in other words, whose
flow maps integral manifolds of $\F$ to integral manifolds of $\F$.
The natural map $\lie{N}\to\Vect(X)\to\Gamma(N\F)$ has image equal to
$\Vect_0(X,\F)$ and kernel equal to $\Gamma(T\F)$, so
$\Vect_0(X,\F)\cong\lie{N}/\Gamma(T\F)$ is naturally a Lie algebra.
If the space of leaves $X/\F$ has a manifold structure making the
quotient map $X\to X/\F$ a submersion, then
$\Vect_0(X,\F)\cong\Vect(X/\F)$.

\begin{remark}
If $X$ is a non-Hausdorff manifold equipped with a foliation $\F$, we
may not have $\Gamma(N\F)\cong\Vect(X)/\Gamma(T\F)$, but we can still
define a Bott connection in the following manner.  Let
$\ca{C}^\infty(E)$ denote the sheaf of smooth sections of a vector
bundle $E$ over $X$.  Then $\ca{C}^\infty(N\F)$ is the sheaf
associated to the presheaf $U\mapsto\Vect(U)/\Gamma(T\F|_U)$, and
$\Gamma(N\F)$ is the space of global sections of $\ca{C}^\infty(N\F)$.
For each open $U\subseteq X$ we have an operation
\[
\nabla_U\colon\Gamma(T\F|_U)\times\Vect(U)/\Gamma(T\F|_U)\longto
\Vect(U)/\Gamma(T\F|_U),
\]
which is a morphism of presheaves.  This presheaf morphism extends to
a morphism of sheaves
$\nabla\colon\ca{C}^\infty(T\F)\times\ca{C}^\infty(N\F)\to
\ca{C}^\infty(N\F)$.  On global sections this gives the Bott
connection~\eqref{equation;bott-connection}.  We define the space
$\Vect_0(X,\F)$ as in~\eqref{equation;lie-foliation}; it carries a
natural Lie bracket just as in the Hausdorff case.
\end{remark}

\begin{definition}\label{definition;foliate}
Let $(X,\F)$ and $(X',\F')$ be foliated manifolds.  A smooth map
$\phi\colon X\to X'$ is \emph{foliate} if the image of each leaf
$\F(x)$ of $\F$ is contained in the leaf $\F'(\phi(x))$ of $\F'$.
\end{definition}

The tangent map $T\phi\colon TX\to TX'$ of a foliate map
$\phi\colon(X,\F)\to(X',\F')$ descends to a map $\phi_*\colon N\F\to
N\F'$.  We say that sections $v\in\Gamma(N\F)$ and $v'\in\Gamma(N\F')$
are \emph{$\phi$-related}, and we write
\begin{equation}\label{equation;foliate-related}
v\sim_\phi v',
\glossary{*@$\sim_\phi$, $\phi$-relatedness of vector fields or
  sections}%
\end{equation}
if $\phi_*(v(x))=v'(\phi(x))$ for all $x\in X$.  The Bott connection
has the following naturality property with respect to foliate maps:
\begin{equation}\label{equation;bott-natural}
\nabla_wv\sim_\phi\nabla_{w'}v'
\end{equation}
for all $v\in\Gamma(N\F)$, $w\in\Gamma(T\F)$, $v'\in\Gamma(N\F')$, and
$w'\in\Gamma(T\F')$ satisfying $v\sim_\phi v'$ and $w\sim_\phi w'$.

\subsection{Basic vector fields and forms}\label{subsection;basic}

Let $X_\bu=(X_1\rightrightarrows X_0)$ be a foliation groupoid
integrating a foliation $\F_0=\F_0(X_\bu)$ of $X_0$.  Let
$N_0=N_0(X_\bu)=TX_0/T\F_0$ be the normal bundle of the foliation.
The leaves of $\F_0$ are the connected components of the orbits
$X_\bu\cdot x$, and for all $f\in X_1$ with $s(f)=x$ we have
\[T_x\F_0\cong\ker(T_fs)/(\ker(T_fs)\cap\ker(T_ft)).\]
There is a left action of the product groupoid $X_\bu\times X_\bu$ on
the arrow manifold~$X_1$
\[(X_1\times X_1)\times_{(s,s),X_0\times X_0,(t,s)}X_1\to X_1\]
given by the formula $(g,h)\cdot f=g\circ f\circ h^{-1}$.  The tangent
space to the orbit $(X_\bu\times X_\bu)\cdot f$ is
$\ker(T_fs)+\ker(T_ft)$, which is of constant dimension.  So the
components of the $X_\bu\times X_\bu$-orbits define a foliation
$\F_1=\F_1(X_\bu)$ of $X_1$ with normal bundle equal to
\[N_1=N_1(X_\bu)=TX_1/T\F_1=TX_1/(\ker(Ts)+\ker(Tt)).\]
The source map $s\colon X_1\to X_0$ is a foliate map and therefore
descends to a vector bundle map
\begin{equation}\label{equation;source-star}
s_*\colon N_1\longto N_0.
\end{equation}

\begin{lemma} \label{lemma;isomorphismpullbackbundles} 
The map~\eqref{equation;source-star} induces an isomorphism $N_1\cong
s^*N_0$.  Therefore we have a well-defined pullback map
\[
s^*\colon\Gamma(N_0)\longto\Gamma(N_1),
\]
which restricts to a Lie algebra homomorphism
$s^*\colon\Vect_0(X_0,\F_0)\to\Vect_0(X_1,\F_1)$.  Similarly, the
target map induces an isomorphism $N_1\cong t^*N_0$ and pullback maps
\[
t^*\colon\Gamma(N_0)\longto\Gamma(N_1),\qquad
t^*\colon\Vect_0(X_0,\F_0)\longto\Vect_0(X_1,\F_1).
\] 
\end{lemma}

\begin{proof}

For $f\in X_1$ and $x=s(f)$ we have
\begin{align*}
(N_1)_f&=T_fX_1/(\ker(T_fs)+\ker(T_ft))\\
&\cong\bigl(T_fX_1/\ker(T_fs)\bigr)\big/
  \bigl((\ker(T_fs)+\ker(T_ft))/\ker(T_fs)\bigr)\\
&\cong\bigl(T_fX_1/\ker(T_fs)\bigr)\big/
  \bigl(\ker(T_fs)/(\ker(T_fs)\cap\ker(T_ft))\bigr)\\
&\cong T_xX_0/T_x\F_0\\
&=(N_0)_x,
\end{align*}
which proves the first statement.  The existence of the pullback map
follows immediately from this, and the naturality property
\eqref{equation;bott-natural} of the Bott connection implies that the
induced map $\Vect_0(X_0,\F_0)\to\Vect_0(X_1,\F_1)$ preserves the Lie
bracket.  The proof of the last statement is analogous.
\end{proof}

Because of Lemma~\ref{lemma;isomorphismpullbackbundles} the following
definition makes sense.

\begin{definition}
\label{definition;basicvectorfield}
A \emph{basic vector field} on a foliation groupoid $X_\bu$ is an
element of
\begin{align*}
\Vect_\bas(X_\bu)&=\{\,v_\bu=(v_0,v_1)\in\Gamma(N_0)\times\Gamma(N_1)\mid
s^*v_0=v_1=t^* v_0\,\}\\
&\cong\{\,v\in\Gamma(N_0)\mid s^*v=t^* v\,\}.
\glossary{Vectbas@$\Vect_\bas(X_\bu)$, basic vector fields on Lie
  groupoid $X_\bu$}%
\end{align*}
\end{definition}

Thus a basic vector field is not a vector field, but a pair of
equivalence classes of vector fields.  Although a basic vector field
$v_\bu=(v_0,v_1)$ is determined by its first component $v_0$, we
prefer, mainly for notational consistency, to think of it as a pair of
sections.

\begin{definition}
\label{definition;basicform} 
A \emph{basic differential form} on a foliation groupoid $X_\bu$ is a
pair of differential forms
$\zeta_\bu=(\zeta_0,\zeta_1)\in\Omega^\bu(X_0)\times\Omega^\bu(X_1)$
satisfying $s^*\zeta_0=\zeta_1=t^*\zeta_0$.  The set of basic
differential forms on $X_\bu$ is the \emph{basic de Rham complex}
\begin{align*}
\Omega_\bas^\bu(X_\bu)&=
\{\,\zeta_\bu=(\zeta_0,\zeta_1)\in\Omega^\bu(X_0)\times\Omega^\bu(X_1)\mid
s^*\zeta_0 = \zeta_1 =t^*\zeta_0\,\}\\
&\cong\{\,\zeta\in\Omega^\bu(X_0)\mid s^*\zeta=t^*\zeta\,\}.
\glossary{Oomegabas@$\Omega_\bas(X_\bu)$, basic forms on $X_\bu$}%
\end{align*}
\end{definition}

Again, although a basic form $\zeta_\bu=(\zeta_0,\zeta_1)$ is
determined by its first component $\zeta_0$, we prefer to think of it
as a pair of forms.  Clearly $\Omega_\bas^\bu(X_\bu)$ is a subcomplex
of $\Omega^\bu(X_0)\times\Omega^\bu(X_1)$.  On the basic de Rham
complex of a foliation groupoid we have the contraction operators and
Lie derivatives
\begin{equation}\label{equation;contraction-lie}
\begin{aligned}
\iota&\colon\Vect_\bas(X_\bu)\times\Omega_\bas^\bu(X_\bu)\longto
\Omega_\bas^{\bu-1}(X_\bu),\\
\ca{L}&\colon\Vect_\bas(X_\bu)\times\Omega_\bas^\bu(X_\bu)\longto
\Omega_\bas^\bu(X_\bu)
\end{aligned}
\end{equation}
defined by
\begin{equation}\label{equation;cartan}
\iota_{v_\bu}\zeta_\bu=
(\iota_{\tilde{v}_0}\zeta_0,\iota_{\tilde{v}_1}\zeta_1),\qquad
\ca{L}_{v_\bu}\zeta_\bu=
(\ca{L}_{\tilde{v}_0}\zeta_0,\ca{L}_{\tilde{v}_1}\zeta_1)
\end{equation}
for basic forms $\zeta_\bu=(\zeta_0,\zeta_1)$ and basic vector fields
$v_\bu=(v_0,v_1)$, where
$(\tilde{v}_0,\tilde{v}_1)\in\Vect(X_0)\times\Vect(X_1)$ is a
representative of $v$.  The contraction operators $\iota_{v_\bu}$ are
well-defined, i.e.\ independent of the choice of representatives
$(\tilde{v}_0,\tilde{v}_1)$.  To see this, we must show that for all
$\zeta\in\Omega^\bu(X_0)$ satisfying $s^*\zeta=t^*\zeta$ and for all
tangent vectors $v\in T_x\F$ we have $\iota_v\zeta=0$.  There exist
$f\in s^{-1}(x)$ and $w\in\ker(T_fs)$ such that $v=T_ft(w)$, so
\[
t^*\iota_v\zeta=\iota_wt^*\zeta=\iota_ws^*\zeta=
s^*\iota_{T_fs(w)}\zeta=0,
\]
and therefore $\iota_v\zeta=0$.  Similarly, the operators
$\ca{L}_{v_\bu}$ are well-defined.

\begin{definition}
\label{definition;horizontalforms} 
Let $X_\bu$ be a foliation groupoid.  A differential form
$\zeta\in\Omega^k(X_0)$ is \emph{horizontal} if
$\iota_{\rho(\sigma)}\zeta=0$ for all sections $\sigma$ of the Lie
algebroid $\Alg(X_\bu)$, and \emph{infinitesimally invariant} if
$\ca{L}_{\rho(\sigma)}\zeta=0$ for all sections $\sigma$ of
$\Alg(X_\bu)$.  We denote by $\Omega_0^k(X_0,\F_0)$ the set of all
$k$-forms on $X_0$ which are horizontal and infinitesimally invariant.
\glossary{Oomega0@$\Omega_0(X_\bu)$, infinitesimally basic forms on
  $X_\bu$}%
\end{definition}

The notions in Definition~\ref{definition;horizontalforms} depend only
on the Lie algebroid of $X_\bu$, i.e.\ on the foliation $\F_0$.  The
notions of horizontal, basic, and invariant forms are well-known in
the context of Lie group actions, for which the second part of the
next result is standard.

\begin{lemma}\label{lemma;bas-inf}
Let $X_\bu$ be a foliation groupoid.
\begin{enumerate}
\item\label{item;bas-vector}
Under the identification
$\Vect_\bas(X_\bu)\cong\{\,v\in\Gamma(N_0)\mid s^*v=t^* v\,\}$, the
set of basic vector fields $\Vect_\bas(X_\bu)$ is a Lie subalgebra of
the Lie algebra $\Vect_0(X_0,\F_0)$ defined
in~\eqref{equation;lie-foliation}.
\item\label{item;bashorinf}
Under the identification
$\Omega_\bas^\bu(X_\bu)\cong\{\,\zeta\in\Omega^\bu(X_0)\mid
s^*\zeta=t^*\zeta\,\}$, the basic complex $\Omega^\bu_\bas(X_\bu)$ is
a subcomplex of the complex $\Omega^\bu_0(X_0,\F_0)$ of
Definition~\ref{definition;horizontalforms}.
\end{enumerate}
Both inclusions are equalities if $X_\bu$ is source-connected.
\end{lemma}

\begin{proof}
\eqref{item;bas-vector}~Let $v_\bu=(v_0,v_1)\in\Vect_\bas(X_\bu)$.
Let $\sigma$ be a section of the Lie algebroid $\Alg(X_\bu)$, which
has anchor map $\rho=Tt|_{\Alg(X_\bu)}$.  Let $\sigma_R$ denote the
right-invariant vector field on $X_1$ induced by $\sigma$.  Then
$\sigma_R$ is $s$-related to the zero vector field on $X_0$ because
$\sigma_R$ is tangent to the source fibres, and $\sigma_R$ is
$t$-related to the vector field $\rho(\sigma)\in\Gamma(T\F)$ on $X_0$.
In other words,
\begin{equation}\label{equation;related}
\sigma_R\sim_s0,\qquad\sigma_R\sim_t\rho(\sigma).
\end{equation}
Since $v_1=s^*v_0=t^*v_0$, the naturality of the Bott
connection~\eqref{equation;bott-natural} yields
\[
t^*\nabla_{\rho(\sigma)}v_0=\nabla_{\sigma_R}t^*v_0=
\nabla_{\sigma_R}v_1=\nabla_{\sigma_R}s^*v_0=0,
\]
and hence $\nabla_{\rho(\sigma)}v_0=0$ for all $\sigma$, which is to
say that $v_0$ is $\Gamma(T\F)$-invariant.  This shows that
$\Vect_\bas(X_\bu)$ is contained in $\Vect_0(X_0,\F_0)$.  That
$\Vect_\bas(X_\bu)$ is a Lie subalgebra follows from the fact that
$s^*$ and $t^*$ are Lie algebra homomorphisms.  Conversely, suppose
$\nabla_{\rho(\sigma)}v_0=0$ for all $\sigma\in\Gamma(\Alg(X_\bu))$
and that $X_\bu$ is source-connected.  From~\eqref{equation;related}
we obtain $\nabla_{\sigma_R}s^*v_0=0$ and
$\nabla_{\sigma_R}t^*v_0=t^*\nabla_{\rho(\sigma)}v_0=0$ for all
$\sigma$.  For $x\in X_0$ we have $(s^*v_0)_{u(x)}=(t^*v_0)_{u(x)}$.
The vector fields $\sigma_R$ span the subbundle $\ker(Ts)$ of $TX_1$.
The source fibre $s^{-1}(x)$ is Hausdorff (by
Lemma~\ref{lemma;hausdorff}, since $X_0$ is Hausdorff by definition)
and connected, so we have $(s^*v_0)_f=(t^*v_0)_f$ for every $f\in
s^{-1}(x)$.  This shows that $\Vect_0(X_0,\F_0)$ is contained in
$\Vect_\bas(X_\bu)$.

\eqref{item;bashorinf}~Let $\zeta\in\Omega^k(X_0)$ be a differential
form.  It follows from~\eqref{equation;related} that
\begin{equation}\label{equation;lie}
\ca{L}_{\sigma_R}s^*\zeta=0,\qquad\ca{L}_{\sigma_R}
t^*\zeta=t^*\ca{L}_{\rho(\sigma)}\zeta
\end{equation}
for all sections $\sigma$ of the Lie algebroid $\Alg(X_\bu)$.  At the
identity bisection $ u(X_0)\subseteq X_1$ the tangent bundle of $X_1$
is a direct sum $TX_1|_{ u(X_0)}=\Alg(X_\bu)\oplus u^*TX_0$.  Let
$x\in X_0$.  For $\sigma_1$, $\sigma_2,\dots$,
$\sigma_k\in\Alg(X_\bu)_{ u(x)}$ and for $w_1$, $w_2,\dots$, $w_k\in
T_xX_0$ we have
\begin{equation}\label{equation;source-target}
\begin{split}
(s^*\zeta)_{ u(x)}&(\sigma_1+ u_*w_1,\sigma_2+ u_*w_2,\dots,\sigma_k+
  u_*w_k)\\
&=\zeta_x(s_*\sigma_1+w_1,s_*\sigma_2+w_2,\dots,s_*\sigma_k+w_k)\\
&=\zeta_x(w_1,w_2,\dots,w_k),\\
(t^*\zeta)_{u(x)}&(\sigma_1+ u_*w_1,\sigma_2+ u_*w_2,\dots,\sigma_k+
  u_*w_k)\\
&=\zeta_x(t_*\sigma_1+w_1,t_*\sigma_2+w_2,\dots,t_*\sigma_k+w_k)\\
&=\zeta_x\bigl(\rho(\sigma_1)+
  w_1,\rho(\sigma_2)+w_2,\dots,\rho(\sigma_k)+w_k\bigr).
\end{split}
\end{equation}
Now assume $s^*\zeta=t^*\zeta$.  Then $\ca{L}_{\rho(\sigma)}\zeta=0$
by~\eqref{equation;lie} and $\iota_{\rho(\sigma)}\zeta=0$
by~\eqref{equation;source-target}, so
$\zeta\in\Omega_0^\bu(X_0,\F_0)$.  Finally, suppose that
$\zeta\in\Omega^\bu_0(X_0,\F_0)$ and that $X_\bu$ is source-connected.
From horizontality and from~\eqref{equation;source-target} we obtain
that $(s^*\zeta)_{u(x)}=(t^*\zeta)_{u(x)}$ for every $x\in X_0$.  From
invariance and from~\eqref{equation;lie} we obtain that
$\ca{L}_{\sigma_R} s^*\zeta=\ca{L}_{\sigma_R}t^*\zeta=0$.  The vector
fields $\sigma_R$ span the subbundle $\ker(Ts)$ of $TX_1$.  The source
fibre $s^{-1}(x)$ is Hausdorff and connected, so we have
$(s^*\zeta)_f=(t^*\zeta)_f$ for every $f\in s^{-1}(x)$.
\end{proof}

The next result states that the notions of basic vector fields and
basic differential forms are Morita invariant.

\begin{proposition}\label{proposition;vector-form-morita}
Let $X_\bu$ and $Y_\bu$ be foliation groupoids.  A Morita morphism
$\phi_\bu\colon X_\bu\to Y_\bu$ induces
\begin{enumerate}
\item\label{item;vector-morita}
an isomorphism $\phi_\bu^*\colon\Vect_\bas(Y_\bu)\stackrel\cong\longto
\Vect_\bas(X_\bu)$ of Lie algebras, and
\item\label{item;form-morita}
an isomorphism
$\phi_\bu^*\colon\Omega_\bas^\bu(Y_\bu)\stackrel\cong\longto
\Omega_\bas^\bu(X_\bu)$ of complexes.
\end{enumerate}
If $\phi_\bu$ and $\chi_\bu$ are naturally isomorphic Morita
morphisms, then $\phi_\bu^*=\chi_\bu^*$.
\end{proposition}

\begin{proof}
We will prove~\eqref{item;vector-morita}; the proof
of~\eqref{item;form-morita} is similar.  Let $\F_0(X_\bu)$ be the
foliation of $X_0$ and $N_0(X_\bu)=TX_0/T\F_0(X_\bu)$ its normal
bundle.  Let $\F_1(X_\bu)$ be the foliation of $X_1$ and
$N_1(X_\bu)=TX_1/(\ker(Ts)+\ker(Tt))$ its normal bundle.  Similarly we
have foliations $\F_0(Y_\bu)$ of $Y_0$, $\F_1(Y_\bu)$ of $Y_1$, and
normal bundles $N_0(Y_\bu)$ over $Y_0$ and $N_1(Y_\bu)$ over $Y_1$.
The tangent maps $T\phi_0\colon TX_0\to TY_0$ and $T\phi_1\colon
TX_1\to TY_1$ descend to vector bundle maps
\[
(\phi_0)_*\colon N_0(X_\bu)\longto N_0(Y_\bu),\qquad(\phi_1)_*\colon
N_1(X_\bu)\longto N_1(Y_\bu).
\]
Let $\chi_\bu\colon X_\bu\to Y_\bu$ be another Morita morphism and
$\gamma\colon\phi\To\chi$ a $2$-morphism.  Then $s\circ\gamma=\phi_0$
and $t\circ\gamma=\chi_0$, and for all arrows $f\in X_1$ we have
\[\gamma(x')=\chi(f)\circ\gamma(x)\circ\phi(f)^{-1},\]
where $x=s(f)$ and $x'=t(f)$.  This shows that $\gamma$ maps the orbit
$X_\bu\cdot x$ to the orbit $(Y_\bu\times Y_\bu)\cdot\gamma(x)$, so
$\gamma\colon(X_0,\F_0(X_\bu))\to(Y_1,\F_1(Y_\bu))$ is a foliate map.
Therefore $\gamma$ induces a vector bundle map $\gamma_*\colon
N_0(X_\bu)\to N_1(Y_\bu)$, which makes the diagram
\begin{equation}\label{equation;normal-2-morphism}
\begin{tikzcd}[row sep=large]
N_0(X_\bu)_x\ar[r,"(\phi_0)_*"]\ar[d,"(\chi_0)_*"']\ar[dr,"\gamma_*"]&
N_0(Y_\bu)_{\phi_0(x)}\\
N_0(Y_\bu)_{\chi_0(x)}&N_1(Y_\bu)_{\gamma(x)}\ar[l,"t_*"',"\cong"]
\ar[u,"s_*"',"\cong"]
\end{tikzcd}
\end{equation}
commute for all $x\in X_0$.  We will use this to establish the
following four facts:
\begin{numerate}
\item\label{item;fibres}
$(\phi_0)_*\colon N_0(X_\bu)_x\to N_0(Y_\bu)_{\phi_0(x)}$ is an
  isomorphism for all $x\in X_0$;
\item\label{item;pullback}
the squares
\[
\begin{tikzcd}[row sep=large]
N_0(X_\bu)\ar[r,"(\phi_0)_*"]\ar[d]&N_0(Y_\bu)\ar[d]\\
X_0\ar[r,"\phi_0"]&Y_0
\end{tikzcd}
\qquad\qquad
\begin{tikzcd}[row sep=large]
N_1(X_\bu)\ar[r,"(\phi_1)_*"]\ar[d]&N_1(Y_\bu)\ar[d]\\
X_1\ar[r,"\phi_1"]&Y_1
\end{tikzcd}
\]
are cartesian, $\phi_\bu$ induces pullback maps
\begin{align*}
\phi^*_0\colon\Gamma(N_0(Y_\bu))\longto\Gamma(N_0(X_\bu)),\qquad
\phi^*_1\colon\Gamma(N_1(Y_\bu))\longto\Gamma(N_1(X_\bu)),
\end{align*}
and a Lie algebra homomorphism
$\phi_\bu^*\colon\Vect_\bas(Y_\bu)\longto\Vect_\bas(X_\bu)$;
\item\label{item;2-morphism}
$\phi_\bu^*=\chi_\bu^*$;
\item\label{item;isomorphism}
$\phi_\bu^*$ is an isomorphism.
\end{numerate}
We assert that fact~\eqref{item;fibres} implies
fact~\eqref{item;pullback}.  Indeed, it follows from
fact~\eqref{item;fibres} that the square on the left
in~\eqref{item;pullback} is cartesian.  The map $(\phi_1)_*$ is
obtained by lifting $(\phi_0)_*$ via $s_*$, and therefore the right
square is cartesian as well.  This means we have isomorphisms
$N_0(X_\bu)\cong\phi_0^*N_0(Y_\bu)$ and
$N_1(X_\bu)\cong\phi_1^*N_1(Y_\bu)$, and the existence of the pullback
maps $\phi_0^*$, $\phi_1^*$, $\phi_\bu^*$ is a formal consequence of
this.  To check that $\phi_\bu^*$ is a Lie algebra homomorphism, it is
enough to show that for all basic vector fields $u_\bu$,
$v_\bu\in\Vect_\bas(Y_\bu)$ we have
\begin{equation}\label{equation;phirelated}
[\phi_0^* u_0,\phi_0^*v_0]\sim_{\phi_0}[u_0,v_0]
\end{equation}
in the sense of~\eqref{equation;foliate-related}.  To prove this, we
may work in foliation charts and assume that the leaf spaces
$\bar{X}=X_0/\F_0(X_\bu)$ and $\bar{Y}=Y_0/\F_0(Y_\bu)$ are both
smooth manifolds, and that $\phi_0$ descends to a smooth map
$\bar{\phi}\colon\bar{X}\to\bar{Y}$.  By Lemma~\ref{lemma;bas-inf} we
have an inclusion
\[
\Vect_\bas(X_\bu)\subseteq\Vect_0(X_0,\F_0(X_\bu))\cong\Vect(\bar{X})
\]
and a similar inclusion for $Y_\bu$.  The
identity~\eqref{equation;phirelated} now follows from the naturality
of the Lie bracket on $\Vect(\bar{X})$ and $\Vect(\bar{Y})$ with
respect to $\bar{\phi}$.  Thus fact~\eqref{item;pullback} follows from
fact~\eqref{item;fibres}.  Fact~\eqref{item;2-morphism} is proved as
follows: if $v\in\Gamma(N_0(Y_\bu))$ satisfies $s^*v=t^*v$, then
\[\phi_0^*v=\gamma^*s^*v=\gamma^*t^*v=\chi_0^*v,\]
so $\phi_\bu^*=\chi_\bu^*$.  It remains to prove~\eqref{item;fibres}
and~\eqref{item;isomorphism}.  We consider three cases.

\emph{Case}~1.  Suppose $\phi_\bu$ has a weak inverse, i.e.\ a triple
$(\psi_\bu,\delta,\eps)$ consisting of a Morita morphism
$\psi_\bu\colon Y_\bu\to X_\bu$ and $2$-morphisms
$\delta\colon\psi_\bu\circ\phi_\bu\To\id_{X_\bu}$ and
$\eps\colon\phi_\bu\circ\psi_\bu\To\id_{Y_\bu}$.
Applying~\eqref{equation;normal-2-morphism} to the morphisms
$\psi_\bu\circ\phi_\bu$ and $\id_{X_\bu}$ shows that
\[(\psi_0)_*\circ(\phi_0)_*\colon(N_0)_x\to(N_0)_{\psi(\phi(x))}\]
is an isomorphism for all $x\in X_0$.  Similarly,
$(\phi_0)_*\circ(\psi_0)_*\colon(N_0)_y\to(N_0)_{\phi(\psi(y))}$ is an
isomorphism for all $y\in Y_0$.  Hence
$(\phi_0)_*\colon(N_0)_x\to(N_0)_{\phi(x)}$ is an isomorphism for all
$x$, which establishes~\eqref{item;fibres}.  It follows from
fact~\eqref{item;2-morphism} that $\psi_\bu^*\phi_\bu^*=\id$ and
$\phi_\bu^*\psi_\bu^*=\id$, so $\phi_\bu^*$ is an isomorphism, which
proves~\eqref{item;isomorphism}.

\emph{Case}~2.  Suppose $\phi_0\colon X_0\to Y_0$ is a surjective
submersion.  Consider an open subset $V$ of $Y_0$ over which there
exists a smooth right inverse $\psi_0\colon V\to X_0$ of $\phi_0$.
Let $Y_\bu|_V$ be the restriction of $Y_\bu$ to $V$, which has object
manifold $V$ and arrow manifold $Y_1|_V=s^{-1}(V)\cap t^{-1}(V)$.  Let
$U=\phi_0^{-1}(V)$ and let $X_\bu|_U$ be the restriction of $X_\bu$ to
$U$.  Then $\phi_\bu$ restricts to a Morita morphism $X_\bu|_U\to
Y_\bu|_V$, which we also denote by $\phi_\bu$.  By full faithfulness
of $\phi_\bu$, the map $\psi_0$ lifts uniquely to a morphism
$\psi_\bu=(\psi_0,\psi_1)\colon Y_\bu|_V\to X_\bu|_U$ satisfying
$\phi_\bu\circ\psi_\bu=\id_{Y_\bu|_V}$.  Full faithfulness also gives
us a map $\gamma$ completing the following commutative diagram:
\[
\begin{tikzcd}
U\ar[drr,bend left,"u\circ\phi_0"]\ar[ddr,bend
  right,"{(\psi_0\circ\phi_0,\id_U)}"']\ar[dr,dotted,"\gamma"]&&\\
&X_1|_U\ar[r,"\phi_1"]\ar[d]&Y_1|_V\ar[d]\\
&U\times U\ar[r,"\phi_0\times\phi_0"]&V\times V
\end{tikzcd}
\]
This means that for each $x\in U$, $\phi_1(\gamma(x))$ is the identity
arrow at $\phi_0(x)$.  Again by full faithfulness, this implies that
$\gamma(x)$ is an arrow from $\psi_0(\phi_0(x))$ to $x$ satisfying
$f\circ\gamma(x)=\gamma(x')\circ\psi_1(\phi_1(f))$ for all arrows
$f\in X_1$ with $s(f)=x\in U$, $t(f)=x'\in U$.  In other words,
$\gamma$ is a $2$-morphism $\psi\circ\phi\To\id_{X_\bu|_U}$.  By
case~1, $\phi_\bu$ induces isomorphisms $N_0(X_\bu)_x\cong
N_0(Y_\bu)_{\phi_0(x)}$ for all $x\in U$, and
\begin{equation}\label{equation;open-iso}
\begin{tikzcd}
\phi_\bu^*\colon\Vect_\bas(Y_\bu|_V)\ar[r,"\cong"]&\Vect_\bas(X_\bu|_U),
\end{tikzcd}
\end{equation}
which are independent of the section $\psi_0$.  Now choose a covering
$\lie{V}$ of $Y_0$ consisting of open sets $V$ over which $\phi_0$
admits a right inverse.  The sets $U=\phi^{-1}(V)$ cover $X_0$, so we
get $N_0(X_\bu)_x\cong N_0(Y_\bu)_{\phi_0(x)}$ for all $x\in X_0$,
which proves fact~\eqref{item;fibres}.  Since $\phi_0$ is surjective,
the pullback map
$\phi_0^*\colon\Gamma(N_0(Y_\bu))\to\Gamma(N_0(X_\bu))$ is injective,
so in particular
$\phi_\bu^*\colon\Vect_\bas(Y_\bu)\to\Vect_\bas(X_\bu)$ is injective.
To show surjectivity of $\phi_\bu^*$, let $v\in\Gamma(N_0(X_\bu))$ be
a section satisfying $s^*v=t^*v$.  By~\eqref{equation;open-iso} there
is for each $V\in\lie{V}$ a unique section $w_V$ of $N_0(X_\bu)$ over
$V$ satisfying $s_V^*w_V=t_V^*w_V$ and $\phi_0^*w_V=v|_U$.  By
injectivity of $\phi_0^*$ we have $w_V=w_{V'}$ on $V\cap V'$ for all
$V$, $V'\in\lie{V}$.  Therefore the $w_V$ glue together to a global
section $w$ of $N_0(Y_\bu)$ with $\phi_0^*w=v$.  Since
\[
\phi_0^*s^*w=s^*\phi_0^*w=s^*v=t^*v=t^*\phi_0^*w=\phi_0^*t^*w,
\]
we have $s^*w=t^*w$, so $\phi_\bu^*$ is surjective, which establishes
fact~\eqref{item;isomorphism}.

\emph{Case}~3.  Given a general Morita morphism $\phi_\bu\colon
X_\bu\to Y_\bu$, essential surjectivity gives us a surjective
submersion
\[\tau_0=t\circ\pr_1\colon Z_0=Y_1\times_{Y_0}X_0\longto Y_0.\]
Let $Z_\bu=\tau_0^*Y_\bu$; then the canonical morphism $\tau_\bu\colon
Z_\bu\to Y_\bu$ is Morita.  The projection onto the second factor
$\pi_0=\pr_2\colon Z_0\to X_0$ is also a surjective submersion, and
lifts to a Morita morphism $\pi_\bu\colon Z_\bu\to X_\bu$.  We do
\emph{not} have $\phi_\bu\circ\pi_\bu=\tau_\bu$, but the map
$\gamma=\pr_1\colon Z_0\to Y_1$ defines a $2$-morphism
$\gamma\colon\phi_\bu\circ\pi_\bu\To\tau_\bu$.  By case~2, the
morphisms $\pi_\bu$ and $\tau_\bu$ induce isomorphisms
\[
(\pi_0)_*\colon N_0(Z_\bu)_z\longto N_0(X_\bu)_{\pi_0(z)}\quad\text{
  and }\quad(\tau_0)_\bu\colon N_0(Z_\bu)_z\longto
N_0(Y_\bu)_{\tau_0(z)}
\]
for all $z\in Z_0$.  Applying~\eqref{equation;normal-2-morphism} to
the morphisms $\phi_\bu\circ\pi_\bu$ and $\tau_\bu$ shows that
$(\phi_0)_*\circ(\pi_0)_*\colon N_0(Z_\bu)_z\to
N_0(Y_\bu)_{\phi_0(\pi_0(z))}$ is an isomorphism for all $z\in Z_0$.
Hence $(\phi_0)_*\colon N_0(X_\bu)_x\to N_0(Y_\bu)_{\phi_0(x)}$ is an
isomorphism for all $x\in X_0$, which proves~\eqref{item;fibres}.
Fact~\eqref{item;2-morphism} shows that
$\pi_\bu^*\circ\phi_\bu^*=\tau_\bu^*$.  The maps $\pi_\bu^*$ and
$\tau_\bu^*$ are isomorphisms by case~2, so $\phi_\bu^*$ is an
isomorphism as well, which proves~\eqref{item;isomorphism}.
\end{proof}

\begin{remark}
If we extend the definitions of $\Omega_\bas^\bu(X_\bu)$ and
$\Omega_0^\bu(X_\bu)$ verbatim to arbitrary Lie groupoids $X_\bu$,
Lemma~\ref{lemma;bas-inf}\eqref{item;bashorinf} and
Proposition~\ref{proposition;vector-form-morita}\eqref{item;form-morita}
still hold.  We omit the proof of this fact, as we have no use for it
in this paper.
\end{remark}

\subsection{Basic versus multiplicative vector fields}
\label{subsection;basic-multiplicative}

Let $X_\bu$ be a Lie groupoid.  Applying the tangent functor $T$ to
all the structure maps of $X_\bu$ gives a Lie groupoid $TX_\bu$ called
the \emph{tangent groupoid}, which is equipped with an obvious
morphism $p_\bu\colon TX_\bu\to X_\bu$.

Mackenzie and Xu~\cite{mackenzie-xu;lifting-multiplicative} defined a
\emph{multiplicative vector field} on $X_\bu$ to be a Lie groupoid
morphism $v_\bu\colon X_\bu\to TX_\bu$ satisfying $p_\bu\circ
v_\bu=\id_{X_\bu}$.  We will denote the set of multiplicative vector
fields by $\lie{M}(X_\bu)$.

By~\cite[Proposition~3.5]{mackenzie-xu;lifting-multiplicative} a
multiplicative vector field is the same as a pair of vector fields
$v_0\in\Vect(X_0)$, $v_1\in\Vect(X_1)$ whose flows form a (local)
one-parameter group of Lie groupoid automorphisms of $X_\bu$.  It
follows that $\lie{M}(X_\bu)$ is a Lie subalgebra of
$\Vect(X_0)\times\Vect(X_1)$.

By~\cite[Example~3.4]{mackenzie-xu;lifting-multiplicative}, for each
section $\sigma$ of the Lie algebroid
$\lie{A}(X_\bu)=\Gamma(\Alg(X_\bu))$ the pair
$\partial(\sigma)=(\rho(\sigma),\sigma_L+\sigma_R)$ is a
multiplicative vector field.  This defines a Lie algebra homomorphism
\[\partial\colon\lie{A}(X_\bu)\longto\lie{M}(X_\bu).\]
Differentiating the conjugation action of $X_\bu$ on $X_1$ gives a Lie
algebra action $\lie{M}(X_\bu)\to\Der(\lie{A}(X_\bu))$, which makes
the pair of Lie algebras $\lie{M}(X_\bu)$, $\lie{A}(X_\bu)$ a crossed
module of Lie algebras.  The associated (strict) Lie $2$-algebra
\[
\begin{tikzcd}
\Vect_\mult(X_\bu)=\bigl(\lie{A}(X_\bu)\rtimes\lie{M}(X_\bu) \ar[r,
  yshift=0.5ex]\ar[r, yshift=-0.5ex]&\lie{M}(X_\bu)\bigr)
\end{tikzcd}
\]
is the \emph{Lie $2$-algebra of multiplicative vector fields} of
$X_\bu$; see Berwick-Evans and
Lerman~\cite[\S\,2]{berwick-evans-lerman;lie-2-algebras-vector}.  We
assert that for a foliation groupoid this Lie $2$-algebra is
equivalent to the Lie algebra of basic vector fields
(Definition~\ref{definition;basicvectorfield}).
\glossary{Vectmult@$\Vect_\mult(X_\bu)$, multiplicative vector fields
  on Lie groupoid $X_\bu$}%

\begin{proposition}\label{proposition;multiplicative-basic}
Let $X_\bu$ be a foliation groupoid.  The natural Lie algebra
homomorphism $\nu=\nu_{X_\bu}\colon\lie{M}(X_\bu)\to\Vect_\bas(X_\bu)$
induces an isomorphism
\[\lie{M}(X_\bu)/\partial\lie{A}(X_\bu)\cong\Vect_\bas(X_\bu).\]
Thus the Lie $2$-algebra of multiplicative vector fields
$\Vect_\mult(X_\bu)$ is equivalent to the Lie algebra of basic vector
fields $\Vect_\bas(X_\bu)$.
\end{proposition}

\begin{proof}
Let $\F_0$, resp.\ $\F_1$, be the foliation of $X_0$, resp.\ $X_1$,
induced by $X_\bu$, and let $v_\bu=(v_0,v_1)$ be a multiplicative
vector field.  Let
$\bar{v}_\bu=(\bar{v}_0,\bar{v}_1)\in\Gamma(N\F_0)\times\Gamma(N\F_1)$
be the pair of sections of the normal bundles determined by $v_\bu$.
The flow of $v_\bu$ acts by groupoid automorphisms, so the flow of
$v_0$ is foliate (Definition~\ref{definition;foliate}).  Hence $v_0$
is in the normalizer of $\Gamma(T\F_0)$ in $\Vect(X_0)$, that is to
say $\bar{v}_\bu$ is a basic vector field.  This defines the natural
homomorphism $\nu\colon\lie{M}(X_\bu)\to\Vect_\bas(X_\bu)$.  The
anchor $\rho\colon\Alg(X_\bu)\to TX_0$ is injective and its image is
$T\F_0$.  In particular $\partial$ is injective, so the Lie
$2$-algebra $\Vect_\mult(X_\bu)$ is equivalent to the quotient Lie
algebra $\lie{M}(X_\bu)/\partial\lie{A}(X_\bu)$.  Moreover $\nu$
descends to a homomorphism
\[
\bar\nu\colon\lie{M}(X_\bu)/\partial\lie{A}(X_\bu)\longto
\Vect_\bas(X_\bu).
\]  
We must show that $\bar\nu$ is an isomorphism.  First suppose $X_\bu$
is \'etale.  Then $\Alg(X_\bu)=0$, the foliations $\F_0$ and $\F_1$
are discrete, and $\nu$ is injective.  If $v_0\in\Vect(X_0)$ satisfies
$s^*v_0=t^*v_0$, then the vector field $(v_1,v_1)$ on $X_1\times X_1$,
where $v_1=s^*v_0\in\Vect(X_1)$, is tangent to the submanifold
$X_2=X_1\times_{s,X_0,t}X_1$ of $X_1\times X_1$.  So $v_1$ restricts
to a vector field $v_2$ on $X_2$, and we have $v_2=m^*v_1$, where
$m\colon X_2\to X_1$ is the multiplication map.  In other words, every
basic vector field is multiplicative and we have equality
\[
\lie{M}(X_\bu)/\partial\lie{A}(X_\bu)=\lie{M}(X_\bu)=\Vect_\bas(X_\bu).
\]
For an arbitrary foliation groupoid $X_\bu$, pick a Morita morphism
$\phi_\bu\colon Y_\bu\to X_\bu$ from an \'etale groupoid $Y_\bu$.  By
Proposition~\ref{proposition;vector-form-morita}%
\eqref{item;vector-morita} $\phi_\bu$ induces an isomorphism
$\phi_\bu^*\colon\Vect_\bas(X_\bu)\to\Vect_\bas(Y_\bu)$.  We let
$Z=Y_0\times_{\phi_0,X_0,s}X_1$ and as
in~\cite[Theorem~4.4]{berwick-evans-lerman;lie-2-algebras-vector} we
form the linking groupoid $L_\bu$ with $L_0=Y_0\coprod X_0$ and
$L_1=Y_1\coprod Z\coprod Z^{-1}\coprod X_1$.  We have obvious open
embeddings $\iota_{X_\bu}\colon X_\bu\inj L_\bu$ and
$\iota_{Y_\bu}\colon Y_\bu\inj L_\bu$, both of which are essentially
surjective.  The triangle
\[\begin{tikzcd}
Y_\bu\arrow[dr,"\iota_{Y_\bu}"']\arrow[rr,"\phi_\bu"]&&X_\bu
\ar[dl,"\iota_{X_\bu}"]\\
&L_\bu&
\end{tikzcd}\]
is $2$-commutative: the map $\gamma\colon Y_0\to L_1$ given by
$\gamma(y)=(\id_{\phi(y)},y)\in Z\inj L_1$ is a $2$-morphism
$\gamma\colon\iota_{Y_\bu}\To\iota_{X_\bu}\circ\phi_\bu$.  It follows
from~\cite[Lemma~4.14]{berwick-evans-lerman;lie-2-algebras-vector}
that $\phi_\bu$ induces an isomorphism
\[
\phi_\bu^*=\iota_{Y_\bu}^*\circ(\iota_X^*)^{-1}\colon
\lie{M}(X_\bu)/\partial\lie{A}(X_\bu)\to\lie{M}(Y_\bu).
\]
The square
\[
\begin{tikzcd}[row sep=large]
\lie{M}(X_\bu)/\partial\lie{A}(X_\bu)\ar[r,"\bar\nu_{X_\bu}"]
\ar[d,"\phi_\bu^*"',"\cong"]&\Vect_\bas(X_\bu)
\ar[d,"\phi_\bu^*","\cong"']\\
\lie{M}(Y_\bu)\ar[r,"\bar\nu_{Y_\bu}","="']&\Vect_\bas(Y_\bu)
\end{tikzcd}
\]
commutes, so $\bar\nu_{X_\bu}$ is an isomorphism.
\end{proof}

\subsection{\texorpdfstring{$0$-Symplectic groupoids}{0-symplectic}}
\label{subsection;symplectic-groupoid}

Let $X_\bu$ be a foliation groupoid with foliation $(X_0,\F)$ and let
$\omega_\bu=(\omega_0,\omega_1)\in\Omega_\bas^2(X_\bu)$ be a basic
$2$-form.  We call $\omega_\bu$ \emph{nondegenerate} if
$\ker(\omega_0)=T\F$, or equivalently,
$\ker(\omega_1)=\ker(Ts)+\ker(Tt)$.  We say $\omega_\bu$ is
\emph{$0$-symplectic}, and the pair $(X_\bu,\omega_\bu)$ is a
\emph{$0$-symplectic Lie groupoid}, if $d\omega_\bu=0$ and
$\omega_\bu$ is nondegenerate.
\glossary{X*omega@$(X_\bu,\omega_\bu)$, $0$-symplectic groupoid}%
\glossary{omega*@$\omega_\bu$, $0$-symplectic form}%

We have adopted the term ``$0$-symplectic'' to avoid confusion with
Weinstein's notion of a symplectic
groupoid~\cite{weinstein;symplectic-groupoids-poisson-manifolds}.  The
nondegeneracy of $\omega_\bu$ can be restated as the vertical map
\[
\begin{tikzcd}
T\F\ar[r]\ar[d]&TX_0\ar[r]\ar[d,"\omega_0^\flat"]&0\ar[d]\\
0\ar[r]&T^*X_0\ar[r]&(T\F)^*
\end{tikzcd}
\]
from the tangent to the cotangent complex of $X_\bu$ being a
quasi-isomorphism.  The notion of a $0$-symplectic form is analogous
to that of a $0$-shifted symplectic structure in the sense
of~\cite[Definition~0.2]{pantev-toen-vaquie-vezzosi;shifted-symplectic}
(whereas the form on a Weinstein symplectic groupoid is analogous to a
$1$-shifted symplectic structure).  Nondegeneracy implies that
contraction with a $0$-symplectic form $\omega_\bu$ induces a linear
isomorphism
\[\begin{tikzcd}
\omega_\bu^\flat\colon\Vect_\bas(X_\bu)\ar[r,"\cong"]&
\Omega_\bas^1(X_\bu).
\end{tikzcd}\]

\begin{remark}
Conversely, if a form $\omega_\bu\in\Omega^2_\bas(X_\bu)$ induces a
linear isomorphism $\Vect_\bas(X_\bu)\cong\Omega_\bas^1(X_\bu)$, it is
not always the case that $\omega_\bu$ is nondegenerate.  For example,
consider the action groupoid $H\ltimes\R^2$, where
$H=\Z/2\Z\ltimes\Q^2$ is the semidirect product of the group of half
rotations of $\R^2$ around the origin, and translation by $\Q^2$.
Then $\Vect_\bas(X_\bu)=\Omega_\bas^1(X_\bu)=0$, so $\omega_\bu=0$
induces an isomorphism $\Vect_\bas(X_\bu)\cong\Omega_\bas^1(X_\bu)$,
but is not a $0$-symplectic form.
\end{remark}

\begin{proposition}\phantomsection\label{proposition;symgpd}
\begin{enumerate}
\item\label{item;tangent-algebroid}
Let $(X_\bu,\omega_\bu)$ be a $0$-symplectic groupoid with foliation
$\ca{F}$.  Then $(X_0,\omega_0)$ is a presymplectic manifold with
$\ker(\omega_0)=T\ca{F}$.
\item\label{item;integrate-algebroid}
Let $(X,\omega)$ be a presymplectic manifold with null foliation
$\ca{F}$.  Let $X_\bu$ be a source-connected foliation groupoid with
Lie algebroid equal to $T\ca{F}$.  Then
$\omega_\bu=(\omega,s^*\omega)$ is $X_\bu$-basic and hence defines a
$0$-symplectic structure on $X_\bu$.
\end{enumerate}
\end{proposition}

\begin{proof}
\eqref{item;tangent-algebroid}~This follows immediately from the
definition of presymplectic (Section~\ref{section;presymplectic}) and
$0$-symplectic.

\eqref{item;integrate-algebroid}~The form $\omega$ is horizontal with
respect to $\ca{F}$.  Since it is closed, it is also infinitesimally
invariant.  By Lemma~\ref{lemma;bas-inf}\eqref{item;bashorinf} it is
basic on $X_\bu$.
\end{proof}

\begin{example}\label{example2}
We revisit our Example~\ref{example1}.  There are many foliation
groupoids $X_\bu$ that integrate $\ker(\omega_0)$.  For instance, we
can take $\tilde{N}\to N$ to be any \'etale Lie group homomorphism,
let $\tilde{N}$ act on $X_0$ through this homomorphism, and take
$X_\bu$ to be the action groupoid $\tilde{N}\ltimes X_0$.  This
groupoid is not source-connected unless $\tilde{N}$ is connected.
Nevertheless, the presymplectic form $\omega_0$ is basic with respect
to the $\tilde{N}$-action, so $X_\bu$ is $0$-symplectic with
$0$-symplectic form $(\omega_0,s^*\omega_0)\in\Omega_\bas^2(X_\bu)$.
Possible choices of $\tilde{N}$ are $\tilde{N}=N$, or
$\tilde{N}=\Lie(N)$, the universal cover of the identity component of
$N$.  Another alternative is the surjective simply connected covering
group $\tilde{N}=\pi^{-1}(N)$, where $\pi\colon\R^n\to\T^n$ is the
projection.
\end{example}

\subsection{\'Etale stacks}\label{subsection;etale}

An atlas $X\to \X$ of a differentiable stack is \emph{\'etale} if, for
every morphism $M\to\X$ from a smooth manifold $M$, the pullback
$M\times_\X X \to M$ is \'etale, i.e.\ a local diffeomorphism.  An
\emph{\'etale stack} is a stack that admits an \'etale atlas.

\begin{lemma}\label{lemma;equivtoetale}
Let $\X$ be a differentiable stack. The following are equivalent:
\begin{enumerate}
\item\label{item;etale1}
$\X$ is an \'etale stack;
\item\label{item;etale2}
$\X\simeq\B X_\bu$ for some foliation groupoid $X_\bu$;
\item\label{item;etale3}
$\X\simeq\B X_\bu$ for some \'etale groupoid $X_\bu$.
\end{enumerate}
\end{lemma}  

\begin{proof}[Sketch of proof]
\eqref{item;etale1}$\iff$\eqref{item;etale3}:~Suppose $\X$ is \'etale.
Let $X \to \X$ be a Hausdorff \'etale atlas of $\X$ and let $X_\bu$ be
the Lie groupoid $X_\bu=( X\times_\X X\rightrightarrows X)$.  Then
$X_\bu$ is an \'etale groupoid by the definition of \'etale atlas, and
$\X \simeq \B X_\bu$.  Conversely, assume $X_1\rightrightarrows X_0$
is an \'etale groupoid, let $f\colon M\to \X$ be a map from a smooth
manifold $M$, and consider the diagram
\[X_1\rightrightarrows X_0 \to \X.\]
Forming the pullback over $f$ gives the diagram
\[X_1\times_\X M \rightrightarrows X_0\times_\X M \to M;\]
that is to say a presentation for $M$ by the Lie groupoid
$X_1\times_\X M\rightrightarrows X_0\times_\X M$.  This Lie groupoid
is also \'etale, and so the map $X_0\times_\X M\to M$ is \'etale, as
desired.

\eqref{item;etale2}$\iff$\eqref{item;etale3}:~This follows from
Lemma~\ref{lemma;moritaequivandstacks} and the fact that every
foliation groupoid $X_\bu$ is Morita equivalent to an \'etale groupoid
$Y_\bu$ (namely, take a complete transversal $\phi\colon Y_0\to X_0$
to the foliation of $X_0$, and let $Y_\bu=\phi^*X_\bu$ be the pullback
groupoid).
\end{proof}

\subsection{Vector fields on \'etale stacks}
\label{subsection;vector-etale}

We now recall the definition of vector fields on stacks
from~\cite{hepworth;vector-flow-stack} and show that equivalence
classes of vector fields on an \'etale stack correspond to basic
vector fields on a presenting Lie groupoid.

In~\cite[\S\,3]{hepworth;vector-flow-stack}, Hepworth constructs the
\emph{tangent stack functor} $T$, which is a lax endofunctor of the
$2$-category $\Stack$, and a transformation $\bo{p}\colon T\To\id$
from the tangent stack functor to the identity functor, which is
unique up to modification.  For a differentiable stack $\X$, we will
often denote the associated morphism $\bo{p}_\X\colon T\X\to\X$ simply
by $\bo{p}$.  By~\cite[Theorem~3.11]{hepworth;vector-flow-stack}, we
have an equivalence
\begin{equation}\label{equation;tanstackequiv}
\B TX_\bu\simeq T\B X_\bu
\end{equation}
which is natural in $X_\bu$, and a $2$-isomorphism $\bo{p}\cong\B
p_\bu$, where $p_\bu\colon TX_\bu\to X_\bu$ is the tangent groupoid
projection.
%
%

\begin{definition}[{\cite[\S\,4]{hepworth;vector-flow-stack}}]
\label{definition;vectorfield}
A \emph{vector field} on a stack $\X$ is a pair $(\bo{v},\balpha)$
consisting of a stack morphism $\bo{v}\colon\X\to T\X$ and a
$2$-isomorphism $\balpha\colon\bo{p}\circ\bo{v}\cong\id_\X$ to the
identity.  An \emph{equivalence} of vector fields $(\bo{v},\balpha)$
and $(\bo{w},\bbeta)$ is a $2$-arrow $\blambda\colon\bo{v}\To\bo{w}$
satisfying $\balpha=\bbeta\circ(\id_{\bo{p}}*\blambda)$.
\end{definition}

Vector fields on $\X$ and their equivalences form a groupoid
$\Bvect(\X)$.  We denote by $\Vect(\X)$ the set of equivalence classes
of the groupoid $\Bvect(\X)$, and we view $\Vect(\X)$ as a groupoid
with only identity arrows.  An equivalence of stacks
$\bphi\colon\X\to\Y$ induces an equivalence of groupoids
$\bphi^*\colon\Bvect(\Y)\to\Bvect(\X)$.
(See~\cite[\S\,4]{hepworth;vector-flow-stack}.)
\glossary{VectXx@$\Vect(\X)$, equivalence classes of vector fields on
  stack $\X$}%
\glossary{VectXX@$\Bvect(\X)$, groupoid of vector fields on
  stack $\X$}%

A multiplicative vector field $v_\bu\colon X_\bu\to TX_\bu$ on a Lie
groupoid $X_\bu$ gives rise to a stack morphism $\B v_\bu\colon\B
X_\bu\to\B TX_\bu\simeq T\B X_\bu$.  Hepworth shows that $\B v_\bu$
determines a vector field $(\B v_\bu,\balpha_{v_\bu})$ on $\X$ and
that the assignment $v_\bu\mapsto(\B v_\bu,\balpha_{v_\bu})$ defines a
functor
\[
\Vect_\mult(X_\bu)\longto\Bvect(\B X_\bu).
\]
%
If we restrict our attention to Lie groupoids where the object
manifold $X_0$ is Hausdorff, then by~\cite[Theorem~4.11,~Remark~5.4]%
{berwick-evans-lerman;lie-2-algebras-vector} this functor is an
equivalence of categories and is natural with respect to Morita
morphisms.
(\cite[Theorem~4.11]{berwick-evans-lerman;lie-2-algebras-vector} is
stated for Lie groupoids with $X_0$ and $X_1$ Hausdorff, but its proof
only makes use of this assumption on $X_0$.)  If $\X$ is an arbitrary
differentiable stack and $\bphi\colon\B X_\bu\simeq\X$ is a
presentation by a Lie groupoid with $X_0$ Hausdorff, we get an
equivalence of groupoids
\begin{equation}\label{equation;mult-equiv}
\begin{tikzcd}
\Vect_\mult(X_\bu)\ar[r]&\Bvect(\B X_\bu)\ar[r,"(\bphi^*)^{-1}"]&
\Bvect(\X),
\end{tikzcd}
\end{equation}
which Berwick-Evans and
Lerman~\cite[\S\,5]{berwick-evans-lerman;lie-2-algebras-vector} call a
``Lie $2$-algebra atlas'' on $\Bvect(\X)$.  A different choice of
presentation $\B Y_\bu\simeq\X$ leads to an equivalent Lie $2$-algebra
atlas $\Vect_\mult(Y_\bu)\simeq\Vect_\mult(X_\bu)$.

For \'etale stacks $\X$ the situation is simpler.  Let $\B X_\bu\simeq
\X $ be a presentation of $\X$.  Then $X_\bu$ is a foliation groupoid,
so Proposition~\ref{proposition;multiplicative-basic} tells us that
the set of equivalence classes of the groupoid $\Vect_\mult(X_\bu)$ is
equal to $\Vect_\bas(X_\bu)$.  By convention foliation groupoids
$X_\bu$ have Hausdorff $X_0$ (see \S\,\ref{subsection;foliation}), so
combining this with the equivalence~\eqref{equation;mult-equiv} we see
that the quotient map $\Bvect(\X)\to\Vect(\X)$ is an equivalence,
i.e.\ vector fields on \'etale stacks have no non-trivial
self-equivalences, and that the induced map
\begin{equation}\label{equation;bas-biject}
\begin{tikzcd}\Vect_\bas(X_\bu)\ar[r,"\cong"]&\Vect(\X)\end{tikzcd}
\end{equation}
is a bijection.

\begin{proposition}\label{proposition;vector-stack}
Let $\X$ be an \'etale stack.  There is a unique Lie algebra structure
on the set $\Vect(\X)$ with the property that for every presentation
$\B X_\bu\to\X$ by a foliation groupoid $X_\bu$ the
bijection~\eqref{equation;bas-biject} is a Lie algebra isomorphism.
\end{proposition}

\begin{proof} By 
Lemma~\ref{lemma;hausdorff}, there exists a presentation $\B
X_\bu\simeq\X$ by a Lie groupoid with $X_0$ Hausdorff, and so the
bijection~\eqref{equation;bas-biject} endows $\Vect(\X)$ with the
structure of a Lie algebra. Uniqueness of this Lie algebra structure
now follows immediately from
Proposition~\ref{proposition;vector-form-morita}\eqref{item;vector-morita}.
\end{proof}

\subsection{Differential forms on \'etale stacks}
\label{subsection;form-etale}

The functor $\Omega^k\colon\group{Diff}\to\group{Set}$ that takes a
manifold $M$ to its set of $k$-forms $\Omega^k(M)$ is a sheaf on the
site $\group{Diff}$.  By viewing $\Omega^k(M)$ as a groupoid with only
identity arrows, we may think of the sheaf $\Omega^k$ as a
(non-differentiable) stack.

Let $\X$ be an \'etale stack.  We define a \emph{differential form of
  degree $k$} on $\X$ to be a morphism of stacks $\X\to\Omega^k$.  The
collection of $k$-forms on $\X$ is a groupoid
$\Bomega^k(\X)=\Hom(\X,\Omega^k)$.  For a morphism of \'etale stacks
$\bphi\colon\X\to\Y$ and a $k$-form $\bzeta\colon\Y\to\Omega^k$ on
$\Y$ we define the \emph{pullback form} on $\X$ by
$\bphi^*\bzeta=\bzeta\circ\bphi\colon\X\to\Omega^k$.

Since $\Omega^k$ takes values in $\group{Set}$, the groupoid
$\Hom(\X,\Omega^k)$ has no nontrivial arrows.  So $\Bomega^k$ really
takes values in $\group{Set}$, and henceforth we write
$\Omega^k=\Bomega^k$ to emphasize this.  By the Yoneda lemma, the set
$\Omega^k(M)$ of $k$-forms on $M$ (considered as a manifold) is
naturally isomorphic to the set $\Omega^k(M)$ of $k$-forms on $M$
(considered as a stack).

\begin{remark}
This notion of a differential form on a stack corresponds to the
notion of a differential form on the diffeological coarse moduli space
of the stack (see~\cite{watts-wolbert;diffeological-coarse}), or to
that of a basic form on a presenting Lie groupoid (see
Proposition~\ref{proposition;basicformsarethesame} below), and is
suitable for the purposes of this paper.  A different notion (which
corresponds to the notion of a simplicial or Bott-Shulman differential
form on a presenting Lie groupoid, and which is more adequate for
other purposes, especially those involving non-\'etale stacks) is
explained for instance in~\cite[\S\,3]{behrend-xu;stacks-gerbes}.
\end{remark}

Let $\X$ be an \'etale stack, and let $\B X_\bu \to \X$ be a
presentation of $\X$.  Composing the diagram $X_1\rightrightarrows X_0
\to \X$ with a $k$-form $\bzeta\colon\X\to\Omega^k$ determines a
commutative diagram $X_1\rightrightarrows X_0\to\Omega^k$, i.e.\ a
basic form on $X_\bu$.  This defines a map
\begin{equation}\label{equation;basicformassignment}
\Omega^k(\X)\longto\Omega_\bas^k(X_\bu).
\end{equation}

\begin{proposition}\label{proposition;basicformsarethesame}
Let $\X$ be an \'etale stack and let $\B X_\bu\to\X$ be a presentation
of $\X$.  The map~\eqref{equation;basicformassignment} is a bijection
$\Omega^k(\X)\cong\Omega_\bas^k(X_\bu)$.
\end{proposition}

\begin{proof}
We display an inverse to~\eqref{equation;basicformassignment}.  Let
$D(X_\bu)$ be the $2$-truncated semisimplicial nerve of the Lie
groupoid $X_1\rightrightarrows X_0$, that is to say the diagram
\[
\begin{tikzcd}
D(X_\bu)\colon\quad X_2=X_1\times_{X_0} X_1\ar[r]\ar[r,shift
  left=2]\ar[r,shift right=2]&X_1\ar[r,shift left]\ar[r,shift
  right]&X_0,
\end{tikzcd}
\]
where the arrows from $X_2$ to $X_1$ are the two projections and the
composition map.  It is known (see
e.g.~\cite[Appendix~A]{carchedi;topological-differentiable-stacks})
that $\X$ is the weak colimit in $\Stack$ of the diagram $D(X_\bu)$.
Let $\zeta_\bu=(\zeta_0,\zeta_1)$ be a basic $k$-form on $X_\bu$.
Then $\zeta_0$, viewed as a map $\zeta_0\colon X_0\to\Omega^k$, has
the property that the diagram
\[
\begin{tikzcd}
X_1\ar[r,shift left]\ar[r,shift right]&X_0\ar[r,"\zeta_0"]&\Omega^k
\end{tikzcd}
\]
commutes.  Using this, it is straightforward to check that the diagram
$\begin{tikzcd}[cramped,sep=small]D(X_\bu)\ar[r,"\zeta_0"]&\Omega^k
\end{tikzcd}$
commutes as well.  Hence,  by the universal property
of weak colimits, $\zeta_0$ descends to a morphism of stacks
$\bzeta\colon\X\to\Omega^k$.  Since $\Omega^k$ takes values in the
$1$-category $\group{Set}$ the morphism $\bzeta$ is unique; and the
assignment $\zeta_\bu\mapsto\bzeta$ is inverse
to~\eqref{equation;basicformassignment}.
\end{proof}

Abusing notation, we will write
\[\B\zeta_\bu\in\Omega^k(\X)\]
\glossary{Bz@$\B\zeta_\bu$, differential form on $\B X_\bu$ determined
  by basic form $\zeta_\bu$ on $X_\bu$}%
for the differential form on $\X=\B X_\bu$ determined by the basic
$k$-form $\zeta_\bu\in\Omega^k(X_\bu)$ via the
bijection~\eqref{equation;basicformassignment}.  The following is the
analogue of Proposition~\ref{proposition;vector-stack}.

\begin{proposition}\label{proposition;form-stack}
Let $\X$ be an \'etale stack.  There is a unique differential graded
algebra structure on the set $\Omega^\bu(\X)$ so that for every
presentation $\B X_\bu \to X_\bu$ by a foliation groupoid $X_\bu$ the
bijection~\eqref{equation;basicformassignment} is an isomorphism of
differential graded algebras.  There are unique contraction and
derivation operations
\begin{align*}
\iota&\colon\Vect(\X)\times\Omega^\bu(\X)\longto\Omega^{\bu-1}(\X),\\
\ca{L}&\colon\Vect(\X)\times\Omega^\bu(\X)\longto\Omega^\bu(\X)
\end{align*}
which correspond to the operations~\eqref{equation;contraction-lie}
via the bijections~\eqref{equation;bas-biject}
and~\eqref{equation;basicformassignment}.
\end{proposition}

\begin{proof}
This follows immediately from
Proposition~\ref{proposition;vector-form-morita}\eqref{item;form-morita}
and the naturality of the operations~\eqref{equation;contraction-lie}
with respect to Morita morphisms.
\end{proof}

\subsection{Symplectic stacks}\label{subsection;symplectic-stack}

Let $\X$ be an \'etale stack.  We call a $2$-form
$\bomega\in\Omega^2(\X)$ \emph{symplectic} if for some (and hence for
every) foliation groupoid $X_\bu$ presenting $\X$ the basic $2$-form
$\omega_\bu\in\Omega_\bas(X_\bu)$ corresponding to $\bomega$ via the
isomorphism~\eqref{equation;basicformassignment} is $0$-symplectic in
the sense of \S\,\ref{subsection;symplectic-groupoid}.  A
\emph{symplectic stack} is a pair $(\X,\bomega)$ consisting of an
\'etale stack $\X$ and a symplectic form $\bomega$.
\glossary{Xxomega@$(\X,\bomega)$, symplectic stack}%
\glossary{omegax@$\bomega$, symplectic form on stack}%

This extends the notion of a symplectic stack introduced
in~\cite[\S\,2.12]{lerman-malkin;deligne-mumford} in that we allow our
stacks to be non-separated.  If $(\X,\bomega)$ is a symplectic stack,
the linear map $\Vect(\X)\to\Omega^1(\X)$ defined by
$\group{v}\mapsto\iota_{\group{v}}\bomega$ is an isomorphism.  The
next statement is an immediate consequence of the definitions.

\begin{proposition}\label{proposition;symstack}
Let $(X_\bu,\omega_\bu)$ be a $0$-symplectic groupoid.  Then $(\B
X_\bu,\B\omega_\bu)$ is a symplectic stack, where
$\B\omega_\bu\in\Omega^2(\B X_\bu)$ is the form corresponding to
$\omega_\bu\in\Omega^2_\bas(X_\bu)$ under the isomorphism $\Omega^2(\B
X_\bu)\cong\Omega^2_\bas(X_\bu)$ of
Proposition~\ref{proposition;basicformsarethesame}.  Conversely, if
$(\X,\bomega)$ is a symplectic stack and $\X\simeq\B X_\bu$, then
$X_\bu$ is a $0$-symplectic groupoid.
\end{proposition}

From Propositions~\ref{proposition;symgpd}
and~\ref{proposition;symstack}, we see that any presymplectic manifold
$(X,\omega)$ gives rise to many symplectic stacks, each of which we
may interpret as a stacky quotient of $X$ along the null foliation of
$\omega$.

\section{Lie \texorpdfstring{$2$-groups}{2-groups} and Lie group stacks}
\label{section;lie-stack}

This section starts with a review of Lie $2$-groups and Lie group
stacks, which we view as group objects internal to the $2$-categories
$\Liegpd$ and $\Diffstack$, respectively.  (Some basic definitions
regarding group objects in a $2$-category are recalled in
Appendix~\ref{appendix;groupobjects}.)  Of main interest to us are
connected \'etale Lie group stacks, which according
to~\cite{trentinaglia-zhu;strictification} are always equivalent to
strict Lie group stacks.  We show that the action of an \'etale Lie
group stack on an \'etale stack can also be strictified.  We end with
a discussion of basic features of \'etale Lie group stacks, such as
the Lie $2$-algebra, the adjoint action, the structure of compact
\'etale Lie group stacks, and maximal stacky tori.

\subsection{Lie \texorpdfstring{$2$-groups}{2-groups}}
\label{subsection;2-group}

A \emph{weak Lie $2$-group} is a weak $2$-group in the $2$-category
$\Liegpd$, as in Definition~\ref{definition;internalgp}.  A
\emph{strict Lie $2$-group} is a strict $2$-group in $\Liegpd$. We
will refer to strict Lie $2$-groups as simply \emph{Lie $2$-groups}.

More explicitly, a (strict) Lie $2$-group is a Lie groupoid $G_\bu$ so
that $G_1$ and $G_0$ are both Lie groups, and all the groupoid
structure maps are Lie group homomorphisms.  We will write $g\cdot h$
for the group product of $g,h\in G_i$ and $g\circ h$ for the groupoid
product (composition) of composable $g$, $h\in G_1$.  We denote the
group identity of $G_1$ and $G_0$ by $1$, and $ u\colon G_0\to G_1$
the groupoid identity bisection.  We write $m_\bu\colon G_\bu\times
G_\bu\to G_\bu$ for group multiplication in $G_\bu$, and note that
$m_\bu$ is a Lie groupoid homomorphism.  We use $(\cdot)\n$ to denote
both inverse with respect to the groupoid structure and the group
structure on $G_\bu$; the meaning should be clear from the context.
\glossary{*@$(\cdot)\n$, inversion law of groupoid or $2$-group}%

A \emph{morphism of Lie $2$-groups} is a strict morphism of strict
$2$-groups in $\Liegpd$, as in
Definition~\ref{definition;hominternal2group}.  Equivalently, it is a
map of Lie groupoids which preserves the group structure on both the
object and the arrow manifolds.  A \emph{Morita morphism} of Lie
$2$-groups is a morphism of Lie $2$-groups which is essentially
surjective and fully faithful, i.e.\ a Morita morphism of the
underlying Lie groupoids
(Definition~\ref{definition;moritaequivalence}).  Two Lie $2$-groups
$G_\bu$ and $H_\bu$ are \emph{Morita equivalent} if there exists a Lie
$2$-group $K_\bu$ and a zigzag of Morita morphisms of Lie $2$-groups
$K_\bu \to G_\bu$ and $K_\bu \to H_\bu$.

The coarse quotient $G_0/G_1$ of a Lie $2$-group $G_\bu$ is a (not
necessarily Hausdorff) topological group, namely
$G_0/G_1=G_0/t(\ker(s))$.  The coarse quotient is preserved under
Morita equivalence.

A \emph{weak action} of a Lie $2$-group $G_\bu$ on a Lie groupoid
$X_\bu$ is a weak action as in
Definition~\ref{definition;actioninternal}, by considering $G_\bu$ to
be a strict $2$-group in $\Liegpd$.  A \emph{strict action} is also as
in Definition~\ref{definition;actioninternal}.  We will usually
abbreviate \emph{strict action} to \emph{action}.  If $G_\bu$ is a Lie
$2$-group acting on $X_\bu$ and $X_\bu'$, an \emph{equivariant map}
$X_\bu \to X_\bu'$ is as in
Definition~\ref{definition;equivariantinternal}, and is always assumed
to be strict.

A (strict) action of a Lie $2$-group $G_\bu$ on $X_\bu$ is equivalent
to a morphism of Lie groupoids $G_\bu\times X_\bu\to X_\bu$ both of
whose component maps
\begin{equation}\label{equation;actions}
\begin{split}
G_0\times X_0\longto X_0&\colon(g,x)\longmapsto L(g)(x)=g\cdot x,\\
G_1\times X_1\longto X_1&\colon(k,f)\longmapsto L(k)(f)=k*f,
\end{split}
\end{equation}
define Lie group actions.

\begin{example}\label{example4} 
Consider the torus $G=\T^n$ and its immersed subgroup $N$ of
Examples~\ref{example1} and~\ref{example2}.  Let $\tilde{N}\to N$ be
an \'etale homomorphism as in Example~\ref{example2}.  Form the action
groupoid $G_\bu=(\tilde{N}\ltimes G\rightrightarrows G)$, where
$\tilde{N}$ acts on $G$ by left translations via $N$.  Then $G_\bu$
has the structure of a Lie $2$-group.  The group structure on
$\tilde{N}\ltimes G$ is the product group structure $\tilde{N}\times
G$.  Taking $X_\bu=(\tilde{N}\ltimes X_0\rightrightarrows X_0)$, from
the action of $G$ on $X_0$ we obtain an action $G_\bu\times X_\bu\to
X_\bu$, which on arrows is given by
\[(n,g)\cdot(n',x)=(nn',g\cdot x).\]
\end{example}

\subsection{Lie \texorpdfstring{$2$-algebras}{2-algebras}}
\label{subsection;2-algebra}

A \emph{strict Lie $2$-algebra} is a Lie groupoid
$\g_1\rightrightarrows\g_0$ so that each $\g_i$ is a Lie algebra and
the groupoid structure maps are Lie algebra homomorphisms.  The strict
Lie $2$-algebra of a strict Lie $2$-group $G_1\rightrightarrows G_0$
is the Lie groupoid $\Lie(G_\bu)=(\g_1\rightrightarrows\g_0)$ obtained
by applying the Lie functor to the structure maps of $G_\bu$.  If
$\g_\bu=\Lie(G_\bu)$, we say that $G_\bu$ \emph{integrates} $\g_\bu$.
\glossary{Lie@$\Lie$, Lie functor}%

We will not consider weak Lie $2$-algebras, so we refer to strict Lie
$2$-algebras as simply Lie $2$-algebras.  Recall
(Definition~\ref{definition;Liealgebroid}) that the Lie algebroid of
$G_\bu$ is denoted $\Alg(G_\bu)$, not $\Lie(G_\bu)$.

A \emph{Morita morphism} of strict Lie $2$-algebras is a Morita
morphism of the underlying Lie groupoids
(Definition~\ref{definition;moritaequivalence}) which preserves the
Lie algebra structure maps.  Two strict Lie $2$-algebras $\g_\bu$ and
$\h_\bu$ are \emph{Morita equivalent} if there is a strict Lie
$2$-algebra $\mathfrak{k}_\bu$ and a zig-zag of Morita morphisms of
Lie $2$-algebras $\mathfrak{k}_\bu\to \g_\bu$ and $\mathfrak{k}_\bu
\to \h_\bu$.

\begin{lemma}\phantomsection\label{lemma;moritaequivliealg}
\begin{enumerate}
\item\label{item;morita-lie-algebra}
Let $\phi_\bu\colon G_\bu\to H_\bu$ be a Morita morphism of Lie
$2$-groups.  Then $\Lie(\phi_\bu)\colon\g_\bu\to\h_\bu$ is a Morita
morphism of Lie $2$-algebras.
\item\label{item;morita-lie-lie}
Let $\g$ and $\g'$ be Lie algebras, considered as Lie $2$-algebras
$\g\rightrightarrows\g$ and $\g'\rightrightarrows\g'$.  Let
$\phi_\bu\colon\h_\bu\to\g$ be a Morita morphism and
$\psi_\bu\colon\h_\bu\to\g'$ a morphism of Lie $2$-algebras.  Then
there is a unique Lie algebra morphism $\zeta\colon\g\to\g'$ such that
$\zeta\circ\phi_\bu=\psi_\bu$.  In particular, if $\g$ and $\g'$ are
Morita equivalent as Lie $2$-algebras, then they are isomorphic as Lie
algebras.
\end{enumerate}
\end{lemma}

\begin{proof}
\eqref{item;morita-lie-algebra}~The Lie functor ``$\Lie$'' preserves
fibred products and takes surjective submersions to surjective
submersions.  Therefore ``$\Lie$'' preserves essential surjectivity
and full faithfullness.

\eqref{item;morita-lie-lie}~Let $\phi\colon\h_\bu\to\g$ be a Morita
morphism of Lie $2$-algebras.  Then $\phi_0\colon\h_0\to\g$ is
surjective, so $\g\cong\h_0/\ker(\phi_0)$.  Note that
$t(\ker(s))\subseteq\ker(\phi_0)$, since $\phi$ is a Lie groupoid
morphism.  Since $\phi$ is a Morita equivalence, we have
\[\dim\g=2\dim\h_0-\dim\h_1=\dim\h_0-\dim t(\ker(s)),\]
and so $t(\ker(s))=\ker(\phi_0)$.  So $\g\cong\h_0/t(\ker(s))$.
Similarly, $t(\ker(s))\subseteq\ker(\psi_0)$, so $\psi_0$ descends to
a unique Lie algebra map $\zeta\colon\g\to\g'$, which satisfies
$\zeta\circ\phi_\bu=\psi_\bu$.  The last assertion is an immediate
consequence of this.
\end{proof}

\begin{example}\label{example5}
The Lie $2$-algebra of the $2$-group $G_\bu$ in Example~\ref{example4}
is $\lie{n}\ltimes\g\rightrightarrows\g$, where $\lie{n}\ltimes\g$ is
isomorphic as a Lie algebra to the abelian Lie algebra
$\lie{n}\oplus\g$.
\end{example}

\subsection{Crossed modules}\label{subsection;crossed}

A \emph{crossed module of Lie groups} is a quadruple
$(G,H,\partial,\alpha)$, where $G$ and $H$ are Lie groups,
$\partial\colon H\to G$ is a Lie group homomorphism, and $\alpha\colon
G\to\Aut(H)$ is an action of $G$ on $H$ subject to the requirements
$\partial(\alpha(g)(h))=g\partial(h)g^{-1}$ and
$\alpha(\partial(h))(h')=hh'h^{-1}$ for all $g\in G$ and $h$, $h'\in
H$.  We will usually abbreviate $\alpha(g)(h)$ to $^gh$.  A
\emph{morphism (of crossed modules of Lie groups)}
$(\psi_G,\psi_H)\colon(G,H,\partial,\alpha)\to(G',H',\partial',\alpha')$
is a pair of Lie group homomorphisms $\psi_G\colon G\to G'$, and
$\psi_H\colon H\to H'$ which commute with the structure maps and
actions.

There is an equivalence of categories between strict Lie $2$-groups
and their strict homomorphisms on one hand, and crossed modules of Lie
groups and their morphisms on the other; see
\cite{brown-spencer;crossed-modules-groupoids-fundamental} and
\cite[Proposition~3.14,~\S\,8.4]{baez-lauda;2-groups}.  (This can be
extended to an equivalence of appropriately defined $2$-categories,
though we will not need this here).  Since many statements involving
Lie $2$-groups can be more compactly stated and proved in terms of
crossed modules, we will use this equivalence throughout what follows.

Let us recall how this equivalence of categories looks on objects. A
crossed module of Lie groups $(G,H,\partial,\alpha)$ gives rise to the
Lie $2$-group $G_\bu$ with $G_0=G$ and $G_1=H\rtimes_\alpha G$.  As a
Lie groupoid, $G_\bu$ is the action groupoid for the action of $H$ on
$G$ by left translations of $\partial(H)$.  As a Lie group,
$H\rtimes_\alpha G$ is the semidirect product of $H$ and $G$ with
respect to the action $\alpha\colon G\to\Aut(H)$.  Conversely, a
$2$-group $G_\bu$ determines the crossed module
$(G,H,\partial,\alpha)$: we have $G=G_0$, $H=\ker(s)$,
$\partial=t|_H\colon H\to G$, and $\alpha$ is the conjugation action
of $G_1$ on $H$ composed with the identity bisection $ u\colon G_0\to
G_1$.  The Lie $2$-algebra $\Lie(H\rtimes_\alpha G)$ is
$\h\rtimes_{\alpha}\g\rightrightarrows\g$, where $\h\rtimes_\alpha\g$
is the semidirect product Lie algebra obtained by applying the Lie
functor to $\alpha$.

A crossed module is a nonabelian version of a $2$-term chain complex,
whose ``homology'' groups are $\ker(\partial)$ and $\coker(\partial)$.
A Morita morphism is a nonabelian version of a quasi-isomorphism.

\begin{definition}
\label{definition;moritacrossed} 
A \emph{Morita morphism (of crossed modules of Lie groups)} is a
morphism
\[
(\phi_G,\phi_H)\colon(G,H,\partial,\alpha)\to(G',H',\partial',\alpha')
\]
that induces (abstract) group isomorphisms
$\ker(\partial)\cong\ker(\partial')$ and
$\coker(\partial)\cong\coker(\partial')$.
\end{definition}

Thus the kernel and the cokernel of $\partial$ are Morita invariants
of a crossed module $(G,H,\partial,\alpha)$.  Note that the cokernel
$G/\partial(H)$ is the coarse quotient group of the corresponding
$2$-group.

The second countability axiom (which we impose on all manifolds; see
Section~\ref{section;notation}) is necessary for the following
statement to be true.  For instance, if $(G,H,\partial,\alpha)$ is a
crossed module of Lie groups and $G^d$, resp.\ $H^d$, denotes $G$,
resp.\ $H$, equipped with the discrete topology, then the identity map
$(G^d,H^d,\partial,\alpha)\to(G,H,\partial,\alpha)$ is a Morita
morphism of crossed modules of (non-second countable) Lie groups, but
the corresponding morphism of Lie $2$-groups $G_\bu^d\to G_\bu$ is not
Morita.

\begin{lemma}\label{lemma;morita-crossed}
Let $\phi_\bu\colon G_\bu\to G_\bu'$ be a morphism of Lie $2$-groups
and let $(\phi_G,\phi_H)$ be the associated morphism of crossed
modules.  Then the following are equivalent:
\begin{enumerate}
\item\label{item;morita-2}
$\phi_\bu$ is a Morita morphism of Lie $2$-groups;
\item\label{item;morita-cross}
$(\phi_G,\phi_H)$ is a Morita morphism of crossed modules;
\item\label{item;morita-fibre}
$G'=\partial'(H')\phi_G(G)$ and the map $\phi_H\times\partial\colon
  H\to H'\times_{G'}G$ is bijective.
\end{enumerate}
\end{lemma}

\begin{proof}[Sketch of proof]
The implications
\eqref{item;morita-2}$\implies$\eqref{item;morita-cross}$\implies$%
\eqref{item;morita-fibre} are straightforward.
From~\eqref{item;morita-fibre} one deduces that the morphism
$\phi_\bu\colon G_\bu\to G_\bu'$ is essentially surjective and fully
faithful considered as a morphism of abstract $2$-groups (i.e.\ strict
$2$-groups in the $2$-category of set-theoretic groupoids).  Now use
the following two facts to conclude that $\phi_\bu$ is a Morita
morphism of Lie $2$-groups: (1)~a surjective homomorphism of Lie
groups is a submersion (which follows from Sard's theorem); (2)~given
two Lie group homomorphisms $K_1\to K$ and $K_2\to K$, the
set-theoretic fibred product $K_1\times_KK_2$ is a Lie group and is a
fibred product in the category of Lie groups (which follows from the
closed subgroup theorem).
\end{proof}

\begin{remark}\label{remark;morita}
We record three special cases of the lemma.  Let
$(G,H,\partial,\alpha)$ be a crossed module.  First, suppose we are
given a closed subgroup $G'$ of $G$ with the property
$\partial(H)G'=G$.  Then we can form the restricted crossed module
$(G',H',\partial',\alpha')$ with $H'=\partial^{-1}(G')$, and the
inclusion into $(G,H,\partial,\alpha)$ is a Morita morphism.  Second,
given a Lie group extension (e.g.\ a covering homomorphism)
$\phi\colon\hat{G}\to G$ we can form the pullback extension
$\hat{H}=H\times_G\hat{G}$ of $H$.  Define
$\hat\partial\colon\hat{H}\to H$ by
$\hat\partial(h,\hat{g})=\phi(\hat{g})$ and a $\hat{G}$-action
$\hat\alpha$ on $\hat{H}$ by $^{\hat{g}}(h,\hat{g}')=
\bigl({}^{\phi(\hat{g})}h,\hat{g}\hat{g}'\hat{g}^{-1}\bigr)$.  Then
$(\hat{G},\hat{H},\hat\partial,\hat\alpha)$ is a crossed module and
$\phi$ induces a Morita morphism
$(\hat{G},\hat{H},\hat\partial,\hat\alpha)\to(G,H,\partial,\alpha)$.
The third case is the second case run in reverse: given a closed
normal subgroup $N$ of $G$ that is contained in $\partial(H)$, we can
form the quotient crossed module
$(\bar{G},\bar{H},\bar\partial,\bar\alpha)$ with $\bar{G}=G/N$ and
$\bar{H}=H/\partial^{-1}(N)$, and we have a Morita morphism
$(G,H,\partial,\alpha)\to(\bar{G},\bar{H},\bar\partial,\bar\alpha)$.
\end{remark}

We can also reformulate the notion of a strict Lie $2$-group action in
terms of the associated crossed module of Lie groups.

\begin{definition}
\label{definition;crossedaction}
Let $(G,H,\partial,\alpha)$ be a crossed module of Lie groups and let
$X_\bu$ be a Lie groupoid.  An \emph{action} of
$(G,H,\partial,\alpha)$ on $X_\bu$ consists of three smooth actions
\begin{alignat}{2}
\label{item;firstaction}
G\times X_0&\longto X_0\colon\qquad&(g,x)&\longmapsto L(g)(x)=g\cdot x,\\
\label{item;secondaction}
G\times X_1&\longto X_1\colon&(g,f)&\longmapsto L(g)(f)=g*f,\\
\label{item;thirdaction}
H\times X_1&\longto X_1\colon&(h,f)&\longmapsto L(h)(f)=h*f,
\end{alignat}
satisfying the compatibility
conditions~\eqref{equation;functor1}--\eqref{equation;conjugation}
below: for all $g\in G$ the pair of maps $L(g)\colon X_0\to X_0$ and
$L(g)\colon X_1\to X_1$ is an endofunctor of $X_\bu$, i.e.\
\begin{gather}
\label{equation;functor1}
g*u(x)=u(g\cdot x),\qquad g\cdot s(f)=s(g*f),\qquad g\cdot
t(f)=t(g*f),\\
\label{equation;functor2}
g*(f_1\circ f_2)=(g*f_1)\circ(g*f_2),
\end{gather}
for all $x\in X_0$ and $f$, $f_1$, $f_2\in X_1$ for which $f_1\circ
f_2$ is defined; for all $h\in H$ the map $L(h)\colon X_1\to X_1$ is a
natural transformation from the identity functor to the functor
$L(\partial(h))$, i.e.\
\begin{gather}
\label{equation;natural1}
s(h*f)=s(f),\quad t(h*f)=\partial(h)\cdot t(f),\\
\label{equation;natural2}
h*f=(h*u(t(f))\circ f=(\partial(h)*f)\circ(h* u(s(f)),
\end{gather}
for all $f\in X_1$; for all $g\in G$, $h\in H$, and $f\in X_1$, one
has $L({}^gh)=L(g)\circ L(h)\circ L(g)^{-1}$, i.e.\
\begin{equation}\label{equation;conjugation}
g*(h*f)={}^gh*(g*f).
\end{equation}
\end{definition}

We omit the proof of the following straightforward lemma.

\begin{lemma}\label{lemma;crossedactionequivalent}
Let $G_\bu=(G_1\rightrightarrows G_0)$ be a Lie $2$-group with
associated crossed module
\[(G=G_0,H=\ker(s),\partial=t|_H,\alpha),\]
and let $X_\bu$ be a Lie groupoid.  Let $a_\bu\colon G_\bu\times
X_\bu\to X_\bu$ be a strict $G_\bu$-action on $X_\bu$.  Then the maps
\begin{align*}
a_0\colon G\times X_0&\longto X_0\\
a_1\circ(u\times\id_{X_1})\colon G\times X_1&\longto X_1\\
a_1|_H\colon H\times X_1&\longto X_1
\end{align*}
define an action of $(G,H,\partial,\alpha)$ on $X_\bu$.  Conversely,
given a $(G,H,\partial,\alpha)$-action on $X_\bu$, the formulas
\begin{alignat*}{2}
G_0\times X_0&\longto X_0\colon&(g,x)&\longmapsto g\cdot x,\\
G_1\times X_1&\longto X_1\colon\qquad&((h,g),x)&\longmapsto h*(g*x),
\end{alignat*}
define a strict $G_\bu$-action on $X_\bu$, where we identify $G_0=G$
and $G_1=H\rtimes_\alpha G$.  These two constructions are inverse, and
so determine a bijection between the set of strict $G_\bu$-actions on
$X_\bu$ and the set of $(G,H,\partial,\alpha)$-actions on $X_\bu$.
\end{lemma}

The next lemma records a simple consequence of
Definition~\ref{definition;crossedaction}.  We use that every element
$\xi\in\g$ gives birth to two vector fields, namely the vector field
$\xi_{X_1}$ induced via the $G$-action on $X_1$ and the vector field
$\xi_{X_0}$ induced via the $G$-action on $X_0$.  Similarly, every
$\eta\in\h$ gives birth to three vector fields: the vector field
$\eta_{X_1}$ induced via the $H$-action on $X_1$, and the two vector
fields $\partial(\eta)_{X_1}$ and $\partial(\eta)_{X_0}$ induced by
the element $\partial(\eta)\in\g$ via the $G$-actions on $X_1$ and
$X_0$.

\begin{lemma}\label{lemma;orbits}
Let $G_\bu$ be a strict Lie $2$-group acting strictly on a Lie
groupoid $X_\bu$.  Let $(G=G_0,H,\partial,\alpha)$ be the crossed
module associated with $G_\bu$.  Let $\eta\in\h$.  The vector field
$\eta_{X_1}$ is tangent to the source fibres, right-invariant, and
$t$-related to the vector field $\partial(\eta)_{X_0}$.  The vector
field $\eta_{X_1}-\partial(\eta)_{X_1}$ is tangent to the target
fibres, left-invariant, and $s$-related to the vector field
$\partial(\eta)_{X_0}$.
\end{lemma}

\begin{proof}
Let $h\in H$.  Property~\eqref{equation;natural1} can be restated as
$s\circ L(h)=s$ and $t\circ L(h)=L(\partial(h))\circ t$.  In other
words, $s\colon X_1\to X_0$ is $H$-invariant and $t\colon X_1\to X_0$
is equivariant with respect to the homomorphism $\partial\colon H\to
G$.  It follows that $\eta_{X_1}$ is tangent to the fibres of $s$ and
that the vector fields $\eta_{X_1}$ and $\partial(\eta)_{X_0}$ are
$t$-related.  Let $f\in X_1$ have source $s(f)=x$ and target $t(f)=y$.
Right composition with $f$, $R(f)(f')=f'\circ f$, defines a map
$R(f)\colon s^{-1}(y)\to s^{-1}(x)$, and it follows
from~\eqref{equation;natural2} that
$$L(h)(f)=R(f)\bigl(L(h)(u(y))\bigr).$$
Differentiating this identity with respect to $h$ yields
$\eta_{X_1,f}=R(f)_*(\eta_{X_1,u(y)})$, which tells us that
$\eta_{X_1}$ is right-invariant.  This proves the first assertion.
The second assertion is proved similarly, by considering the element
$k=\partial(h)^{-1}h=h\partial(h)^{-1}\in G_1=H\rtimes G$ and the left
composition map $L(f)\colon t^{-1}(y)\to t^{-1}(x)$, and by noting the
properties
$$
s\circ L(k)=L(\partial(h)^{-1})\circ s,\quad t\circ L(k)=t,\quad
L(k)(f)=L(f)\bigl(L(k)(u(x))\bigr),
$$
which follow from~\eqref{equation;natural1}
and~\eqref{equation;natural2}.
\end{proof}  

\begin{example}\label{example6}
The crossed module associated to the Lie $2$-group $G_\bu$ of
Example~\ref{example4} is the homomorphism $\tilde{N}\to G$ obtained
by composing the map $\tilde{N}\to N$ with the inclusion $N\to G$.
The action of $G$ on $\tilde{N}$ is trivial.
\end{example}

\subsection{Foliation \texorpdfstring{$2$-groups}{2-groups}}
\label{subsection;foliation-2group}

Let $G_\bu$ be a Lie $2$-group with crossed module
$(G=G_0,H,\partial,\alpha)$.  The subgroup $\partial(H)$ of $G_0$ is
normal, so
\[\Lie(t)(\ker(\Lie(s)))=\Lie(\partial)(\h)\]
is an ideal of $\g$, and $\g/\Lie(\partial)(\h)$ is a Lie algebra.

\begin{definition}
\label{definition;foliation2group}
A Lie $2$-group $G_\bu$ is a \emph{foliation $2$-group} if any of the
following equivalent conditions hold: (1)~$G_\bu$ is a foliation
groupoid; (2)~the homomorphism $\partial\colon H\to G$ has discrete
kernel; or (3)~the homomorphism $\Lie(\partial)\colon\h\to\g$ is
injective.  A foliation $2$-group $G_\bu$ is \emph{effective} if
$\partial$ is injective.  We say $G_\bu$ is an \emph{\'etale
  $2$-group} if either of the following equivalent conditions holds:
(1)~$G_\bu$ is an \'etale groupoid; or (2)~$H$ is discrete.
\end{definition}

When $G_\bu$ is a foliation $2$-group, we will consider $\h$ as an
ideal of $\g$.  Being a foliation $2$-group is a Morita invariant
property, and so is being effective.

We denote by $\xi_L\in\Vect(G)$ the left-invariant vector field on a
Lie group $G$ induced by a Lie algebra element $\xi\in\g$.  We call a
basic vector field $v_\bu=(v_0,v_1)$ on a foliation $2$-group $G_\bu$
\emph{left-invariant} if for $i=1$, $2$ the vector field $v_i$ is
invariant under the left multiplication action of $G_i$ on itself.
The left-invariant basic vector fields form a Lie subalgebra of
$\Vect_\bas(G_\bu)$ denoted by $\Vect_\bas(G_\bu)_L$.
\glossary{xil@$\xi_L$, left-invariant vector field on Lie group $G$
  induced by $\xi\in\Lie(G)$}%
\glossary{Vectbasl@$\Vect_\bas(G_\bu)_L$, left-invariant basic vector
  fields on Lie $2$-group $G_\bu$}%

We require some basic structural results on foliation $2$-groups.  The
first result says that the Lie $2$-algebra of a foliation $2$-group
$G_\bu$ is equivalent to the quotient Lie algebra $\g/\h$, and that
$\g/\h$ is isomorphic to the Lie algebra of left-invariant basic
vector fields on $G_\bu$.  Part~\eqref{item;left-invariant-commute} of
this result is a special case of
Lerman~\cite[Theorem~1.1]{lerman;vector-lie-2-group}, which says that
the Lie $2$-algebra of left-invariant multiplicative vector fields on
a strict Lie $2$-group $G_\bu$ is isomorphic to the Lie $2$-algebra
$\Lie(G_1)\rightrightarrows \Lie(G_0)$.

\begin{lemma}\label{lemma;liealg}
Let $G_\bu$ be a foliation $2$-group with crossed module
$(G,H,\partial,\alpha)$, where $G_0=G$ and $G_1=H\rtimes G$.  Let
$\pi\colon\g\to\g/\h$ be the quotient map.
\begin{enumerate}
\item\label{item;morita-algebra}
The map $\pi_\bu\colon \g_\bu \to \g/\h$, defined by
\[
\begin{tikzcd}[sep=large]
\g_1\ar[r,"\pi\circ\Lie(t)"]\ar[d,shift left]\ar[d,shift right]&
\g/\h\ar[d,shift left]\ar[d,shift right]
\\
\g_0\ar[r,"\pi"]&\g/\h
\end{tikzcd}
\]
is a Morita morphism of Lie $2$-algebras.
\item\label{item;morita-isomorphism}
Let $G_\bu'$ be a second foliation $2$-group with crossed module
$(G',H',\partial',\alpha')$. A Morita morphism of Lie $2$-groups
$\phi_\bu\colon G_\bu\to G_\bu'$ induces a Lie algebra isomorphism
$\g/\h\to\g'/\h'$.
\item\label{item;left-invariant}
The map $f\colon\g\to\Vect(G_0)\times\Vect(G_1)$ defined by
$f(\xi)=(\xi_L,(\Lie(u)\xi)_L)$ descends to a Lie algebra embedding
$\g/\h\inj\Vect_\bas(G_\bu)$, whose image is equal to
$\Vect_\bas(G_\bu)_L$.
\item\label{item;left-invariant-commute}
Let $\phi_\bu\colon G_\bu\to G'_\bu$ be a Morita morphism of Lie
$2$-groups.  Then the diagram
\[
\begin{tikzcd}
\g/\h\ar[r]\ar[d,"\cong"]&\Vect_\bas(G_\bu)\\
\g'/\h'\ar[r]&\Vect_\bas(G'_\bu)\ar[u,"\cong"']
\end{tikzcd}
\]
commutes.  Here the horizontal arrows are
from~\eqref{item;left-invariant}, the left vertical arrow is
from~\eqref{item;morita-isomorphism}, and the right vertical arrow is
from Proposition%
~\ref{proposition;vector-form-morita}\eqref{item;vector-morita}.
\end{enumerate}
\end{lemma}

\begin{proof}
For simplicity, we write $s=\Lie(s)$ for the source map of $\g_\bu$,
and similarly for the other structure maps.

\eqref{item;morita-algebra}~Since $\pi$ and $t$ are Lie algebra
homomorphisms, to show that $\pi_\bu$ is a homomorphism of Lie
$2$-algebras it suffices to show $t(\xi)-s(\xi)\in t(\ker s)=\h$ for
all $\xi\in\g_1$.  Indeed, $\xi-u(s(\xi))\in\ker s$ and
\[(t-s)(\xi)=(t-s)(\xi-u(s(\xi)))=t(\xi-u(s(\xi))).\]
Next we show that $\pi_\bu$ is a Morita morphism.  Essential
surjectivity is automatic since $\pi$ is a surjective linear map.  For
full faithfullness it is enough to show that the canonical map
$(s,t)\colon \g_1\to\g_0\times_{\g_0/\h}\g_0$ is a linear isomorphism.
To show the map is injective assume $s(\xi)=t(\xi)=0$.  Then
$\xi\in\ker(\partial\colon\h\to\g_0),$ which is $0$ since $G_\bu$ is
assumed to be a foliation groupoid.  So $\xi=0$ and the map is
injective.  Surjectivity follows from counting dimensions, as in the
proof of Lemma~\ref{lemma;moritaequivliealg}.

\eqref{item;morita-isomorphism}~By
Lemma~\ref{lemma;moritaequivliealg}\eqref{item;morita-lie-algebra},
applying the $\Lie$ functor to $\phi_\bu$ gives a Morita morphism
\[
\Lie(\phi_\bu) \colon \g_\bu \to \g'_\bu
\] 
of Lie $2$-algebras. Composing with the projection $\pi'_\bu\colon
\g'_\bu \to \g'/\h'$, by part~\eqref{item;morita-algebra} we have a
Morita morphism of Lie $2$-algebras $\g_\bu \to \g'/\h'$. We also have
the Morita morphism $\g_\bu \to \g'/\h'$. By
Lemma~\ref{lemma;moritaequivliealg}\eqref{item;morita-lie-lie},
$\pi'_\bu \circ \Lie(\phi_\bu)$ descends to a Lie algebra isomorphism
$ \g/\h\cong \g'/\h'$.

\eqref{item;left-invariant}~For $\xi\in\g$ we have
$s(u(\xi)_L)=t(u(\xi)_L)=\xi_L$, so $u(\xi)_L\sim_s\xi_L$ and
$u(\xi)_L\sim_t\xi_L$.  The leaves of the foliation $\F$ of $G$ are
the connected components of the cosets of the immersed normal subgroup
$\partial(H)$.  Therefore the flow of $\xi_L$ maps leaves to leaves,
and $\xi_L$ is in the normalizer of the subalgebra $\Gamma(T\F)$ of
$\Vect(G)$.  This shows that $f(\xi)$ defines a left-invariant element
$\hat{f}(\xi)\in\Vect_\bas(G_\bu)_L$.  We have $\hat{f}(\xi)=0$ if and
only if $\xi_L$ is tangent to the leaves of $\F$ if and only if
$\xi\in\h$.  So $\hat{f}\colon\g\to\Vect_\bas(G_\bu)$ induces an
isomorphism $\g/\h\cong\Vect_\bas(G_\bu)_L$.

\eqref{item;left-invariant-commute}~Since the basic vector fields in
the image of the embeddings from~\eqref{item;left-invariant} are left
invariant, they are uniquely determined by their value at the group
unit~$1$.  Write $N_0=N_0(G_\bu)$ and $N_0'=N_0(G_\bu')$ for the
normal bundles of $\F_0(G_\bu)$ and $\F_0(G'_0)$, respectively.  Then
$N_0|_1=\g/\h$ and $N_0'|_1=\g'/\h'$, and the pullback of $N_0'|_1$ to
$N_0|_1$ is just the inverse of the isomorphism $\g/\h\cong\g'/\h'$,
and the assertion follows.
\end{proof}

In the remainder of this section we will typically identify the Lie
$2$-algebra $\Lie(G_\bu)$ of a foliation $2$-group with the quotient
Lie algebra $\g/\h$.  We will without further mention apply
Lemma~\ref{lemma;morita-crossed}, which allows us to define Morita
morphisms of $2$-groups in terms of their crossed modules.  We will
also use the restriction, extension, and quotient Morita morphisms of
Remark~\ref{remark;morita}.

\begin{lemma}\label{lemma;etale}
Every foliation $2$-group $G_\bu$ is Morita equivalent to an \'etale
$2$-group $G_\bu'$ with the property that the identity component of
$G_0'$ is simply connected.
\end{lemma}

\begin{proof}
Let $(G,H,\partial,\alpha)$ be the crossed module of $G_\bu$, then
$\ker(\partial)$ is discrete.  We must show that
$(G,H,\partial,\alpha)$ is Morita equivalent to a crossed module
$(G',H',\partial',\alpha')$ where the identity component of $G'$ is
simply connected $H'$ is discrete.  Let $\phi\colon\hat{G}\to G$ be a
surjective \'etale homomorphism, where $\hat{G}$ is simply connected
(this exists even if $G$ is not connected;
see~\cite[Corollary~5.6]{brown-mucuk;covering-groups}). Let
$\hat{H}=H\times_G\hat{G}$ be the pullback cover.  The extension
$(\hat{G},\hat{H},\hat\partial,\hat\alpha)\to(G,H,\partial,\alpha)$ is
a Morita morphism.  Let $N$ be the identity component of
$\hat\partial(\hat{H})$.  Then $N$ is the connected immersed normal
subgroup of $\hat{G}$ with Lie algebra $\h$.  It follows
from~\cite[Proposition~III.6.14]{bourbaki;groupes-algebres} that $N$
is closed and the identity component of the quotient $G'=\hat{G}/N$ is
simply connected.  So we can form the quotient crossed module
$(G',H',\partial',\alpha')$ with $H'=\hat{H}/\hat\partial^{-1}(N)$,
and we have a Morita morphism
$(\hat{G},\hat{H},\hat\partial,\hat\alpha)\to
(G',H',\partial',\alpha')$.  In conclusion we have defined a zigzag of
Morita morphisms
\[
\begin{tikzcd}
(G,H,\partial,\alpha)&(\hat{G},\hat{H},\hat\partial,\hat\alpha)
  \ar[l]\ar[r]&(G',H',\partial',\alpha')
\end{tikzcd}
\]
with the identity component of $G'$ simply connected and $H'$
discrete.
\end{proof}

For any foliation $2$-group $G_\bu$, the quotient
$G_0/\bar{G}_1=G\big/\overline{\partial(H)}$ (i.e.\ the largest
Hausdorff quotient of the coarse quotient group $G_0/G_1$) is a Lie
group, which is a Morita invariant of $G_\bu$.  Compactness properties
of this Lie group translate into compactness properties of the Lie
algebra of $G_\bu$.

\begin{lemma}\label{lemma;etale-compact}
Let $G_\bu$ be an \'etale $2$-group.  If the Lie group $G_0/\bar{G}_1$
is compact, then the Lie algebra of $G_\bu$ is compact.
\end{lemma}

\begin{proof}
Let $(G,H,\partial,\alpha)$ be the crossed module of $G_\bu$.  We may
assume without loss of generality that $G$ is connected.  Let $N$ be
the immersed subgroup $\partial(H)$ of $G$, let $\bar{N}$ be its
closure, and let $\nnn=\Lie(N)$ and $\bar\nnn=\Lie(\bar{N})$.  Since
$N$ is $0$-dimensional and normal, it is central in $G$, so $\bar{N}$
is central in $G$, so $\bar\nnn$ is a central subalgebra of $\g$.
Since $K=G/\bar{N}$ is compact Lie and its Lie algebra is
$\g/\bar\nnn$, the central extension of Lie algebras
$
\begin{tikzcd}[cramped,sep=small]
\bar\nnn\ar[r,hook]&\g\ar[r,two heads]&\lie{k}
\end{tikzcd}
$
is split.  (To produce a splitting, start with an arbitrary linear
right inverse $s$ of the projection $\g\to\lie{k}$.  The adjoint
action of $G$ on $\g$ descends to an action of $K$, and the map
$\bar{s}=\int_K\Ad_k\circ s\circ\Ad_k^{-1}\,dk$, where $dk$ is
normalized Haar measure on $K$, is a Lie algebra splitting of $p$.)
This shows that $\g\cong\bar\nnn\oplus\lie{k}$ is compact.
\end{proof}

These lemmas can be made more precise when the coarse quotient
$G_0/G_1$ is connected.

\begin{definition}\label{definition;connected-compact}
Let $G_\bu$ be a foliation $2$-group.  We say $G_\bu$ is
\emph{base-connected} (resp. \emph{base-simply connected}) if $G_0$ is
connected (resp. if $G_0$ is simply connected).  We say $G_\bu$ is of
\emph{compact type} if the Lie algebra $\g_0$ is a compact Lie algebra
and the coarse quotient $G_0/G_1$ is a compact topological space.  We
say $G_\bu$ is \emph{base-compact} if $G_0$ is compact.
\end{definition}

\begin{lemma}\label{lemma;connected-etale}
Let $G_\bu$ be a foliation $2$-group.  The following conditions are
equivalent.
\begin{enumerate}
\item\label{item;connected}
$G_0/G_1$ is connected;
\item\label{item;closure-connected}
$G_0/\bar{G}_1$ is connected;
\item\label{item;base-connected}
the inclusion $G_\bu'\inj G_\bu$ is a Morita morphism, where $G_\bu'$
is the full subgroupoid obtained by restricting $G_\bu$ to the
identity component $G_0'$ of $G_0$;
%
\item\label{item;base-1-connected-etale}
$G_\bu$ is Morita equivalent to a base-simply connected \'etale
  $2$-group.
\end{enumerate}
\end{lemma}

\begin{proof}
Let $(G,H,\partial,\alpha)$ be the crossed module of $G_\bu$.  The
equivalence~\eqref{item;connected}$\iff$\eqref{item;closure-connected}
is straightforward.

\eqref{item;connected}$\implies$\eqref{item;base-connected}:~If
$G/\partial(H)$ is connected, then $G=\partial(H)G'$, where $G'$ is
the identity component of $G$.  Therefore the inclusion
$(G',H',\partial',\alpha')\to(G,H,\partial,\alpha)$ is a Morita
morphism, where $(G',H',\partial',\alpha')$ is the restricted crossed
module with $H'=\partial^{-1}(H\cap G')=H\times_GG'$.

\eqref{item;base-connected}$\implies$%
\eqref{item;base-1-connected-etale}:~Given the Morita morphism of
crossed modules $(G',H',\partial',\alpha')\inj (G,H,\partial,\alpha)$,
we let $\tilde{G}$ be the universal cover of $G'$, and as in the proof
of Lemma~\ref{lemma;etale} we obtain a base-simply connected crossed
module $(\tilde{G},\tilde{H},\tilde\partial,\tilde\alpha)$ and a
Morita morphism $(\tilde{G},\tilde{H},\tilde\partial,\tilde\alpha)\to
(G',H',\partial',\alpha')$.  Next we let $N$ the identity component of
$\tilde\partial(\tilde{H})$, we form the quotient crossed module
$(G'',H'',\partial'',\alpha'')$ with $G''=\tilde{G}/N$, and we obtain
a zigzag of Morita morphisms
\[
\begin{tikzcd}
(G,H,\partial,\alpha)&(G',H',\partial',\alpha')\ar[l,hook']&
  (\tilde{G},\tilde{H},\tilde\partial,\tilde\alpha)
  \ar[l]\ar[r]&(G'',H'',\partial'',\alpha''),
\end{tikzcd}
\]
where $G''$ is simply connected and $H''$ is discrete.

\eqref{item;base-1-connected-etale}$\implies%
$\eqref{item;connected}:~If $G_\bu\simeq G_\bu''$ with $G_\bu''$
base-connected, then $G/\partial(H)=G''/\partial''(H'')$ is connected.
\end{proof}

Under additional compactness assumptions we have the following result,
which is a $2$-group analogue of the fact that every compact connected
Lie group is up to a covering isomorphic to a product of a torus and a
simply connected compact Lie group.

\begin{proposition}\label{proposition;compact-connected}
Let $G_\bu$ be a foliation $2$-group.  The following conditions are
equivalent.
\begin{enumerate}
\item\label{item;compact-connected}
$G_0/G_1$ is compact and connected;
\item\label{item;closure-compact-connected}
$G_0/\bar{G}_1$ is compact and connected;
\item\label{item;base-compact-connected}
%
$G_\bu$ is Morita equivalent to a base-simply connected \'etale
  $2$-group of compact type;
\item\label{item;base-compact}
$G_\bu$ is Morita equivalent to a base-compact and base-connected
  \'etale $2$-group.
\end{enumerate}
\end{proposition}

\begin{proof}
Let $(G,H,\partial,\alpha)$ be the crossed module of $G_\bu$.  The
implication~\eqref{item;compact-connected}$\implies$%
\eqref{item;closure-compact-connected} is obvious.

\eqref{item;closure-compact-connected}$\implies%
$\eqref{item;base-compact-connected}:~Suppose
$G\big/\overline{\partial(H)}$ is compact and connected.  By
Lemma~\ref{lemma;connected-etale} we may assume without loss of
generality that $G_\bu$ is base-simply connected and \'etale.  It then
follows from Lemma~\ref{lemma;etale-compact} that the Lie algebra $\g$
is compact.

\eqref{item;base-compact-connected}$\implies%
$\eqref{item;base-compact}:~Suppose that $G_\bu$ is base-simply
connected, \'etale, and of compact type.  Then $G$ is isomorphic to
the product $E\times K$ of a vector group $E$ and a simply connected
compact Lie group $K$.  Let $\pr_E\colon G\to E$ be the projection and
let $N_E=\pr_E(N)$, where $N=\partial(H)$.  Then $\pr_E$ induces a
surjection $G/N\to E/N_E$, so $E/N_E$ is compact, so $N_E$ contains a
basis $e_1$, $e_2$,~\dots, $e_k$ of $E$.  Choose $n_i\in N$ with
$\pr_E(n_i)=e_i$; then the subgroup $L$ of $N$ generated by the $n_i$
is a discrete cocompact normal subgroup of $G$ isomorphic to $\Z^k$.
Then $G'=G/L$ is compact and connected, and we can form the quotient
crossed module $(G',H'=H/\partial^{-1}(L),\partial',\alpha')$, which
is Morita equivalent to $(G,H,\partial,\alpha)$.  The associated
$2$-group $G_\bu'$ is Morita equivalent to $G_\bu$ and is
base-compact, base-connected, and \'etale.

\eqref{item;base-compact}$\implies$\eqref{item;compact-connected}:~If
$G_\bu\simeq G_\bu'$ with $G_\bu'$ base-compact and base-connected,
then $G/\partial(H)=G'/\partial'(H')$ is compact and connected.
\end{proof}

\begin{example}\label{example7}
The Lie $2$-group $G_\bu$ of Example~\ref{example4} is a foliation
$2$-group of compact type.  Its Lie $2$-algebra
$\lie{n}\ltimes\g\rightrightarrows\g$ (see Example~\ref{example5}) is
Morita equivalent to the abelian Lie algebra $\g/\lie{n}$.
\end{example}

\subsection{Lie group stacks}\label{subsection;group-stack}

A \emph{weak Lie group stack} is weak $2$-group in the $2$-category
$\Diffstack$, as in Definition~\ref{definition;internalgp}.  A
\emph{weak homomorphism} of weak Lie group stacks is a weak
homomorphism of $2$-groups in $\Diffstack$, as in
Definition~\ref{definition;hominternal2group}.  An \emph{equivalence}
of weak Lie group stacks is an equivalence of $2$-groups in
$\Diffstack$, as in Definition~\ref{definition;equivalence}.  A
\emph{strict Lie group stack} is a strict $2$-group in $\Diffstack$.
We will usually abbreviate \emph{strict Lie group stack} to \emph{Lie
  group stack}.
\glossary{Gx@$\G$, Lie group stack}%

By Lemma~\ref{lemma;preservelimits}, the $2$-functor $\B\colon
\Liegpd\to \Diffstack$ preserves weak pullbacks.  So $\B$ preserves
takes weak (resp.\ strict) Lie $2$-groups to weak (resp.\ strict) Lie
group stacks.  Note that equivalent weak Lie group stacks are
equivalent as differentiable stacks.  A weak homomorphism of Lie
groups, considered as Lie group stacks, is the same as a homomorphism
of Lie groups.

\begin{definition}
\label{definition;presentationLiegroupstack}
Let $\G$ be a Lie group stack.  A \emph{presentation} of $\G$ is an
equivalence of Lie group stacks $\B G_\bu\to\G$ $G_\bu$, where $G_\bu$
is a Lie $2$-group.
\end{definition}

Given a stack $\X$, there is an underlying topological space called
the \emph{coarse moduli space}; see for
instance~\cite[\S\,3]{noohi;foundations-topological-stacks}
or~\cite{watts-wolbert;diffeological-coarse} for a definition.  For
our purposes, it is enough to know that, given a presentation $\B
X_\bu \simeq \X$ of a differentiable stack, the coarse moduli space is
homeomorphic to the coarse quotient $X_0/X_1$.

\begin{definition}
\label{definition;compactLiegroupstacks}
A weak Lie group stack $\G$ is \emph{compact},
resp.\ \emph{connected}, if the coarse moduli space is compact,
resp.\ connected.
\end{definition}

Our focus on the strict case is justified by the following
strictification theorem, which says that a weak Lie group stack $\G$
that is connected and \'etale is equivalent to a strict Lie group
stack and has a particularly nice atlas.

\begin{theorem}%
[Trentinaglia and
  Zhu~{\cite[Theorem~5.13]{trentinaglia-zhu;strictification}}]
\label{theorem;lie-group-stack}
Let $\G$ be a weak Lie group stack.  Then the following properties are
equivalent.
\begin{enumerate}
\item\label{item;weak-lie-stack}
$\G$ is connected and \'etale;
\item\label{item;lie-stack}
$\G$ is equivalent to a connected \'etale strict Lie group stack;
\item\label{item;lie-equivalent}
$\G$ is equivalent to the classifying stack $\B G_\bu$ of a
  base-simply connected \'etale $2$-group~$G_\bu$;
\item\label{item;lie-atlas}
there exist a simply connected Lie group $G$ and an \'etale atlas
$\bpsi\colon G\to\G$ which is a weak homomorphism.
\end{enumerate}
\end{theorem}

We will often use the following case of the strictification theorem.
See Definition~\ref{definition;connected-compact} for the notions
base-connected and base-compact.

\begin{corollary}\label{corollary;compact-lie-group-stack}
The following conditions on a weak Lie group stack $\G$ are
equivalent: \emph{(1)}~$\G$ is compact, connected, and \'etale;
\emph{(2)}~there exists a presentation $\B G_\bu\simeq\G$, where
$G_\bu$ is a base-compact, base-connected, \'etale $2$-group;
\emph{(3)}~there exist a compact connected Lie group $G$ and an
\'etale atlas $\bpsi\colon G\to\G$ which is a weak homomorphism.
\end{corollary}

\begin{proof}
Combine Theorem~\ref{theorem;lie-group-stack} with
Proposition~\ref{proposition;compact-connected}, using the fact that
the coarse moduli space of $\G$ is homeomorphic to the coarse quotient
of a presenting $2$-group.
\end{proof}

\subsection{Presentations of equivalent Lie group stacks}
\label{subsection;presentation-group-stack}

The following result states that the fibred product of two strict Lie
group stacks (if it exists as a differentiable stack) is a weak Lie
group stack.

\begin{theorem}\label{theorem;pullbackgroups}
Let $\G\to\HH$ and $\G'\to\HH$ be weak homomorphisms of (strict) Lie
group stacks, and assume that the fibred product of stacks
$\KK=\G\times_{\HH}\G'$ is a differentiable stack.  Then $\KK$ is
naturally a weak Lie group stack, and the projections $\KK\to\G$ and
$\KK\to\G'$ are strict homomorphisms.
\end{theorem}

\begin{proof} See Appendix~\ref{appendix;zhu}.
\end{proof}

We will deduce from this that if the classifying stacks of two Lie
$2$-groups are equivalent as Lie group stacks, then the two Lie
$2$-groups are Morita equivalent as Lie $2$-groups.

\begin{proposition}\label{proposition;strictmor}
Let $\G$ be a (strict) Lie group stack. Assume there are two
presentations
\[
\begin{tikzcd}
\bpsi\colon\B G_\bu\ar[r,"\simeq"]&\G
\end{tikzcd}, 
\qquad
\begin{tikzcd}
\bpsi'\colon\B G_\bu'\ar[r,"\simeq"]&\G
\end{tikzcd},
\]
where $G_\bu$ and $G_\bu'$ are Lie $2$-groups.  Let $K_\bu$ be the Lie
groupoid defined by
\[
K_0=G_0\times_\G G_0'\qquad\text{and}\qquad K_1=K_0\times_\G K_0\cong
G_1\times_{G_0}K_0\times_{G_0'}G_1'.
\]
Then
\begin{enumerate}
\item\label{item;morita-lie}
$K_\bu$ is naturally a Lie $2$-group. The maps $\bpsi$ and $\bpsi'$
  induce Morita morphisms of Lie $2$-groups $K_\bu\to G_\bu$ and
  $K_\bu\to G_\bu'$.  The maps $K_0\to G_0$ and $K_0\to G_0'$ are
  surjective submersions.  
%
\item\label{item;morita-lie-tangent}
$\bpsi$ and $\bpsi'$ induce Morita morphisms of Lie $2$-algebras
  $\Lie(K_\bu)\to\Lie(G_\bu)$ and $\Lie(K_\bu)\to\Lie(G_\bu')$.
\item\label{item;morita-lie-compact}
If $\G$ is compact, connected, and \'etale, and if $G_\bu$ and
$G_\bu'$ are of compact type, then $K_\bu$ is of compact type.
\end{enumerate}
\end{proposition}

\begin{proof}
\eqref{item;morita-lie}~Composing $\bpsi$ and $\bpsi'$ with the
quotient maps $G_0\to\B G_\bu$ and $G_0'\to\B G_\bu'$ we get two
atlases $\bpsi_0\colon G_0\to\G$ and $\bpsi_0'\colon G_0'\to\G'$, both
of which are weak homomorphisms.  Because atlases are surjective
representable submersions, the fibred product $K_0=G_0\times_\G G_0'$
is (equivalent to) a manifold, and the projections $K_0\to G_0$ and
$K_0\to G_0'$ are surjective submersions.  On the other hand, by
Theorem~\ref{theorem;pullbackgroups}, $K_0$ is also a weak Lie group
stack.  It follows that $K_0$ is a Lie group.  Moreover, the
composition $K_0\to G_0\to\G$ is an atlas which is a weak
homomorphism.  Repeating the argument we see that $K_1$ is likewise a
Lie group, and that the Lie groupoid $K_\bu$ is a Lie $2$-group
equipped with two Morita morphisms of $2$-groups $K_\bu\to G_\bu$ and
$K_\bu\to G_\bu'$.

\eqref{item;morita-lie-tangent}~This follows
from~\eqref{item;morita-lie} and Lemma~\ref{lemma;moritaequivliealg}.

\eqref{item;morita-lie-compact}~It follows
from~\eqref{item;morita-lie} that $K_0/K_1\cong G_0/G_1$ is compact.
Let $G$ be a simply connected Lie group and $\bchi\colon G\to\G$ an
\'etale atlas which is a weak homomorphism as in the strictification
theorem, Theorem~\ref{theorem;lie-group-stack}\eqref{item;lie-atlas}.
It follows from~\eqref{item;morita-lie-tangent} and from
Lemma~\ref{lemma;liealg} that $\bpsi$ and $\bpsi'$ induce isomorphisms
of Lie algebras $\g_0/\h\cong\Lie(G)$ and $\g_0'/\h'\cong\Lie(\G)$.
Therefore the Lie algebra $\kk_0':=\g_0\times_{\Lie(G)}\g_0'$ is
compact.  We claim that $\kk_0'$ is isomorphic to $\kk_0=\Lie(K_0)$.
Let $\tilde{G}_0=G_0\times_\G G$ and $\tilde{G}_0'=G_0'\times_\G G$.
Consider the cube
\[
\begin{tikzcd}
&\tilde{K}_0\arrow[dl]\arrow[rr]\arrow[dd]&&
  \tilde{G}_0'\arrow[dl]\arrow[dd]\\
\tilde{G}_0\arrow[rr,crossing over]\arrow[dd]&&G\\
&K_0\arrow[dl]\arrow[rr]&&G_0'\arrow[dl,"\bpsi_0'"]\\
G_0\arrow[rr,"\bpsi_0"]&&\G\arrow[from=uu,crossing over,"\bchi"near
  start]\\
\end{tikzcd}
\]
where $\tilde{K}_0$ is the weak limit of the three squares which
contain $\G$.  Since these squares are all weak pullbacks, every face
of the cube is a weak pullback of stacks and we have
$\tilde{K}_0\cong\tilde{G}_0\times_G\tilde{G}_0'$.  Since $\bchi\colon
G\to\G$ is \'etale, the vertical maps are all \'etale.  In particular,
$\kk_0=\Lie(K_0)\cong\Lie(\tilde{K}_0)$.  We then have
\[
\kk_0\cong\Lie(\tilde{K}_0)\cong
\Lie(\tilde{G}_0\times_G\tilde{G}_0')\cong
\Lie(\tilde{G}_0)\times_{\Lie(G)}\Lie(\tilde{G}_0')\cong
\g_0\times_\g\g_0'=\kk_0',
\]
as was claimed.
\end{proof}

\begin{remark}\label{remark;morita-lie-base-connected}
Even if the $2$-groups $G_\bu$ and $G'_\bu$ in
part~\eqref{item;morita-lie-compact} are base-connected, the fibred
product $K_\bu$ is not necessarily base-connected.  However, by
Lemma~\ref{lemma;connected-etale}\eqref{item;base-connected} $K_\bu$
has a Morita equivalent full subgroupoid $K_\bu'$ which is a
base-connected sub-Lie $2$-group, and the maps $K_0'\to G_0$ and
$K_0'\to G_0'$ are still surjective submersions.
\end{remark}

\subsection{The Lie algebra of an \'etale Lie group stack}
\label{subsection;lie-algebra}

We define the \emph{Lie algebra} of a connected \'etale weak Lie group
stack $\G$ to be the Lie subalgebra $\Lie(\G)$ of $\Vect(\G)$
characterized by the next proposition.  Here $\Vect(\G)$ is the Lie
algebra of vector fields of $\G$ described in
\S\,\ref{subsection;vector-etale}.
\glossary{Lie@$\Lie$, Lie functor}%

\begin{proposition}\label{proposition;LiealgLiegroupstack}
Let $\G$ be a connected \'etale weak Lie group stack.  There is a
unique Lie subalgebra $\Lie(\G)$ of $\Vect(\G)$ so that for every
presentation $\B G_\bu\simeq\G$ by a strict Lie $2$-group the image of
the Lie algebra embedding
\begin{equation}\label{equation;Liealgembedding}
\begin{tikzcd}
\g/\h\ar[r,hookrightarrow]&\Vect_\bas(G_\bu)\ar[r,"\cong"]&\Vect(\G)
\end{tikzcd}
\end{equation}
is $\Lie(\G)$.  The first map of~\eqref{equation;Liealgembedding} is
from Lemma~\ref{lemma;liealg}\eqref{item;left-invariant}, and the
second map is~\eqref{equation;bas-biject}.
\end{proposition}

\begin{proof}
Let $\B G_\bu\simeq \G$ and $\B G_\bu'\simeq \G$ be two presentations
of $\G$; such presentations exist by
Theorem~\ref{theorem;lie-group-stack}.  By
Proposition~\ref{proposition;strictmor}, the Lie $2$-groups $G_\bu$
and $G'_\bu$ are Morita equivalent. Without loss of generality we may
assume there is a Morita morphism $\phi_\bu \colon G_\bu \to
G_\bu'$. Then the result follows from naturality
of~\eqref{equation;bas-biject} with respect to Morita morphisms as
well as Lemma~\ref{lemma;liealg}\eqref{item;left-invariant-commute}.
\end{proof}
 
Given a presentation $\B G_\bu\simeq \G$ by a Lie $2$-group, we have
canonical isomorphisms
\[\Lie(\G)\cong\g/\h\cong\Vect_\bas(G_\bu)_L.\] 
If the groupoid $G_\bu$ is \'etale, then $\Lie(\G)\cong\g$.


%

Let
$\begin{tikzcd}[cramped,sep=small]
\bphi\colon\G\ar[r,"\simeq"]&\G'
\end{tikzcd}$
be an equivalence of connected \'etale weak Lie group stacks. This
induces an isomorphism of Lie algebras $\Vect(\G) \cong \Vect(\G')$ by
Proposition~\ref{proposition;vector-stack}. By
Lemma~\ref{lemma;liealg}\eqref{item;left-invariant-commute}, this
restricts to an isomorphism of Lie algebras, denoted by
\begin{equation}\label{equation;lie-stack}
\begin{tikzcd}
\Lie(\bphi)\colon\Lie(\G)\ar[r,"\cong"]&\Lie(\G').
\end{tikzcd}
\end{equation}
%
%
%
%

\subsection{Actions of Lie group stacks}\label{subsection;action-stack}

A \emph{weak action} of a (strict) Lie group stack $\G$ on a
differentiable stack $\X$ is as in
Definition~\ref{definition;actioninternal}, by considering $\G$ to be
a strict $2$-group in $\Diffstack$.  A \emph{strict action} is also as
in Definition~\ref{definition;actioninternal}.  A \emph{weakly
  equivariant morphism} $\X\to\X'$ of stacks with $\G$-actions is as
in Definition~\ref{definition;equivariantinternal}.

The $2$-functor $\B\colon\Liegpd\to\Diffstack$ preserves products
(Lemma~\ref{lemma;preservelimits}), and so it takes (strict) actions
of Lie $2$-groups to (strict) actions of Lie group stacks.


Actions of connected \'etale Lie group stacks, like the stacks
themselves, can be strictified in the sense of the following
statement, which
generalizes~\cite[Proposition~3.2]{lerman-malkin;deligne-mumford}.

\begin{theorem}\label{theorem;strictaction}
Let $\G$ be a connected (strict) Lie group stack acting weakly on an
\'etale differentiable stack $\X$.  Suppose that $\G$ admits a
presentation $\B G_\bu\simeq\G$ by a base-connected Lie $2$-group
$G_\bu$.  For every such presentation $\B G_\bu\simeq\G$ there exists
a presentation $\B X_\bu\simeq\X$ of $\X$ by a Lie groupoid $X_\bu$ so
that
\begin{enumerate}
\item
$G_\bu$ acts strictly on $X_\bu$;
\item
identifying $\B G_\bu=\G$, the equivalence $\B X_\bu\simeq\X$ is
weakly $\G$-equivariant.
\end{enumerate}
\end{theorem}

\begin{proof}
See Appendix~\ref{appendix;strict}.
\end{proof}

Theorem~\ref{theorem;lie-group-stack} and
Theorem~\ref{theorem;strictaction} together justify our focus on
strict actions of foliation Lie $2$-groups on foliation groupoids.

\subsection{Fundamental vector fields}\label{subsection;fundamental}

Let $\G$ be a connected strict Lie group stack, $\X$ an \'etale stack,
and $\ba\colon\G\times\X\to\X$ a weak action.  Recall that
$\Bvect(\X)$ denotes the groupoid of vector fields on $\X$ defined in
\S\,\ref{subsection;vector-etale}, and that the set $\Vect(\X)$ of
equivalence classes of $\Bvect(\X)$ carries a natural Lie bracket
(Proposition~\ref{proposition;vector-stack}).  Let $\bxi$ be an
element of the Lie algebra $\Lie(\G)$ of $\G$ (as defined in
\S\,\ref{subsection;lie-algebra}), regard $\bxi$ as an element of the
Lie algebra $\Vect(\G)$, and let $\tilde{\bxi}\in \Bvect(\G)$ be a
lift of $\bxi$ to $\Bvect(\G)$.  The composition of the morphisms
\begin{equation}\label{eqn;funvec}
\begin{tikzcd}
\X\simeq\star\times\X\ar[r,"\bunit\times{\id}"]&
\G\times\X\ar[r,"\tilde{\bxi}\times\bzero"]&T\G\times T\X\simeq
T(\G\times\X)\ar[r,"T\ba"]&T\X
\end{tikzcd}
\end{equation}
defines an object of the groupoid $\Bvect(\X)$.  Here the equivalence
$T\G\times T\X\simeq T(\G\times\X)$ follows
from~\cite[Example~4.6]{hepworth;vector-flow-stack}, and $T\ba$ is the
tangent morphism of the action $\ba$ as defined
in~\cite[\S\,3.1]{hepworth;vector-flow-stack}.  The isomorphism class
$\bxi_\X\in\Vect(\X)$ of this object does not depend on the choice of
lift $\tilde{\bxi}$, and we call $\bxi_\X$ the \emph{fundamental
  vector field} of $\bxi$ associated with the action.
\glossary{xixx@$\bxi_\X$, fundamental vector field on $\G$-stack $\X$
  induced by $\bxi\in\Lie(\G)$}%

The following describes the fundamental vector field in terms of an
atlas.

\begin{proposition}\label{proposition;funvec} 
Let a foliation Lie $2$-group $G_\bu$ act strictly on a foliation
groupoid $X_\bu$.  The map
\[\g\longto\Vect(X_0)\times\Vect(X_1)\]
defined by $\xi\mapsto(\xi_{X_0},(\Lie(u)\xi)_{X_1})$ descends to a
Lie algebra anti-homomorphism
\begin{equation}\label{eqn;funcoords}
\g/\h\cong \Vect_\bas(G_\bu)_L \longto\Vect_\bas(X_\bu).
\end{equation}
The diagram
\[
\begin{tikzcd}
\g/\h\ar[r]\ar[d,"\cong"']&\Vect_\bas(X_\bu)\ar[d,"\cong"]\\
\Lie(\B G_\bu)\arrow[r]&\Vect(\B X_\bu)
\end{tikzcd}
\]
commutes, where the bottom arrow is the assignment
$\stackmorphism{\xi}\mapsto\stackmorphism{\xi}_{\B X_\bu}$ described
in~\eqref{eqn;funvec}, the left arrow is from
Proposition~\ref{proposition;LiealgLiegroupstack}, and the right arrow
is~\eqref{equation;bas-biject}.
%
%
%
\end{proposition}

\begin{proof}
Let us show that the map~\eqref{eqn;funcoords} is well defined.  That
it is an anti-homo\-morphism then follows from the fact that the map
$\g\to\Vect(X_0)\times\Vect(X_1)$ is an anti-homomorphism, and
commutivity of the diagram follows
from~\eqref{equation;tanstackequiv}.  Let $\xi\in\g$.  Then
$\Lie(s)(\Lie(u)\xi)=\xi$.  Hence, since the source map is equivariant
with respect to the action, the vector field $(\Lie(u)\xi)_{X_1}$ is
$s$-related to $\xi_{X_0}$.  Similarly, $(\Lie(u)\xi)_{X_1}$ is
$t$-related to $\xi_{X_0}$, so the pair
$(\xi_{X_0},(\Lie(u)\xi)_{X_1})$ descends to an element of
$\Vect_\bas(X_\bu)$.  Now let $\xi\in\h$.  The vector field
$(\xi_L,0)\in\Vect(G_0)\times\Vect(X_0)$ is tangent to the foliation
of $G_0\times X_0$ by $G_\bu\times X_\bu$-orbits, and so its
pushforward by the derivative of the action map is tangent to the
foliation $\F_0(X_\bu)$ of $X_0$.  So $\xi_{X_0}$ descends to the zero
section of the normal bundle $N_0(X_\bu)$, which shows that
$(\xi_{X_0},(\Lie(u)\xi)_{X_1})$ is the zero basic vector field.
\end{proof}

\begin{definition}
\label{definition;fundamentalbasic}
Let a foliation $2$-group $G_\bu$ act strictly on a Lie groupoid
$X_\bu$.  For $\xi\in\g/\h$, let $\xi_{X_\bu}\in \Vect_\bas(X_\bu)$
denote the image of $\xi$ under the map \eqref{eqn;funcoords}. The
basic vector field $\xi_{X_\bu}$ is the \emph{fundamental vector field
  of $\xi$}.
\glossary{xix*@$\xi_{X_\bu}$, fundamental vector field on groupoid
  $X_\bu$ induced by action of foliation $2$-group $G_\bu$ by
  $\xi\in\Lie(G_\bu)$}%
\end{definition}

\subsection{The (co)adjoint action}\label{subsection;adjoint}

Let $G_\bu$ be a Lie $2$-group.  The \emph{adjoint action} of $G_\bu$
is the Lie groupoid homomorphism
\[\Ad_\bu\colon G_\bu\times\g_\bu\to\g_\bu,\]
where, for $i=0$, $1$, $\Ad_i$ is the adjoint action of $G_i$ on its
Lie algebra $\g_i$.  The adjoint action is a strict $G_\bu$-action on
$\g_\bu$ in the sense of Definition \ref{definition;actioninternal}
since each map $\Ad_i$ is an action of a Lie group.
\glossary{Ad@$\Ad_\bu$, adjoint action of Lie $2$-group}%

Since the adjoint action is natural with respect to Lie group
homomorphisms, a homomorphism of Lie $2$-groups $G_\bu\to G'_\bu$
takes the adjoint action of $G_\bu$ on $\g_\bu$ to the adjoint action
of $G'_\bu$ on $\g'_\bu$.  In particular, Morita morphisms of Lie
$2$-groups intertwine the adjoint actions.

\begin{lemma}\label{lemma;descendadjoint}
Let $G_\bu$ be a foliation Lie $2$-group with crossed module
$(G,H,\partial,\alpha)$.  Then the adjoint action of $G_\bu$ on
$\g_\bu$ descends to a strict action of $G_\bu$ on $\g/\h$.
\end{lemma}

\begin{proof}
Recall Morita morphism of Lie $2$-algebras
$\pi_\bu\colon\g_\bu\to\g/\h$ of
Lemma~\ref{lemma;liealg}\eqref{item;morita-algebra}:
$\pi_0\colon\g\to\g/\h$ is the projection and
$\pi_1\colon\g_1\to\g/\h$ is $\pi_0$ composed with $\Lie(t)$.  The
kernel of $\pi_0$ is $\h$, and this is preserved by $\Ad_0$ since it
is an ideal of $\g$.  The kernel of $\pi_1$ is
$\ker(\Lie(s))+\ker(\Lie(t))$, and this is preserved by $\Ad_1$ since
it is an ideal of $\g_1$.  Therefore the action $\Ad_\bu$ descends
through the morphism~$\pi_\bu$.
\end{proof}

By abuse of language we will also call the induced action of $G_\bu$
on $\g/\h$ from Lemma \ref{lemma;descendadjoint} the \emph{adjoint
  action} and denote it by
\[\Ad_\bu\colon G_\bu\times\g/\h\to\g/\h.\]
Let $(\g/\h)^*$ be vector space dual to $\g/\h$ and let
$\langle{\cdot},{\cdot}\rangle\colon(\g/\h)^*\times\g/\h\to\R$ be the
dual pairing.  The \emph{coadjoint action}
\[\Ad_\bu^*\colon G_\bu\times(\g/\h)^*\to(\g/\h)^*\]
is the strict action given by the defining property
\[
\langle\Ad_i^*(g\n)(\lambda),\xi\rangle=\langle g,\Ad_i(g)(\xi)\rangle
\]
for $i=0$, $1$ and for all $g\in G_i$, $\lambda\in(\g/\h)^*$, and
$\xi\in\g/\h$.

The notions of adjoint and coadjoint action carry over to stacks as
follows.

\begin{proposition}\label{proposition;adjointactionstack}
Let $\G$ be a connected \'etale weak Lie group stack and let
$\Lie(\G)$ be its Lie algebra, as defined in
\S\,\ref{subsection;lie-algebra}.  There is a unique strict action
$\Bad$ of $\G$ on $\Lie(\G)$ so that for every presentation $\B
G_\bu\simeq\G$ by a Lie $2$-group the diagram
\[
\begin{tikzcd}
\B G_\bu\times\g/\h\ar[r,"\B\Ad_\bu"]\ar[d]&\g/\h\ar[d]\\
\G\times\Lie(\G)\ar[r,"\Bad"]&\Lie(\G),
\end{tikzcd}
\]
commutes, where left arrow is the product of the presentation $\B
G_\bu\simeq\G$ and the map~\eqref{equation;Liealgembedding}.
\end{proposition}

\begin{proof}
Given a presentation $\B G_\bu\simeq\G$, one can define the morphism
$\Bad$ so that the diagram commutes.  Since $\Lie(\G)$ (considered as
a differentiable stack) takes values in the $1$-category
$\group{Set}$, the action is automatically strict.  Uniqueness of
$\Bad$ then follows from naturality of the adjoint action $\Ad_\bu$
with respect to Morita equivalences.
\end{proof}

We call the strict action
\[\Bad\colon\G\times\Lie(\G)\to\Lie(\G)\]
of Proposition~\ref{proposition;adjointactionstack} the \emph{adjoint
  action} of the Lie group stack $\G$.
\glossary{Ad@$\Bad$, adjoint action of Lie group stack}%
Let $\Lie(\G)^*$ be the vector space dual to $\Lie(\G)$ and let
$\langle{\cdot},{\cdot}\rangle\colon\Lie(\G)^*\times\Lie(\G)\to\R$ be
the dual pairing.  Then the \emph{coadjoint action} of $\G$ is the
strict action
\[\Bad^*\colon\G\times\Lie(\G)^*\to\Lie(\G)^*\]
given by the defining property
\[
\langle\Bad^*\circ((\cdot)\n\times\id_{\Lie(\G)^*}),\id_{\Lie(\G)}\rangle=
\langle\id_{\Lie(\G)^*},\Bad\rangle.
\]
If $\bo{g}\colon\star\to\G$ is a categorical point of $\G$, the
isomorphism
\[
\begin{tikzcd}
\Lie(\G)\ar[r,"\simeq"]&\star\times\Lie(\G)
\ar[r,"\bo{g}\times\id_{\Lie(\G)}"]&
\G\times\Lie(\G)\ar[r,"\Bad"]&\Lie(\G)
\end{tikzcd}
\]
is denoted by $\Bad_{\bo{g}}$.  The equivalence
$\Bad_{\bo{g}}^*\colon\Lie(\G)^*\to\Lie(\G)^*$ is similarly defined.


\subsection{Stacky tori}\label{subsection;torus}

Stacky tori play an analogous role to that of compact tori in the
theory of compact Lie groups.  A \emph{$2$-torus} is a Lie $2$-group
which is Morita equivalent to a foliation $2$-group $G_\bu$ with the
property that $G_0$ is a torus.  A \emph{stacky torus} is an \'etale
Lie group stack equivalent to $\B G_\bu$, where $G_\bu$ is a
$2$-torus.

\begin{lemma}\label{lemma;abelian}
Suppose that $G_\bu$ is a foliation $2$-group and that $G_0$ is
connected and abelian.  Then $G_1$ is abelian and the action
$\alpha\colon G_0\to\Aut(\ker(s))$ is trivial.
\end{lemma}

\begin{proof}
Let $g\in G$ and $h\in H$.  Then
$\partial({}^gh)=g\partial(h)g^{-1}=\partial(h)$ because $G$ is
abelian.  Therefore $\partial({}^ghh^{-1})=1$,
i.e.\ ${}^ghh^{-1}\in\ker(\partial)$.  In other words, for each $h\in
H$ the map $f(g)={}^ghh^{-1}$ maps $G$ to $\ker(\partial)$.  But
$G_\bu$ is a foliation groupoid, so $\ker(\partial)$ is discrete, and
$G$ is connected, so $f$ is constant.  Thus
${}^ghh^{-1}={}^1hh^{-1}=hh^{-1}=1$, so ${}^gh=h$, i.e.\ the action
$\alpha\colon G\to\Aut(H)$ is trivial.  Hence for all $h$, $h'\in H$
we have $h'={}^{\partial(h)}h'=hh'h^{-1}$, i.e.\ $h'h=hh'$.  It
follows that $G_1=H\rtimes_\alpha G$ is abelian.
\end{proof}

In the setting of Lemma~\ref{lemma;abelian} we will often write the
crossed module $(G,H,\partial,\alpha)$ simply as $\partial\colon H\to
G$ or $H\to G$.

\begin{definition}
\label{definition;quasilattice}
A \emph{quasi-lattice} is a crossed module $\partial\colon A\to E$,
where $E$ is (the additive group of) a finite-dimensional real vector
space and $A$ is a countable discrete abelian group, and where the
image $\partial(A)$ is required to span $E$ as a vector space.
\end{definition}

Definition \ref{definition;quasilattice} generalizes a definition of
Prato~\cite{prato;non-rational-symplectic}, who assumes the map
$\partial\colon A\to E$ to be injective.  Our next result is a stacky
analogue of the familiar fact that a torus is isomorphic to the
quotient of its Lie algebra by the exponential lattice.  Recall
from~\cite[\S\,5]{trentinaglia-zhu;strictification} that the
fundamental group $\pi_1(\G)$ of a Lie group stack is the set of
equivalence classes of maps $\group{S}^1\to\G$ based at the identity
of $\G$, modulo homotopy.

\begin{proposition}\label{proposition;2-torus}
Let $\G$ be a stacky torus and let $G_\bu$ be a $2$-torus with
$\G\simeq\B G_\bu$.  The crossed module of $G_\bu$ is Morita
equivalent to a quasi-lattice $\partial\colon A\to E$.  This
quasi-lattice is isomorphic to $\pi_1(\G)\to\Lie(\G)$, and hence is
uniquely determined up to isomorphism by $\G$.
\end{proposition}

\begin{proof}
We may assume that $G_\bu$ is \'etale and that $G_0$ is a torus.  By
Lemma~\ref{lemma;abelian}, $G_1$ is abelian and $G_0$ acts trivially
on $G_1$.  We obtain our Morita equivalence from two weak equivalences
as in the diagram
\[
\begin{tikzcd}[row sep=large]
H\ar[d,"\partial"']&\tilde{H}\ar[l]\ar[d,"\tilde{\partial}"']\ar[r]&
A\ar[d,"\partial"']
\\
G&\g\ar[l,"\exp"']\ar[r]&E.
\end{tikzcd}
\]
Here $\tilde{H}$ is the fibred product
\[H\times_G\g=\{\,(h,\xi)\mid\partial(h)=\exp(\xi)\,\}\]
and $E$ is the quotient $E=\g/\h$.  The kernel of the exponential map
$\exp\colon\g\to G$ is isomorphic to $\pi_1(G)=\Hom(\U(1),G)$, the
fundamental group of $G$, and we have a short exact sequence
\[
\pi_1(G)\longinj\tilde{H}\longsur H,
\]
which shows that $\tilde{H}$ is an extension of $H$ by $\pi_1(G)$.
The group $\tilde{H}$ contains a copy $\tilde{\h}$ of $\h$, namely the
image of the embedding $\h\to\tilde{H}$ which sends $\eta$ to
$(\exp(\eta),\Lie(\partial)(\eta))$.  We have
$\Lie(\tilde{H})\cong\tilde{\h}$.  We define $A=\tilde{H}/\tilde{\h}$
to complete the diagram.  Then $\Lie(A)=0$, so $A$ is discrete.  The
Morita equivalence between $H\to G$ and $A\to E$ gives us a group
isomorphism $G/\partial(H)\cong E/\partial(A)$, and hence a surjection
$G\to E/\partial(A)$, which implies that $\partial(A)$ generates $E$
as a vector space, because $G$ is compact.  This proves the existence
of the quasi-lattice $\partial\colon A\to E$.

The uniqueness is proved as follows.  Since $H\to G$ is Morita
equivalent to $A\to E$,
Proposition~\ref{proposition;LiealgLiegroupstack} guarantees that
there is an isomorphism of the abelian Lie algebras
\[
\Lie(\G) \cong \Lie(E)/\Lie(A) = \Lie(E) = E
\]  
Let $E_\bu$ be the Lie $2$-group associated with the crossed module
$A\to E$, then the equivalence $\G\simeq \B E_\bu$ gives a fibration
of stacks $A\to E\to\G$ in the sense of
Noohi~\cite[\S\,5]{noohi;fibrations-topological-stacks}, and hence a
long exact homotopy sequence, which yields $A\cong\pi_1(\G)$.
\end{proof}

\subsection{Maximal stacky tori}\label{subsection;maximal-torus}

We now introduce the appropriate analogue of a maximal torus in the
$2$-categories of Lie $2$-groups and Lie group stacks.

\begin{definition}\label{definition;maximal-2-torus}
Let $G_\bu$ be a foliation $2$-group of compact type.  A \emph{maximal
  $2$-torus} of $G_\bu$ is a full Lie $2$-subgroup $T_\bu$ of $G_\bu$
such that $T_0\subseteq G_0$ is connected and $\ttt_0=\Lie(T_0)$ is a
maximal abelian subalgebra of $\g_0$.
\end{definition}

A maximal $2$-torus is a $2$-torus in the sense of
\S\,\ref{subsection;torus}.  If $(G,H,\partial, \alpha)$ is the
crossed module of $G_\bu$, then the crossed module of $T_\bu$ is
$(T,H_T=\partial^{-1}(T),\partial,\alpha)$.  By
Lemma~\ref{lemma;abelian}, the closed subgroup $H_T$ of $H$ is abelian
and the action of $T$ on $H_T$ is trivial.

\begin{definition}
\label{definition;maximal-stacky-torus}
Let $\G$ be a compact connected \'etale Lie group stack.  A
\emph{maximal stacky torus} is a strict homomorphism $\TT\to\G$ of Lie
group stacks with the following property: there exist a foliation
$2$-group $G_\bu$ of compact type, a maximal $2$-torus $T_\bu$ of
$G_\bu$, and presentations
\[
\B G_\bu\simeq\G\qquad\text{and}\qquad\B T_\bu\simeq\TT,
\]
such that the following diagram commutes:
\[
\begin{tikzcd}
\B T_\bu\ar[r]\ar[d,"\simeq"']&\B G_\bu\ar[d,"\simeq"]\\
\TT\ar[r]&\G.
\end{tikzcd}
\]
\end{definition}

It follows from the strictification theorem,
Corollary~\ref{corollary;compact-lie-group-stack}, that every compact
connected \'etale Lie group stack has a maximal stacky torus.
%
%

\begin{lemma}\label{lemma;torus-morita}
Let
$(\psi_G,\psi_H)\colon(G,H,\partial,\alpha)\to(G',H',\partial',\alpha')$
be Morita morphism of crossed modules.  Assume the Lie algebras $\g$
and $\g'$ are compact, and assume $\psi_G\colon G\to G'$ is
surjective.  Let $\lie{t}$ and $\lie{t}'$ be maximal abelian Lie
subalgebras of $\g$ and $\g'$, respectively, chosen so that $\lie{t}'$
contains $\Lie(\psi_G)(\lie{t})$.  Let $T\subseteq G$ and $T'\subseteq
G'$ be the connected subgroups with Lie algebras $\lie{t}$ and
$\lie{t}'$, respectively.  Then the restriction
\[
(\psi_T,\psi_{H_T}):=(\psi_G|_{T},\psi_H|_{H_T})\colon
(T,H_T,\partial,\alpha)\longto(T',H'_T,\partial',\alpha')
\]
is a Morita morphism.
\end{lemma}

\begin{proof}
Let us first show that $\psi_T\colon T\to T'$ is surjective.  Since
$\g$ and $\g'$ are compact, they decompose into
$\g\cong[\g,\g]\oplus\z(\g)$ and $\g'\cong[\g',\g']\oplus\z(\g')$,
where $[\g,\g]$ and $[\g',\g']$ are the derived subalgebras of $\g$
and $\g'$, respectively, and $\z(\g)$ and $\z(\g')$ are the centers of
$\g$ and $\g'$, respectively.  Since $\psi_G$ is surjective,
$\Lie(\psi_G)$ maps $\z(\g)$ onto $\z(\g')$.  If
$[\g,\g]=\bigoplus_i\g_i$ is a decomposition into simple subalgebras,
then $[\g',\g']=\bigoplus_i\g_i'$, where
$\g_i'=\Lie(\psi_G)(\g_i)\cong\g_i$ or $\g'_i=0$.  It follows that
$\Lie(\psi_G)(\ttt)$ is a maximal abelian subalgebra of $\g'$, and
hence $\psi_T(T)=T'$.

Next we show that $H_T=\partial\n(T)\cong
T\times_{T'}(\partial')\n(T')$.  Since $(\psi_G,\psi_H)$ is a Morita
morphism, we may identify $H= G\times_{G'} H'$ and $H_T= T\times_{G'}
H'=T\times_{T'} H'$.  Assume that $t\in T$ and $h'\in H'$ with
$\psi_T(t)=\partial'(h')$.  Then $\partial'(h')\in T'$, so
$h'\in(\partial')\n(T')$.  Thus $H_T=T\times_{T'}(\partial')\n(T')$.
The result now follows from Lemma~\ref{lemma;morita-crossed}.
\end{proof}

The conjugation action $\bo{C}\colon\G\times\G\to\G$ is a strict
action of $\G$ on itself.  Composing this action with a categorical
point $\bo{g}\colon\star\to\G$ gives an equivalence
$\bo{C_g}\colon\G\to\G$, called \emph{conjugation by $\bo{g}$}.

\begin{corollary}\label{cor;maxtorusinvariant}
Let $\G$ be a compact connected \'etale Lie group stack and let
$\TT\to\G$ and $\TT'\to\G$ be two maximal stacky tori.  Then there
exist a categorical point $\bo{g}\colon\star\to\G$ and an equivalence
of Lie group stacks $\TT\simeq\TT'$ so that the following diagram
commutes:
\[
\begin{tikzcd}
\TT\ar[r]\ar[d,"\simeq"']&\G\ar[d,"\bo{C_g}"]\\
\TT'\ar[r]&\G.
\end{tikzcd}
\]
\end{corollary} 

\begin{proof}
By Lemma~\ref{lemma;torus-morita} and
Proposition~\ref{proposition;strictmor}, we can assume that there
exists a single base-connected foliation Lie $2$-groups of compact
type $G_\bu$ presenting $\G$, so that $\TT$ and $\TT'$ both come from
maximal $2$-tori $T_\bu$ and $T'_\bu$ of $G_\bu$. From the theory of
Lie groups, the subgroups $T_0$ and $T_0'$ of $G_0$ are related by
conjugation in $G_0$. Passing to the associated Lie group stacks gives
the result.
\end{proof}

By an argument as in Corollary~\ref{cor;maxtorusinvariant}, one can
also show that a maximal stacky torus $\TT$ is maximal in the sense
that, if $\TT'\to \G$ is a sub-Lie group stack and $\TT'$ is a stacky
torus, then there is some $\bo{g}\colon\star\to\G$ so that $\TT'\to
\G$ factors through the morphism
$\begin{tikzcd}[cramped,column sep=1.4em]
\TT\ar[r]&\G\ar[r,"\bo{C_g}"]&\G
\end{tikzcd}$.

\subsection{Principal bundles}\label{subsection;principal}

Bursztyn, Noseda, and
Zhu~\cite{bursztyn-noseda-zhu;principal-stacky-groupoids} have
introduced the notion of a principal bundle in the $2$-category of
differentiable stacks.  We briefly introduce the analogous notion in
the $2$-category of Lie groupoids, and compare the two.

\begin{definition}
\label{definition;principal-groupoid}
Let $G_\bu$ be a Lie $2$-group, $X_\bu$ and $Y_\bu$ Lie groupoids, and
$a_\bu\colon G_\bu\times X_\bu\to X_\bu$ a strict action.  A Lie
groupoid morphism $\psi_\bu\colon X_\bu \to Y_\bu$ is
\emph{$G_\bu$-invariant} if there is a given $2$-isomorphism
$\gamma\colon\psi_\bu\circ\pr_2\To\psi_\bu\circ a_\bu$ making the
following diagram $2$-commute:
\[
\begin{tikzcd}
G_\bu\times X_\bu\ar[r,"a_\bu"]\ar[d,"\pr_2"']&X_\bu\ar[d,"\psi_\bu"]
\\
X_\bu\ar[r,"\psi_\bu"]&Y_\bu.
\end{tikzcd}
\]
The $2$-isomorphism $\gamma\colon G_0\times X_0 \to Y_1$ is required
to satisfy two coherence conditions, for all $x\in X_0$ and $g,h\in
G_0$:
\begin{equation}\label{equation;coherencegpdprinc}
\gamma(h,g\cdot x)\circ\gamma(g,x)=\gamma(hg,x)\qquad\text{and}\qquad
\gamma(1,x)=u(\psi_0(x)).
\end{equation}
A Lie groupoid morphism $\psi_\bu\colon X_\bu\to Y_\bu$ is a
\emph{principal $G_\bu$-bundle} if
\begin{enumerate}
\item\label{item;principal1} $\psi_\bu$ is essentially surjective;
\item\label{item;principal2} $\psi_\bu$ is $G_\bu$-invariant;
\item\label{item;principal3} the canonical morphism $G_\bu\times
  X_\bu\to X_\bu\times_{Y_\bu}^{(w)}X_\bu$ is a Morita morphism.
\end{enumerate}
\end{definition}

\begin{example}
Let $G$ be a Lie group (viewed as the identity $2$-group
$G\rightrightarrows G$) and $X$ a $G$-manifold (viewed as the identity
groupoid $X\rightrightarrows X$).  Then the morphism $\psi_\bu\colon
X\to G\ltimes X$ defined by $\psi_1(x)=(1,x)$ is a principal
$G$-bundle.
\end{example}

\begin{definition}
\label{definition;principal-stack}
Let $\G$ be a Lie group stack, $\X$ and $\Y$ differentiable stacks,
and $\ba\colon\G\times\X\to\X$ a strict action.  A morphism
$\bpsi\colon\X\to\Y$ is \emph{$\G$-invariant} if there is a given
$2$-isomorphism $\bgamma\colon\bpsi\circ\pr_2\To\bpsi\circ\ba$
satisfying two coherence conditions analogous to
\eqref{equation;coherencegpdprinc}, namely that the following diagrams
commute:
\[
\begin{tikzcd}[row sep=large,column sep=small,/tikz/column 1/.append
    style={anchor=base east},/tikz/column 3/.append style={anchor=base
      west}]
\bpsi\circ\pr_2\circ\pr_{23}\ar[r,Rightarrow,"\bgamma"]
\ar[d,equal,start anchor=south east,end anchor=north east,shift
  right=8]& \bpsi\circ\ba\circ\pr_{23}
\ar[r,equal]&
\bpsi\circ\pr_2\circ(\id_\G\times\ba)
\ar[d,Rightarrow,start anchor=south west,end anchor=north west,shift
  left=8,"\bgamma"]\\
\bpsi\circ\pr_2\circ(\bm\times\id_\X)
\ar[r,Rightarrow,"\bgamma"]&
\bpsi\circ\ba\circ(\bm\times\id_\X)
\ar[r, equal]&
\bpsi\circ\ba\circ(\id_\G\times\ba)
\end{tikzcd}
\]
\[
\begin{tikzcd}[row sep=large]
\bpsi\circ\pr_2\circ(1_\G\times\id_\X)\ar[d,Rightarrow,"\bpsi"']
\ar[dr,equal,start anchor=south east]&\\
\bpsi\circ\ba\circ(1_\G\times\id_\X)\ar[r,equal]&\bpsi\circ\pr_2.
\end{tikzcd}
\]
Here we omit horizontal composition of $1$-morphisms and
$2$-morphisms, and $\pr_{23}\colon\G\times\G\times\X\to\G\times\X$ is
the projection to the second and third factors.  A stack morphism
$\bpsi\colon\X\to\Y$ is a \emph{principal $\G$-bundle} if
\begin{enumerate}
\item\label{item;principal1-stack}
$\bpsi$ is a stack epimorphism, and there is an atlas $X\to\X$ so the
  composition $X\to\X\to\Y$ is representable;
\item\label{item;principal2-stack}
$\bpsi$ is $\G$-invariant;
\item\label{item;principal3-stack}
the canonical morphism $\G\times\X\to\X\times_\Y\X$ is an equivalence.
\end{enumerate}
\end{definition}

\begin{proposition}\label{proposition;principal-groupoid-stack}
Let $(\psi_\bu\colon X_\bu\to Y_\bu,\gamma)$ be a principal
$G_\bu$-bundle.  Then $(\B\psi_\bu\colon\B X_\bu\to\B Y_\bu,\B\gamma)$
is a principal $\B G_\bu$-bundle.
\end{proposition}

\begin{proof}
Since $\psi_\bu$ is essentially surjective, the map $\B\psi_\bu$ is an
epimorphism of stacks.  To check the conditions of
Definition~\ref{definition;principal-stack}, it suffices to note that
$\B$ is a $2$-functor which preserves weak fibred products
(Lemma~\ref{lemma;preservelimits}), and to show that there is an atlas
$M\to\B X_\bu$ so that the composition $M\to\B X_\bu\to\B Y_\bu$ is
representable.  Indeed, we can take $M=X_0$.  Then $\B
Y_\bu\simeq\B\psi_0^*(Y_\bu)$, where $\psi_0^*(Y_\bu)$ is the pullback
groupoid, and the composition $X_0\to\B
X_\bu\to\B(\psi_0^*(Y_\bu)\simeq\B Y_\bu$ is representable.
\end{proof}

\begin{remark}\label{remark;formsinject}
Let $\psi_\bu\colon X_\bu\to Y_\bu$ be a principal $G_\bu$-bundle.
The canonical morphism $X_\bu\to\psi_0^*(Y_\bu)$ is the identity on
the manifold of objects $X_0$, and therefore induces an injection
$\Omega_\bas^\bu (\psi_0^*(Y_\bu))\to\Omega_\bas^\bu(X_\bu)$.  The
canonical morphism $\psi_0^*(Y_\bu)\to Y_\bu$ is a Morita morphism, so
$\Omega_\bas^\bu(\psi^*(Y_\bu))$ is isomorphic to
$\Omega_\bas^\bu(Y_\bu)$. We conclude that the pullback map of basic
forms $\psi^*\colon\Omega_\bas^\bu(Y_\bu)\to\Omega_\bas^\bu(X_\bu)$ is
an injection.
\end{remark}

\section{Hamiltonian actions on groupoids and stacks}
\label{section;hamiltonian-stack}

In this section we introduce Hamiltonian actions in the $2$-categories
of Lie groupoids and of differentiable stacks.  Our notion of
Hamiltonian actions extends that of Hamiltonian Lie group actions on
stacks defined by Lerman and
Malkin~\cite{lerman-malkin;deligne-mumford} in two ways: we allow our
groups to be \'etale Lie group stacks, and we allow our stacks to be
non-separated.  We show that every presymplectic Hamiltonian action
can be integrated, in several different ways, to a Hamiltonian action
of a foliation groupoid on a $0$-symplectic groupoid
(Theorem~\ref{theorem;integrate}).  The classifying functor $\B$ takes
any such Hamiltonian groupoid to a Hamiltonian stack.  The converse is
also true: any Hamiltonian stack arises from a Hamiltonian groupoid
(Theorem~\ref{theorem;stack-groupoid}).  We then show that, for
compact Lie group stacks, the moment map image is an invariant of a
stacky Hamiltonian action and obtain the stacky convexity theorem
(Theorem~\ref{theorem;convex}).

\subsection{Hamiltonian actions on
  \texorpdfstring{$0$-symplectic}{0-symplectic} Lie groupoids}
\label{subsection;hamiltonian}

Let $(X_\bu,\omega_\bu)$ be a $0$-symplectic Lie groupoid
(\S\,\ref{subsection;symplectic-groupoid}), let $G_\bu$ be a foliation
$2$-group (Definition~\ref{definition;foliation2group}), and let
$(G,H,\partial,\alpha)$ be the associated crossed module
(\S\,\ref{subsection;crossed}).  We assume the coarse quotient group
$G_0/G_1$ to be connected.  A strict action $a_\bu\colon G_\bu\times
X_\bu\to X_\bu$ is \emph{Hamiltonian} if there is a morphism of Lie
groupoids called the \emph{moment map}
\[
\mu_\bu=(\mu_0,\mu_1)\colon
X_\bu\longto(\g/\h)^*\cong\ann(\h)\subseteq\g^*
\]
which satisfies the following conditions:
\glossary{mu*@$\mu_\bu$, moment map for $2$-group action on
  $0$-symplectic groupoid}%
\begin{enumerate}
\item
Let $\xi\in \g/\h$, and let $\xi_{X_\bu}\in \Vect_\bas(X_\bu)$ be the
fundamental vector field of $\xi$ on $X_\bu$, as in Definition
\ref{definition;fundamentalbasic}.  Then we require
\[
d\mu_0^{\xi}=\iota_{\xi_{X_0}}\omega_0.
\] 
\item
Recall the \emph{coadjoint action} of $G_\bu$ on $(\g/\h)^*$ defined
in \S\,\ref{subsection;adjoint}.  We require that $\mu_\bu$ is
(strictly) equivariant with respect to the $G_\bu$ action on $X_\bu$
and the coadjoint action.
\end{enumerate}
A \emph{Hamiltonian $G_\bu$-groupoid} is a tuple
$(X_\bu,\omega_\bu,G_\bu,\mu_\bu)$ consisting of a $0$-symplectic
groupoid equipped with a Hamiltonian $G_\bu$-action with moment map
$\mu_\bu$.
\glossary{X*omega*G*mu*@$(X_\bu,\omega_\bu,G_\bu,\mu_\bu)$,
  Hamiltonian $G_\bu$-groupoid}%

\begin{example}\label{example8} 
The actions of Example~\ref{example4} are Hamiltonian, with moment map
$(\mu_0,\mu_0\circ s)$ which was described in Example~\ref{example1}.
\end{example}

\subsection{From manifolds to Lie groupoids}
\label{subsection;manifold-groupoid}

Let $(X,\omega,G,\mu)$ be a Hamiltonian presymplectic $G$-manifold
with null foliation $\F=\ker(\omega)$ and null ideal $\nnn=\nnn(\F)$.
We will show we can integrate these data to a Hamiltonian Lie groupoid
$(X_\bu,\omega_\bu,G_\bu,\mu_\bu)$, where $X_\bu$ is a foliation
groupoid that integrates $(X,\F)$, and $G_\bu$ is a foliation
$2$-group that integrates the Lie $2$-algebra
$\g_\bu=(\nnn\rtimes\g\rightrightarrows\g)$.  As usual we can
integrate both $(X,\F)$ and $\g_\bu$ in a number of different ways,
but we have to integrate them in a compatible manner.  The following
theorem shows that the monodromy groupoid of $X$ and the source-simply
connected integration of $\g_\bu$ always work.  If the action of $G$
on the monodromy groupoid descends to an action on the holonomy
groupoid, then we can take $X_\bu$ to be the holonomy groupoid and
$G_\bu=(N\rtimes G\rightrightarrows G)$, where $N=N(\F)$ is the null
subgroup, i.e.\ the immersed subgroup of $G$ generated by $\nnn$.

\begin{theorem}\label{theorem;integrate}
Let $G$ be a Lie group and $(X,\omega,G,\mu)$ a presymplectic
Hamiltonian $G$-manifold with null foliation $\F$ and null ideal
$\nnn$.
\begin{enumerate}
\item\label{item;2-group}
There exists a source-simply connected Lie $2$-group $G_\bu$ with
object group $G_0=G$ and Lie $2$-algebra
$\Lie(G_\bu)=\nnn\rtimes\g\rightrightarrows \g$.  This Lie $2$-group
is unique up to a unique isomorphism that induces the identity map of
$G$ and of $\nnn\rtimes\g$.
\item\label{item;integrate}
Let $X_\bu$ be a source-connected Lie groupoid over $X_0=X$
integrating $\F$ and let $\psi_\bu=\psi_{X_\bu}\colon\Mon(X,\F)\to
X_\bu$ be the universal morphism as in Theorem~\ref{theorem;monhol}.
There exists a $G_\bu$-action on $X_\bu$ that extends the action of
$G_0=G$ on $X$ if and only if $\ker(\psi_\bu)$ is preserved by the
$G$-action, where $\ker(\psi_\bu)$ is the kernel of $\phi_\bu$ as
defined in~\ref{definition;Liegroupbundle}.  This $G_\bu$-action is
unique and it is Hamiltonian with respect to the $0$-symplectic
structure $\omega_\bu$ on $X_\bu$ determined by $\omega$.
\item\label{item;holdesc}
Assume that the $G$-action preserves the kernel of the holonomy
homomorphism $\hol\colon\Mon(X,\F)\to\Hol(X,\F)$.  Then the
$G_\bu$-action on $\Hol(X,\F)$ given by~\eqref{item;integrate}
descends to a Hamiltonian action of the Lie $2$-group $N\rtimes
G\rightrightarrows G$, where $N\subseteq G$ is the null subgroup.
\end{enumerate}
\end{theorem}

\begin{proof}
\eqref{item;2-group}~Let $H$ be the universal cover of $N$.  We
interpret elements of $H$ as homotopy classes relative to endpoints of
paths $\nu\colon[0,1]\to N$ starting at $\nu(0)=1$.  Define the
homomorphism $\partial\colon H\to G$ by $\partial([\nu])=\nu(1)$ and
the action $\alpha\colon G\to\Aut(H)$ by
$\alpha(g)([\nu])={}^g[\nu]=[\tau\mapsto g\nu(\tau)g^{-1}]$.  This
defines a crossed module $(G,H,\partial,\alpha)$ and hence a Lie
$2$-group $G_\bu$ with $G_0=G$, simply connected source fibre
$\ker(s)=H$, and $\Lie(G_1)=\nnn\rtimes\g$.  The uniqueness of $G_\bu$
follows from the uniqueness of the simply connected group $H$.

\eqref{item;integrate}~We start by showing that the $G$-action on $X$
extends in at most one way to a $G_\bu$-action on $X_\bu$.  The
$G$-action on $X$ preserves the foliation $\ca{F}$ and therefore
induces an action on $\Lie(X_\bu)=T\ca{F}$ by Lie algebroid
automorphisms.  Since $X_\bu$ is source-connected, by Lie's theorems
for Lie groupoids
(see~\cite[\S\,6.3]{moerdijk-mrcun;foliations-groupoids}) there can
exist no more than one $G$-action on $X_\bu$ that is compatible with
the action on $\Lie(X_\bu)$.  We show that the $H$-action on $X_1$ is
unique by showing that the action of its Lie algebra $\h$ is unique.
Let $\eta\in\h$.  By assumption $\partial\colon\h\to\g$ is an
isomorphism onto $\nnn$, so the vector field $\partial(\eta)_X$ is
tangent to the foliation, and therefore
$\partial(\eta)_X=\rho(\sigma(\eta))$ for a unique section
$\sigma(\eta)$ of $\Lie(X_\bu)=T\ca{F}$.  Let $\sigma(\eta)_R$ be the
right-invariant vector field on $X_1$ determined by $\sigma(\eta)$.
By Lemma~\ref{lemma;orbits} we must have $\sigma(\eta)_R=\eta_{X_1}$.
So the $\h$-action on $X_1$ is determined by the $G$-action on $X$.

Next we show the existence of a $G_\bu$-action on $X_\bu$.  First
consider the case of $M_\bu=\Mon(X,\ca{F})$, the monodromy groupoid of
$\ca{F}$.  An element of $M_1$ is a leafwise homotopy class relative
to endpoints of a path $\gamma\colon[0,1]\to M_0=X$ in a leaf of
$\ca{F}$.  Let $[\nu]\in H$.  The homotopy class of the path
$\nu\cdot\gamma\colon\tau\mapsto\nu(\tau)\cdot\gamma(\tau)$ depends
only on the homotopy classes $[\nu]\in H$ and $[\gamma]\in M_1$, so we
have a well-defined action of $H$ on $M_1$ given by
$[\nu]*[\gamma]=[\nu\cdot\gamma]$.  Similarly, the homotopy class of
the path $g\cdot\gamma\colon\tau\mapsto g\cdot\gamma(\tau)$ depends
only on $g\in G$ and on the homotopy class $[\gamma]\in M_1$, so we
have a well-defined action of $G$ on $M_1$ given by
$g*[\gamma]=[g\cdot\gamma]$.  The actions of $G$ on $M_0$ and on $M_1$
and the action of $H$ on $M_1$ satisfy the
rules~\eqref{equation;functor1}--\eqref{equation;conjugation}, and
therefore combine to an action of $G_\bu$ on $M_\bu$.

Now consider the general case of a source-connected groupoid $X_\bu$
integrating $(X,\ca{F})$ and for which the kernel of $\psi_\bu\colon
M_\bu\to X_\bu$ is preserved by $G$.  Then the $G$-action on $M_1$
descends to a $G$-action on $X_1=M_1/\ker(\psi)$.  As we saw in the
discussion of uniqueness, the $G$-action on $X_0$ determines an action
of $\h$ on $X_1$.  For each $\eta\in\h$ the vector field $\eta_{X_1}$
is complete because it lifts to the complete vector field
$\eta_{M_1}$.  Hence, by the Lie-Palais theorem, the $\h$-action on
$X_1$ integrates to an $H$-action.  (Here we use that the source
fibres of $X_\bu$ are Hausdorff by Lemma~\ref{lemma;hausdorff} and
that the vector fields $\eta_{X_1}$ are tangent to the source fibres
by Lemma~\ref{lemma;orbits}.)  The
conditions~\eqref{equation;functor1}--\eqref{equation;conjugation}
hold because they hold on $M_\bu$.

Finally, the moment map $\mu$ is invariant on leaves of the null
foliation, so it defines a map of Lie groupoids $\mu\colon
X_\bu\to\ann(\nnn)$.  The map $\mu_0$ is $G$-equivariant by
assumption.

\eqref{item;holdesc} It suffices to show that the kernel of the
homomorphism $\partial\colon H\to G$ acts trivially on the arrows of
$\Hol(X,\F)$.  If $[\nu]\in\ker(\partial)$, then $\nu$ is a loop based
at the identity of $N$.  Let $S$ be a local transversal to $x\in X$,
then $\hol([\nu\cdot u(x)])$ is the germ of the holonomy action of the
path $\nu(\tau)\cdot x\colon[0,1]\times\{x\}\to X$ on $S$.  This path
extends to $\nu(\tau)\cdot S\colon [0,1]\times S\to X$, and so the
holonomy action on $S$ is just $\nu(1)\cdot S= \nu(0)\cdot S$, and
therefore $\hol([\nu\cdot u(x)])=u(x)$.  For $\gamma\in\Mon(X,\F)$,
applying~\eqref{equation;natural2} gives
\[
[\nu]*\hol([\gamma])=\hol([\nu\cdot\gamma])=\hol([\nu\cdot
  u(t(g))]\circ[\gamma])=\hol([\gamma]),
\]
which proves the claim.
\end{proof}  

\subsection{Hamiltonian actions on stacks}
\label{subsection;hamiltonian-stack}

Let $(\X,\bomega)$ be a symplectic stack
(\S\,\ref{subsection;symplectic-stack}) and let $\G$ be a connected
\'etale Lie group stack (\S\,\ref{subsection;2-group}).  A weak action
$\ba\colon\G\times\X\to\X$ is \emph{Hamiltonian} if there is a
morphism $\bmu\colon\X\to(\Lie(\G))^*$ called the \emph{moment map},
where $(\Lie(\G))^*$ is the dual vector space of $\Lie(\G)$.  We
require that
\glossary{mux@$\bmu$, moment map for Lie group stack action on
  symplectic stack}%
\begin{enumerate} 
\item
$d\bmu^\bxi=\iota_{\bxi_\X}\bomega$ for all $\bxi\in\Lie(\G)$, and
\item
$\bmu$ is $\G$-equivariant with respect to the coadjoint action of
  $\G$ on $\Lie(\G)^*$.
\end{enumerate}
A \emph{Hamiltonian $\G$-stack} is a tuple $(\X,\bomega,\G,\bmu)$
consisting of a symplectic stack $(\X,\bomega)$ and a Hamiltonian
$\G$-action with moment map $\bmu$.
\glossary{Xxomegagmu@$(\X,\bomega,\G,\bmu)$, Hamiltonian
  $\G$-stack}%

\begin{definition}\label{definition;equivhamstack}
An \emph{equivalence of Hamiltonian $\G$-stacks}
\[\bphi\colon(\X,\bomega,\G,\bmu)\simeq(\X',\bomega',\G',\bmu')\]
is a pair $(\bphi_\X,\bphi_\G)$, where
$\bphi_\X\colon(\X,\bomega)\simeq(\X',\bomega')$ is an equivalence of
symplectic stacks, and $\bphi_\G\colon\G\simeq\G'$ is an equivalence
of Lie group stacks, subject to the following conditions.  The
equivalence $\bphi_\G$ determines an action of $\G'$ on $\X$, which is
\[
\begin{tikzcd}[column sep=3em]
\G'\times\X\ar[r,"\bphi_\G\n\times{\id}"]&\G\times\X\ar[r,"\ba"]&\X,
\end{tikzcd}
\]
where $\bphi_\G\n$ is a weak inverse of $\bphi_\G$.  The tuple
$\bigl(\X,\bomega,\G',\Lie\bigl(\bphi_\G\n\bigr)^*\circ\bmu\bigr)$ is
a Hamiltonian $\G'$-stack, where $\Lie(\bphi_\G)$ is as
in~\eqref{equation;lie-stack}.  We then require that
\begin{enumerate}
\item
$\bphi_\X$ is $\G'$-equivariant, and
\item\label{item;equivstackmoment}
$\bmu'\circ\bphi_\X = \Lie\bigl(\bphi_\G\n\bigr)^*\circ\bmu$.
\end{enumerate}
\end{definition}

The preceding sections establish the following bijection up to
equivalence between Hamiltonian stacks and groupoids.  For the
following theorem, recall that by the strictification theorem
(Theorem~\ref{theorem;lie-group-stack}), every connected \'etale Lie
group stack $\G$ has a presentation $\B G_\bu\simeq\G$ by a
base-connected foliation $2$-group $G_\bu$.

\begin{theorem}\phantomsection\label{theorem;stack-groupoid}
\begin{enumerate}
\item\label{item;groupoid-stack}
Let $(X_\bu,\omega_\bu, G_\bu,\mu_\bu)$ be a Hamiltonian
$G_\bu$-groupoid.  Then the classifying stack $(\B
X_\bu,\B\omega_\bu,\B G_\bu,\B\mu_\bu)$ is a Hamiltonian $\G$-stack.
\item\label{item;stack-groupoid}
Let $(\X,\bomega,\G,\bmu)$ be a Hamiltonian $\G$-stack. For every
base-connected Lie $2$-group $G_\bu$ with $\B G_\bu\simeq\G$, there
exists a Hamiltonian $G_\bu$-groupoid
$(X_\bu,\omega_\bu,G_\bu,\mu_\bu)$ so that $(\B X_\bu,\B\omega_\bu,\B
G_\bu,\B\mu_\bu)$ is equivalent to $(\X,\bomega,\G,\bmu)$ as a
Hamiltonian $\G$-stack.
\end{enumerate}
\end{theorem}

\begin{proof}
\eqref{item;groupoid-stack}~This follows from
Proposition~\ref{proposition;symstack} and the definitions of
Hamiltonian groupoids and stacks.

\eqref{item;stack-groupoid}~Let $X_\bu$ and $G_\bu$ be Lie groupoids
presenting $\X$ and $\G$ as in Theorem~\ref{theorem;strictaction}, so
that $G_\bu\times X_\bu\to X_\bu$ is a strict action presenting $\ba$.
It follows from Proposition~\ref{proposition;symstack} that $X_\bu$ is
$0$-symplectic Lie groupoid.  The map
$\begin{tikzcd}[cramped,sep=small]X_0\ar[r]&\X\ar[r,"\bmu"]&\Lie(\G)^*
\end{tikzcd}$
gives a moment map for the $G_\bu$ action on $X_\bu$, and $X_\bu$ is a
Hamiltonian $G_\bu$-groupoid.  Finally, $\B X_\bu$ is equivalent to
$\X$ as a Hamiltonian stack by construction.
\end{proof}

\subsection{The stacky moment body}\label{subsection;moment-body}

In this subsection $\G$ denotes a compact connected \'etale Lie group
stack and $\TT\to\G$ a maximal stacky torus.  The Lie algebra
$\Lie(\G)$ is compact and $\Lie(\TT)$ is a maximal abelian subalgebra.
Moreover, $\Lie(\TT)$ is in a natural way a direct summand of
$\Lie(\G)$, so we can identify $\Lie(\TT)^*$ with a subspace of
$\Lie(\G)^*$.  The \emph{Weyl group} of $\G$ relative to $\TT$ is by
definition $W=W(\Lie(\G),\Lie(\TT))$, the Weyl group of $\Lie(\G)$
relative to $\Lie(\TT)$.  We choose a closed Weyl chamber $C$ for the
$W$-action on $\Lie(\TT)^*$.

Let $(\X,\bomega,\G,\bmu)$ be a Hamiltonian $\G$-stack.  The
\emph{moment map image} is the set $\bmu(\X)\subseteq\Lie(\G)^*$
defined by $\bmu(\X)=\mu_0(X_0)=\mu_1(X_1)$ for any presentation $\B
X_\bu\simeq\X$ of $\X$ and $\B\mu_\bu\cong\bmu$ of $\bmu$.  Picking
different presentations of $\X$ and $\bmu$ leaves $\bmu(\X)$
unchanged.  We define the \emph{stacky moment body} of
$(\X,\bomega,\G,\bmu)$ to be the pair $(\Delta(\X),\TT)$, where
$\Delta(\X)=\bmu(\X)\cap C$.  An \emph{equivalence} of stacky moment
bodies $(\Delta,\TT)\simeq(\Delta',\TT')$ is an equivalence
$\begin{tikzcd}[cramped,sep=small]
  \bphi_\TT\colon\TT\ar[r,"\simeq"]&\TT'\end{tikzcd}$ so that
  $\Lie(\bphi_\TT)^*(\Delta')=\Delta$.
\glossary{DeltaXT@$(\Delta(\X),\TT)$, stacky moment body of $\X$}%

\begin{proposition}\label{proposition;momentinvariance}
Let $(\X,\bomega,\G,\bmu)$ be a Hamiltonian $\G$-stack.
\begin{enumerate}
\item\label{item;body}
Up to equivalence, the stacky moment body of $\X$ is independent of
the choice of the maximal stacky torus $\TT\to\G$ and the Weyl chamber
$C$.
\item\label{item;body-equivalent}
If $\bphi\colon(\X,\bomega,\G,\bmu)\simeq(\X',\bomega',\G',\bmu')$ is
an equivalence of Hamiltonian stacks, then the stacky moment body of
$\X$ is equivalent to the stacky moment body of $\X'$.
\end{enumerate}
\end{proposition}

\begin{proof}
\eqref{item;body}~Given two maximal stacky tori $\TT$ and $\TT'$ of
$\G$ and Weyl chambers $C\subseteq\Lie(\TT)^*$ and
$C'\subseteq\Lie(\TT')^*$, arguing as in
Corollary~\ref{cor;maxtorusinvariant}, there exists a categorical
point $\bo{g}\colon\star\to\G$ so that
$\Bad_{\bo{g}}^*\colon\Lie(\G)^*\to\Lie(\G)^*$ maps $\Lie(\TT)^*$ to
$\Lie(\TT')^*$ and $C$ to $C'$.  Therefore, by equivariance of the
moment map,
\[
\Bad_{\bo{g}}^*(\bmu(\X)\cap C)=\bmu(\X)\cap\Bad_{\bo{g}}^*(C)=
\bmu(\X)\cap C',
\]
as desired.

\eqref{item;body-equivalent}~Given a pair of equivalences
$\bphi_\X\colon\X\simeq\X'$ and $\bphi_\G\colon\G\simeq\G'$ as in
Definition~\ref{definition;equivhamstack}, we have
$\bmu'(\X')=\Lie\bigl(\bphi_\G\n\bigr)^*(\bmu(\X))$ by
Definition~\ref{definition;equivhamstack}\eqref{item;equivstackmoment}.
Choose a maximal torus $\TT\to\G$ of $\G$ and a chamber
$C\subseteq\ttt^*$, let $C'$ be the chamber
$\Lie\bigl(\bphi_\G\n\bigr)^*(C)$ of the maximal torus
\[
\begin{tikzcd}[cramped]
\TT\ar[r]&\G\ar[r,"\bphi_\G"]&\G'
\end{tikzcd}
\] 
of $\G'$.  Then $\bphi_G$ defines an equivalence from
$(\Delta(\X),\TT\to\G)$ to $(\Delta(\X'),\TT\to\G')$.
\end{proof}

We can now rephrase the main result
of~\cite{lin-sjamaar;presymplectic} in the language of stacks.

\begin{theorem}\label{theorem;convex}
Let $(\X,\bomega,\G,\bmu)$ be a Hamiltonian $\G$-stack.  If
$(\X,\bomega,\G,\bmu)$ is equivalent to $(\B X_\bu,\B \omega_\bu,\B
G_\bu,\B\mu_\bu)$, where $X_0$ and $G_0$ are compact and connected,
and the action of $G_0$ on $X_0$ is clean, then $\Delta(\X)$ is a
closed convex polyhedron.
\end{theorem}

\begin{proof}
This follows from Proposition~\ref{proposition;momentinvariance} and
Theorem~\ref{theorem;LiSj}.
\end{proof}

\begin{example}\label{example;quasi}
Let $(X_\bu,\omega_\bu,G_\bu,\mu_\bu)$ be the Hamiltonian groupoid of
Example~\ref{example8}.  The Lie $2$-group $G_\bu$ presents a stacky
torus $\G$.  Let $(\X,\bomega,\G,\bmu)$ be the Hamiltonian $\G$-stack
presented by $(X_\bu,\omega_\bu,G_\bu,\mu_\bu)$.  Referring back to
Example~\ref{example1}, we see that the image of the moment map
$\mu_0\colon X_0\to\ann(\lie{n})$ is the polyhedron
\[
P=\{\,\eta\in\ann(\lie{n})\mid\text{$\inner{a_i,\eta}\ge\lambda_i$ for
  $1\le i\le n$}\,\}.
\]
Here $a_i\in\g/\lie{n}$ is the image of the standard basis vector
$e_i\in\g=\R^n$ under the projection $\g\to\g/\lie{n}$.  We conclude
that the stacky moment body is $\Delta(\X)=(P,\G)$.  Following
Prato~\cite{prato;non-rational-symplectic} we call $\X$ a \emph{toric
  quasifold} associated with the stacky polyhedron $(P,\G)$.
\end{example}

\begin{remark}
Lerman and Tolman%
~\cite{lerman-tolman;hamiltonian-torus-actions-symplectic-orbifolds}
show that compact symplectic toric orbifolds (which can be thought of
as certain separated Hamiltonian stacks, as
in~\cite{lerman-malkin;deligne-mumford}) are classified by their
moment polytopes which have positive integer labels attached to their
faces. Let $M$ be a toric $\T^k$-orbifold and let $n$ be the number of
faces of the moment polytope $\Delta$ of $M$.  The labels of $\Delta$
can be thought of as defining a homomorphism $\Z^n\to \Z^k$.  More
generally, we can consider homomorphisms $\Z^n\to A$ labeling the
stacky moment polytope of a compact Hamiltonian $\G$-stack, where
$A\to E$ is a quasi-lattice presenting a stacky torus $\G$.
See~\cite{hoffman;toric-stacks} for more on this perspective and for
an interpretation of the results
of~\cite{lerman-tolman;hamiltonian-torus-actions-symplectic-orbifolds},
\cite{prato;non-rational-symplectic},
and~\cite{ratiu-zung;presymplectic} in a stacky context.
\end{remark}

\section{Leafwise transitivity}\label{section;hamiltonian-groupoid}

In this section we prove two basic structural results about $2$-group
actions on groupoids, with an eye toward the symplectic reduction
theorem and the Duistermaat-Heckman theorem.  We introduce the notion
of a leafwise transitive Lie $2$-group action.  We show that if
$(X_\bu,\omega_\bu,G_\bu,\mu_\bu)$ is a Hamiltonian $G_\bu$-groupoid,
then a regular fibre of $\mu_\bu$ is Morita equivalent to a Lie
groupoid with a locally leafwise transitive $G_\bu$-action
(Proposition~\ref{proposition;invarianttransversal}).  If $G_\bu$ is a
$2$-torus acting leafwise transitively on a foliation groupoid
$X_\bu$, we show that $X_\bu$ is isomorphic to an action groupoid
(Proposition~\ref{proposition;leafwisetrans}).

Throughout Section~\ref{section;hamiltonian-groupoid} $X_\bu$ denotes
a foliation groupoid and $G_\bu$ denotes a foliation $2$-group acting
strictly on $X_\bu$.  We denote by $\F$ the foliation of $X_0$ defined
by $X_\bu$, and by $(G,H,\partial,\alpha)$ the crossed module
associated to $G_\bu$.

\subsection{Leafwise transitive and regular actions}
\label{subsection;leafwise}

Let $x\in X_0$ and $h\in H$.  Let $f=h*u(x)\in X_1$.  It follows
from~\eqref{equation;natural1} that $s(f)=x$ and
$t(f)=\partial(h)\cdot x$.  This shows that the orbit of $x$ under the
$\partial(H)$-action is contained in the $X_\bu$-orbit of $x$.  We
call the $G_\bu$-action \emph{leafwise transitive} if this inclusion
is an equality, in other words if
\[\partial(H)\cdot x=t(s\n(x))\]
for all $x\in X_0$.  We call the action \emph{locally leafwise
  transitive} if for every $x\in X_0$ the image of the map $\h\to
T_xX_0$ defined by $\eta\mapsto(\partial(\eta)_{X_0})_x$ is equal to
$T_x\F$.

A leafwise transitive action is locally leafwise transitive.
Conversely, if the action is locally leafwise transitive and if
additionally $G_\bu$ and $X_\bu$ are both source-connected, then
$\partial(H)\cdot x=\F(x)=t(s\n(x))$ for all $x\in X_0$, so in
particular the action is leafwise transitive.

We call the $G_\bu$-action \emph{regular} if the following conditions
hold: $G_\bu$-action is locally leafwise transitive; the $G$-action on
$X_0$ is free; and if $h\in H$ satisfies $h*f=f$ for any $f\in X_1$,
then $h\in\ker(\partial)$.

\subsection{Regular form of the zero fibre}\label{subsection;zero}

In this subsection $(X_\bu,\omega_\bu,G_\bu,\mu_\bu)$ denotes a
Hamiltonian $G_\bu$-groupoid (as defined in
\S\,\ref{subsection;hamiltonian}).  The \emph{zero fibre} of $\mu_\bu$
is the subgroupoid
$\mu_\bu\n(0)=(\mu_1\n(0)\rightrightarrows\mu_0\n(0))$ of $X_\bu$.  We
say that $0$ is a \emph{regular value} of $\mu_\bu$ if
$0\in(\g_0/\h)^*$ is a regular value of $\mu_0\colon
X_0\to(\g_0/\h)^*$.

By~\cite[\S\,3.9]{lerman-malkin;deligne-mumford}, $\mu_\bu\n(0)$ is a
Lie subgroupoid of $X_\bu$ if $0$ is a regular value of $\mu_\bu$.

Suppose that $0$ is a regular value of $\mu_\bu$.  We define a
\emph{regular form} of the zero fibre $\mu_\bu\n(0)$ to be a pair
$(R_\bu,\phi_\bu)$ with the following properties: $R_\bu$ is a
foliation groupoid equipped with a regular $G_\bu$-action;
$\phi_\bu\colon R_\bu\to\mu_\bu\n(0)$ is a (strictly)
$G_\bu$-equivariant Morita morphism; and the $G_0$-orbits of $R_0$ are
the leaves of the null foliation of
$\phi_0^*\omega_0\in\Omega^2(R_0)$.

\begin{proposition}\label{proposition;invarianttransversal} 
Suppose that $0$ is a regular value of $\mu_\bu$.  Then there exists a
regular form $(R_\bu,\phi_\bu)$ of the zero fibre
$Z_\bu=\mu_\bu^{-1}(0)$.
\end{proposition}

\begin{proof}
Let $\F_0$ be the restriction of the foliation $\F$ to $Z_0$.  Let $D$
be the subbundle of $TZ_0$ spanned by $T\F_0$ and the fundamental
vector fields of $G$.  Then $\rank D=\rank T\F+\dim\g-\dim\h$ is
constant because $0$ is a regular value of $\mu$.  Let
$v\in\Gamma(T\F)$ and $\xi\in\g$.  The identity
\[
\iota_{[\xi_{Z_0},v]}\omega_0=
\ca{L}_{\xi_{Z_0}}\iota_v\omega_0-\iota_v\ca{L}_{\xi_{Z_0}}\omega_0=0-0=0
\]
shows that $[\xi_{Z_0},v]\in\Gamma(T\F)$, so $D$ is involutive.
Therefore $D=T\ca{D}$ for a unique foliation $\ca{D}$ of $Z_0$.  This
is the foliation denoted by $\g\ltimes\F_0$
in~\cite[\S\,2]{lin-sjamaar;riemannian}; its leaves are the $G$-orbits
$G\cdot\F_0(x)$ of leaves of $\F_0$.  Let $S\inj Z_0$ be a complete
transversal of $\ca{D}$ and define $\phi_0\colon G\times S\to Z_0$ by
$\phi_0(g,x)=g\cdot x\in Z_0$.  Then $\phi_0$ is transverse to $\F_0$
and is complete, so it follows from Lemma~\ref{lemma;transversal} that
the pullback groupoid $R_\bu=\phi_0^*Z_\bu$ is a Lie groupoid and that
the induced morphism $\phi_\bu\colon R_\bu\to Z_\bu$ is a Morita
morphism.  The object manifold of $R_\bu$ is $R_0=G\times S$.
Elements of $R_1$ are tuples $((j,x),f,(j',x'))\in R_0\times Z_1\times
R_0$ where $s(f)=j\cdot x$ and $t(f)=j'\cdot x'$.  Using the action of
the crossed module $(G,H,\partial,\alpha)$ on $Z_\bu$ we define an
action on $R_\bu$ as follows.
\begin{align*}
G\times R_0\longto R_0&\colon\quad g\cdot(j,x))=(gj,x)\\
G\times R_1\longto R_1&\colon\quad
g*((j,x),f,(j',x'))=((gj,x),g*f,(gj',x'))\\
H\times R_1\longto R_1&\colon\quad
h*((j,x),f,(j',x'))=((j,x),h*f,(\partial(h)j',x')).
\end{align*}
These actions satisfy the
conditions~\eqref{equation;functor1}--\eqref{equation;conjugation} and
so by Lemma~\ref{lemma;crossedactionequivalent} determine a strict
action of $G_\bu$ on $R_\bu$.  The morphism $\phi$ is
$G_\bu$-equivariant, the $G$-action on $R_0$ is free, and if $h\in H$
fixes any tuple $((j,x),f,(j',x'))\in R_1$, then $\partial(h)=1$.  To
check that the action of $G_\bu$ on $R_\bu$ is locally leafwise
transitive, we count dimensions.  Since the $G$-action on $R_0$ is
free, the $H$-action $y\mapsto\partial(h)\cdot y$ is locally free.
Hence for every $y\in R_0$ the infinitesimal orbit map $\h\to T_yR_0$
is injective, so the dimension of its image is equal to $\dim\h$.  On
the other hand, by Lemma~\ref{lemma;orbits} the image is contained in
the fibre at $y$ of the Lie algebroid $\Alg(R_\bu)$, which has rank
equal to $\dim R_1-\dim R_0=\dim\h$.  This proves that the
$G_\bu$-action on $R_\bu$ is regular.  Since $0$ is a regular value of
$\mu_\bu$, the foliation $\ca{D}$ of $Z_0$ is transversely symplectic,
so $(S,\omega_0|_S)$ is a symplectic manifold, and
$(R_0,\phi_0^*\omega_0)$ is a presymplectic manifold with null
foliation given by the $G$-orbits in $R_0$.
\end{proof}


\subsection{Foliation groupoids versus action groupoids}
\label{subsection;foliation-action}

Recall (Definition~\ref{definition;Liegroupbundle}) that a Lie group
bundle is a Lie groupoid where every arrow $g$ has $s(g)=t(g)$.  For a
group $K$ acting on a manifold $X$, let $K_x\subseteq K$ denote the
stabilizer of a point $x\in X$.

\begin{lemma}\label{lemma;foliation-action} 
Let $K$ be a simply connected Lie group acting locally freely on a
manifold $X$ and let $\F$ be the foliation of $X$ into $K$-orbits.
Let $L=\ker(K\to\Diff(X))$ be the kernel of the $K$-action.  Assume
that the set $X_L=\{\,x\in X\mid K_x=L\,\}$ is dense in $X$.
\begin{enumerate}
\item\label{item;monodromy}
The monodromy groupoid $\Mon(X,\F)$ is isomorphic to the action
groupoid $K\ltimes X$.
\item\label{item;holonomy}
The holonomy groupoid $\Hol(X,\F)$ is isomorphic to the action
groupoid $(K/L)\ltimes X$.
\item\label{item;between}
Every source-connected Lie groupoid $X_1\rightrightarrows X_0=X$
integrating $\F$ is isomorphic to one of the form $(K\ltimes X)/Z$,
where $Z\to X$ is an open Lie group subbundle of the Lie group bundle
$L\times X\to X$; and $X_1$ is Hausdorff if and only if $Z$ is closed.
\end{enumerate}
\end{lemma}

\begin{proof}
The Lie groupoids $K\ltimes X$ and $(K/L)\ltimes X$ are
source-connected, have Lie algebroid isomorphic to $T\ca{F}$, and
$K\ltimes X$ is source-simply connected.  Therefore $K\ltimes X$ is
isomorphic to $\Mon(X,\ca{F})$, which proves~\eqref{item;monodromy}.

Let $X_\bu$ be any source-connected Lie groupoid with base $X_0=X$ and
Lie algebroid $T\ca{F}$.  By Theorem~\ref{theorem;monhol} the holonomy
morphism $\Mon(X,\F)=K\ltimes X\to\Hol(X,\F)$ factors as two
surjective \'etale maps:
\begin{equation}\label{equation;groupoids}
\begin{tikzcd}
\hol\colon K\ltimes X\ar[r,"\psi_{X_\bu}"]&X_\bu\ar[r,"\hol_{X_\bu}"]&
\Hol(X,\F).
\end{tikzcd}
\end{equation}
Let us first take $X_\bu=(K/L)\ltimes X$.  To
establish~\eqref{item;holonomy} we need to prove that $\hol_{X_\bu}$
is an isomorphism, which amounts to showing that
$\ker(\hol)=\ker(\psi_{X_\bu})$.  The kernel of $\psi_{X_\bu}$ is the
Lie group bundle $L\times X$.  Consider any pair $(g,x_0)\in K\ltimes
X$ contained in the kernel of $\hol$.  The equation
$\hol(g,x_0)=u(x_0)$ means that $g\cdot x_0=x_0$ and, for a
sufficiently small section $S$ of the foliation $\ca{F}$ at $x_0$, for
every $x\in S$ the points $x$ and $g\cdot x$ are in the same plaque of
the foliation near $S$.  The map $\Phi\colon K\times S\to X$ that
sends $(g,x)$ to $g\cdot x$ is \'etale at $(1_K,x_0)$, so, after
shrinking $S$ if necessary, $\Phi$ restricts to a diffeomorphism
$\phi\colon U\times S\to V$, where $U$ is a neighbourhood of $1_K\in
K$ and $V$ a neighbourhood of $x_0\in X$.  The inverse $\phi\n$ is a
foliation chart at $x_0$.  The slice $g\cdot S$ of $\ca{F}$ is
transverse to the orbit $K\cdot x_0$ at $x_0$ and therefore (if $S$ is
small enough) $\phi\n(g\cdot S)\subseteq U\times S$ is the graph of a
unique smooth function $f\colon S\to U$.  Since $x$ and $g\cdot x$ are
in the same plaque we have $\phi\n(g\cdot x)\in U\times\{x\}$, which
is equivalent to $g\cdot x=f(x)\cdot x$.  Hence $f(x)\in gK_x$ for all
$x\in S$.  The set $X_L$, being preserved by the $K$-action and dense
in $X$, intersects $S$ in a dense set.  Thus $f(x)\in g^{-1}L$ for a
dense set of $x\in S$.  Since $L$ is discrete, the smooth map $f$ is
constant on $S$.  It follows that $f(x)=f(x_0)=1_K$ for $x\in S$, in
other words that $g\in L$ and $(g,x_0)\in L\times X$.  Thus
$\ker(\hol)=L\times X=\ker(\psi_{X_\bu})$, which
proves~\eqref{item;holonomy}.

Returning to a general Lie groupoid $X_\bu$ integrating $\ca{F}$ as in
the diagram~\eqref{equation;groupoids}, we see that
$X_\bu\cong(K\ltimes X)/Z$, where $Z=\ker(\psi_{X_\bu})$ is a group
subbundle of the Lie group bundle $\ker(\hol)=L\times X$.  The Lie
group bundles $Z$ and $L\times X$ are manifolds of dimension equal to
that of $X$, and therefore $Z$ is open in $L\times X$.  The
equivalence relation on $K\ltimes X$ defined by the action of $Z$ is
closed if and only if $Z$ is closed.  This
proves~\eqref{item;between}.
\end{proof}  

Specializing to the abelian case gives the following result.

\begin{proposition}\label{proposition;leafwisetrans}
Let $G_\bu$ be a $2$-torus.  Assume that $G_\bu$ and $X_\bu$ are
source-connected.  Also assume that $X_0$ is connected, $X_1$ is
Hausdorff, the action of $G_\bu$ on $X_\bu$ is leafwise transitive,
and the action
$\begin{tikzcd}[cramped,sep=small]
  H\ar[r,"\partial"]&G\ar[r]&\Diff(X_0)
\end{tikzcd}$
of $H$ on $X_0$ is locally free.  Then $X_\bu$ is isomorphic to the
action groupoid $H/Z\ltimes X_0$, where $Z$ is a subgroup of the
discrete subgroup $\ker(H\to\Diff(X_0))$ of $H$.  The action of
$G_1\cong H\times G_0$ on $X_1\cong H/Z\ltimes X_0$ is given by
$(h,g)*(kZ,x)=(hkZ,g\cdot x)$.
\end{proposition} 

\begin{proof}
Assume first that $H$ is simply connected.  Put $X=X_0$.  Let $L_G$ be
the kernel of the action $G\to\Diff(X)$ and $L_H=\partial^{-1}(L_G)$
the kernel of the action $H\to\Diff(X)$.  Since $H$ acts locally
freely, $L_H$ is a discrete subgroup of $H$.  Let $X_{L_G}=\{\,x\in
X\mid G_x=L\,\}$ and $X_{L_H}=\{\,x\in X\mid H_x=L\,\}$.  Then
$X_{L_G}\subseteq X_{L_H}$.  Since $G$ is a torus, it follows from the
principal orbit type theorem (see e.g.~\cite[\S\,IX.9,
  Theor\`eme~2]{bourbaki;groupes-algebres}) that $X_{L_G}$ is dense in
$X$.  Hence $X_{L_H}$ is dense in $X$.  Therefore we can apply
Lemma~\ref{lemma;foliation-action}\eqref{item;between} to the
$H$-action on $X$.  An open and closed Lie group subbundle of
$\ker(\hol)=L_H\times X$ is a trivial bundle $Z\times X$ for some
subgroup $Z$ of $L_H$, so we see that $X_\bu\cong H/Z\ltimes X_0$.
The identification $\Mon(X,\F)\cong H\ltimes X$ is $G_\bu$-equivariant
with respect to the action of $G_1=H\rtimes G$ on $H\ltimes X$ given
by $(h,g)\cdot(k,x)=(hk,g\cdot x)$.  This action descends to the
action $(h,g)*(kZ,x)=(hkZ,g\cdot x)$ on $H/Z\ltimes X_0\cong X_1$.  If
$H$ is not simply connected, we can apply the previous argument to the
crossed module $\tilde{H}\to G$, where $\tilde{H}$ is the universal
cover of $H$, to get that $X_1\cong\tilde{H}/\tilde{Z}\ltimes X_0$,
where $\tilde{Z}$ is a subgroup of $\tilde{H}$.  But the action of
$\tilde{H}$ on $X_1$ descends to an action of $H$, so
$N=\ker(\tilde{H}\to H)$ is a subgroup of $\tilde{Z}$.  Write
$Z=\tilde{Z}/N$; then $H/Z\ltimes X\cong\tilde{H}/\tilde{Z}\ltimes X$.
\end{proof}

\section{Symplectic reduction}\label{section;reduction}

The main result of this section is Theorem~\ref{theorem;reduction},
which is an extension of the Meyer-Marsden-Weinstein symplectic
reduction theorem to the setting of Hamiltonian $\G$-stacks, where
$\G$ is an \'etale Lie group stack.  The theorem is valid under the
assumption that $0$ is a regular value of the moment map
$\bmu\colon\X\to\Lie(\G)^*$, which ensures that the zero fibre
$\bmu^{-1}(0)$ is a differentiable stack.  The group stack $\G$ is not
required to be compact or separated, nor is it required to act freely
or properly.  A quotient stack $\bmu^{-1}(0)/\G$ then exists and is
symplectic, provided that a ``second-order'' freeness condition is
fulfilled.  This second-order condition holds automatically if $\G$ is
equivalent to a Lie group.  If it fails, the quotient does not exist
as a $1$-stack, although it might still exist as a higher-order stack.
To keep the size of this paper within reasonable limits we have
omitted any discussion of reduction at nonzero levels.  We draw the
reader's attention to the recent
preprint~\cite{battaglia-prato;nonrational-reduction}, which handles a
special case of our situation, namely symplectic reduction of toric
quasifolds by stacky tori.

\subsection{Symplectic reduction theorem}\label{subsection;reduction}

Let $\G$ be a connected \'etale Lie group stack (as defined in
\S\,\ref{subsection;group-stack}) and $(\X,\bomega,\G,\bmu)$ a
Hamiltonian $\G$-stack (as defined in
\S\,\ref{subsection;hamiltonian-stack}).  We denote the zero fibre
$\bmu\n(0)=\X\times_{\mu,\g^*,0}\star$ by $\stack{Z}$ and the natural
morphism $\stack{Z}\to\X$ by $\bi$.  We say that $0\in\Lie(\G)^*$ is a
\emph{regular value} of the moment map $\bmu$ if $0$ is a regular
value of the composite map $X_0\to\X\to\Lie(\G)^*$ for every atlas
$X_0\to\X$.  If $0$ is a regular value of $\bmu$, then the zero fibre
$\stack{Z}=\bmu\n(0)$ is an \'etale stack
by~\cite[\S\,3.9]{lerman-malkin;deligne-mumford}.

Assume that $0\in\Lie(\G)^*$ is a regular value of $\bmu$.  A
\emph{symplectic reduction (at $0$)} of $\X$ is a triple
$(\Y,\bp,\bomega_\Y)$ consisting of an \'etale stack $\Y$, a stack
morphism $\bp\colon\stack{Z}\to\Y$ which is a principal $\G$-bundle in
the sense of Definition~\ref{definition;principal-stack}, and a
symplectic form $\bomega_\Y\in\Omega^2(\Y)$ with the property
$\bp^*\bomega_\Y=\bi^*\bomega$.  As a consequence of~\cite[Theorem
  5.2]{bursztyn-noseda-zhu;principal-stacky-groupoids} and
Remark~\ref{remark;formsinject}, if a symplectic reduction exists it
is unique up to equivalence.

The goal of this section is to prove the following theorem, which
provides a necessary and sufficient condition for a symplectic
reduction to exist.  The theorem is formulated in terms of a
presentation of the Hamiltonian stack.
  
\begin{theorem}\label{theorem;reduction}
Let $\G$ be a connected \'etale Lie group stack and
$(\X,\bomega,\G,\bmu)$ a Hamiltonian $\G$-stack.  Assume that
$0\in\Lie(\G)^*$ is a regular value of $\bmu$.  Let $\B G_\bu\simeq\G$
be a presentation of $\G$ by a base-connected foliation $2$-group
$G_\bu$ with crossed module $(G,H,\partial,\alpha)$, and let
$(X_\bu,\omega_\bu,G_\bu,\mu_\bu)$ be a Hamiltonian groupoid
presenting $(\X,\bomega,\G,\bmu)$.  (Such presentations exists by
Corollary~\ref{corollary;compact-lie-group-stack} and
Theorem~\ref{theorem;stack-groupoid}\eqref{item;stack-groupoid}).  Let
$Z_\bu=\mu_\bu\n(0)$ and let $(R_\bu,\phi_\bu)$ be a regular form of
$Z_\bu$ (which exists by
Proposition~\ref{proposition;invarianttransversal}).  Then a
symplectic reduction of $\X$ exists if and only if the Lie group $H$
acts freely on the manifold $R_1$.  If $H$ acts freely on $R_1$, the
$0$-symplectic groupoid $(G\times^HR_\bu,\omega_\bu^\red)$ defined in
Lemma~\ref{lemma;reduction} below presents a symplectic reduction of
$\X$.
\end{theorem}

We will give the proof after establishing some preliminary results.

\begin{lemma}\label{lemma;reduction}
In the situation described in Theorem~\ref{theorem;reduction}, suppose
that $H$ acts freely on $R_1$.  Then the orbit space $R_1/H$ is a (not
necessarily Hausdorff) manifold and the projection $p\colon R_1\to
R_1/H$ is a principal $H$-bundle.  The associated bundle
$G\times^HR_1$ with fibre $G$ is the arrow manifold of a foliation
groupoid $G\times^HR_\bu=(G\times^HR_1\rightrightarrows R_0)$.  The
presymplectic form $\phi_0^*\omega_0\in\Omega^2(R_0)$ defines a
$0$-symplectic form $\omega_\bu^\red$ on $G\times^HR_\bu$.
\end{lemma}

\begin{proof}
The source map $s\colon R_1\to R_0$ is $H$-invariant.  Since the
action of $G_\bu$ on $R_\bu$ is locally leafwise transitive, the
kernel of $T_xs$ is precisely the span of the fundamental vector
fields at $x$.  The first assertion now follows from
Lemma~\ref{lemma;quotient} below.  The associated bundle $G\times^H
R_1$ is the quotient of $G\times R_1$ by the action
$h\cdot(g,f)=(g\partial(h\n),h*f)$.  Using the local trivializations
of $R_1\to R_1/H$, one gives $G\times^HR_1$ a smooth manifold
structure which makes the projection $G\times R_1\to G\times^HR_1$ a
surjective submersion.  Let $[g,f]\in G\times^HR_1$ denote the
equivalence class of the pair $(g,f)\in G\times R_1$.  We define the
groupoid $G\times^HR_\bu$ as follows.  For $[g,f]$, $[g',f']\in
G\times^HR_1$ and $x\in R_0$ put
\begin{gather*}
s[g,f]=s(f),\qquad t[g,f]=g\cdot t(f),\qquad u(x)=[1, u(x)],\\
[g,f]\circ[g',f']=[gg',((g')\n*f)\circ f'],\\
[g,f]\n=[g\n,(g*f)\n]=[g\n, g* f\n].
\end{gather*}
It follows
from~\eqref{equation;functor1}--\eqref{equation;conjugation} that
these structure maps are well defined.  Because $G$ is connected, the
form $\omega_0^\red=\phi^*\omega_0$ is $G$-invariant.  Since the
action of $H$ preserves source fibres, the form $\phi^*\omega_1$ is
$H$-invariant, and so descends to a form $\omega_1^\red$ on
$G\times^HR_1$.  The pair
$\omega_\bu^\red=(\omega_0^\red,\omega_1^\red)$ is a basic form on the
groupoid $G\times^HR_\bu$.  Since the $G$-orbits of $R_0$ are the
leaves of the null foliation of $\omega_0^\red$, the form
$\omega_\bu^\red$ is $0$-symplectic.
\end{proof}

Lemma~\ref{lemma;reduction} makes use of the following fact, which is
part of~\cite[Lemma~5.5]{moerdijk-mrcun;foliations-groupoids}.

\begin{lemma}\label{lemma;quotient}
Let $X$ be a (possibly non-Hausdorff) manifold and let $G$ be a Lie
group with a smooth free action $a\colon G\times X\to X$.  There is a
(necessarily unique) smooth (possibly non-Hausdorff) manifold
structure on orbit space $X/G$ such that the quotient map $X\to X/G$
is a principal $G$-bundle, if and only if there exist a (possibly
non-Hausdorff) manifold $Y$ and a smooth map $f\colon X\to Y$ which is
$G$-invariant and satisfies $\ker(Tf)_x=T_x(G\cdot x)$ for all $x\in
X$.
\end{lemma}

The next proposition completes one direction of the proof of
Theorem~\ref{theorem;reduction}.

\begin{proposition}\label{proposition;principal} 
In the situation described in Theorem~\ref{theorem;reduction}, suppose
that $H$ acts freely on $R_1$.  Let $G\times^HR_\bu$ be the Lie
groupoid described in Lemma~\ref{lemma;reduction}.  Then the Lie
groupoid morphism $\psi_\bu\colon R_\bu\to G\times^HR_\bu$ defined by
$\psi_0=\id_{R_0}$ and $\psi_1(f)=[1_G,f]$ is a principal
$G_\bu$-bundle in the sense of
Definition~\ref{definition;principal-groupoid}, and satisfies
$\psi_\bu^*\omega_\bu^\red=\phi_\bu^*\omega_\bu$.
\end{proposition}

\begin{proof}
The statement $\psi_\bu^*\omega_\bu^\red=\phi_\bu^*\omega_\bu$ holds
because
\[
\psi_0^*\omega^\red_0=\omega^\red_0=\phi_0^*\omega_0\in\Omega^2(R_0).
\]
To show that $\psi_\bu$ is a principal $G_\bu$-bundle we will verify
conditions~\eqref{item;principal1}--\eqref{item;principal3} of
Definition~\ref{definition;principal-groupoid}.
Condition~\eqref{item;principal1} is obvious: $\psi_\bu$ is
essentially surjective because $\psi_0=\id_{R_0}$.  To check
condition~\eqref{item;principal2} let $\pr_2\colon G_\bu\times
R_\bu\to R_\bu$ be the projection and $a_\bu\colon G_\bu\times
X_\bu\to X_\bu$ the action.  Recall that $G_0=G$ and $G_1=
H\rtimes_\alpha G$, and define
$\gamma\colon\psi_\bu\circ\pr_2\To\psi_\bu\circ a_\bu$ to be the map
\[
\gamma\colon G\times R_0\longto
G\times^HR_1,\qquad(g,x)\longmapsto[g,u(x)].
\]
Let $f\in R_1$ be an arrow in $R_\bu$ from $x$ to $y$, and let
$((h,g),f)\in G_1\times R_1$ be an arrow from $s((h,g),f)=(g,x)$ to
$t((h,g),f)=(\partial(h)g,y)$.  Then the following diagram in
$G\times^H R_\bu$ commutes:
\[
\begin{tikzcd}[column sep=7em,row sep=3em]
\psi_0\circ\pr_2(g,x)=x\ar[r,"{\gamma(g,x)=[g,u(x)]}"]
\ar[d,"{\psi_1\circ\pr_2((h,g),f)=[1,f]}"']&\psi_0\circ
a_0(g,x)=g\cdot x\ar[d,"{\psi_1\circ a_1((h,g),f)=[1,h*(g*f)]}"]
\\
\psi_0\circ\pr_2(\partial(h)g,y)=y
\ar[r,"{\gamma(\partial(h)g,y)=[\partial(h)g,u(y)]}"]&
\partial(h)g\cdot y.
\end{tikzcd}
\]
Thus $\gamma$ is a natural transformation.  It is automatically a
natural isomorphism, because $\Liegpd$ is a $(2,1)$-category.  The
coherence conditions on $\gamma$ are verified in a similar way.  This
shows that $\psi_\bu$ is $G_\bu$-invariant.  To check
condition~\eqref{item;principal3} we first describe the groupoid
\[
P_\bu:=R_\bu\times_{G\times^H R_\bu}^{(w)}R_\bu,
\]
which is the weak fibre product as described in
Definition~\ref{definition;fibreproductLiegpd}.  The object and arrow
manifolds are
\begin{gather*}
P_0=R_0\times_{R_0}(G\times^HR_1)\times_{R_0}R_0\cong G\times^HR_1,\\
P_1=R_1\times_{R_0}(G\times^HR_1)\times_{R_0}R_1.
\end{gather*}
The source and target maps of $P_\bu$ can be written
\[
s(r,[g,f],r')=[g,f],\qquad t(r,[g,f],r')=[g,(g\n*r')\circ f\circ r\n].
\]
We must show that the canonical morphism $\tau_\bu=(\pr_2,a_\bu)\colon
G_\bu\times R_\bu\to P_\bu$ is a Morita morphism.  On objects this is
the map $\tau_0\colon G_0\times R_0\to P_0$ given by
\[\tau_0(g,x)=\gamma(g,x)=[g,u(x)],\]
and on arrows this is the map $\tau_1\colon(H\rtimes_\alpha G)\times
R_1\to P_1$ given by
\[
\tau_1((h,g),f)=(f,\gamma(s((h,g),f)),h*(g*f))=(f,[g,u(s(f))],h*(g*f)).
\]
First we show that $\tau_\bu$ is essentially surjective, i.e.\ the map
\begin{align*}
t\circ\pr_1\colon P_1\times_{s,P_0,\tau_0}(G\times R_0)&\longto P_0\\
((r,[g,u(s(r))],r'),(g,u(s(r))))&\longmapsto[g,(g\n * r')\circ r\n]
\end{align*}
is a surjective submersion.  Since the action of $H$ on $R_1$
preserves source fibres and is free, whenever $[g,u(x)]=[g',u(x')]\in
G\times^HR_1$ we must have $g=g'$ and $x=x'$.  It follows that we have
a diffeomorphism $\lambda\colon P_1\times_{P_0}(G\times
R_0)\cong(R_1\times G)\times_{R_0}R_1$ given by
\[\lambda((r,[g,u(s(r))],r'),(g,u(s(r))))=(r,g,r'),\]
where
\[
(R_1\times G)\times_{R_0}R_1=\{\,((r,g,r')\in R_1\times G\times
R_1\mid g\cdot s(r)=s(r')\,\}.
\]
Under the isomorphism $\lambda$ the map $t\circ\pr_1$ becomes the map
$\kappa\colon(R_1\times G)\times_{R_0}R_1\to P_0=G\times^HR_1$ given
by
\[\kappa(r,g,r')=[g,(g\n * r')\circ r\n].\]
For surjectivity, let $[g,f]\in P_0$.  Then
$(f\n,g,g*u(t(f)))\in(R_1\times G)\times_{R_0}R_1$ and
$\kappa(f\n,g,g*u(t(f)))=[g,f]$.  For submersivity, note that there is
a $G$-action on $(R_1\times G)\times_{R_0} R_1$ given by
\[g\cdot(r,g',r')=(g*r,g'g\n, r')\]
and also a $G$-action on $P_0=G\times^HR_1$ given by
\[g\cdot [g',f] =[g'g\n,g*f].\]
The map $\kappa$ is $G$-equivariant.  So it suffices to show that, for
each $[1,f]\in P_0$ and for any $(r,1,r')\in (R_1\times G)\times_{R_0}
R_1$ with $r'\circ r\n= f$, there is a local section $\sigma\colon
P_0\to (R_1\times G)\times_{R_0} R_1$ of $\kappa$ with
$\sigma[1,f]=(r,1,r')$.  But this follows from the fact that the
multiplication map of $R_\bu$ is a submersion.  Therefore $\tau_\bu$
is essentially surjective.  Next we show $\tau_\bu$ is fully faithful.
Consider the fibred product
\[
M=\bigl((G\times R_0)\times(G\times R_0)\bigr)\times_{P_0\times
  P_0}P_1
\]
with respect to the maps $\tau_0\times\tau_0\colon(G\times
R_0)\times(G\times R_0)\to P_0\times P_0$ and $(s,t)\colon P_1\to
P_0\times P_0$.  A typical element of $M$ is a tuple
\[
\bigl((g,s(r)),(g',x),(r,[g,u(s(r))],r')\bigr)\in(G\times
R_0)\times(G\times R_0)\times P_1
\]
satisfying
\[
[g',u(x)]=t(r,[g,u(s(r))],r')=\bigl[g,(g\n * r')\circ r\n\bigr],
\]
where $x=t(r)$ because $h*u(x)=(g\n*r')\circ r\n$ for some $h\in H$,
and because the $H$-action preserves the $s$-fibres.  The universal
property of the fibred product $M$ yields a canonical map
$\chi\colon(H\rtimes_\alpha G)\times R_1\to M$ given by
\[
\chi(h,g,f)=
\bigl((g,s(f)),(\partial(h)g,t(f)),(f,[g,u(s(f))],h*(g*f))\bigr).
\]
We must show $\chi$ is a diffeomorphism.  For $r$, $r'\in R_1$ in the
same $H$-orbit, let $\delta(r,r')$ be the unique element $h\in H$
satisfying $hr=r'$.  The map $\delta\colon R_1\times_{R_1/H}R_1\to H$
is smooth, because $R_1\to R_1/H$ is a principal $H$-bundle.  Define
$\zeta\colon M\to(H\rtimes_\alpha G)\times R_1$ by
\[
\zeta\bigl((g,s(r)),(g',x),(r,[g,u(s(r))],r')\bigr)=(\delta(g*r,r'),g,r).
\]
We assert that $\zeta$ is the inverse of $\chi$.  Indeed,
\begin{align*}
(\zeta\circ\chi)(h,g,f)&=
  \zeta\bigl((g,s(f)),(\partial(h)g,t(f)),(f,[g,u(s(f))],h*(g*f))\bigr)\\
&=(\delta(g*f,h*(g*f)),g,f)\\
&=(h,g,f),
\end{align*}
and
\begin{multline*}
(\chi\circ\zeta)\bigl((g,s(r)),(g',t(r)),(r,[g,u(s(r))],r')\bigr)
\\
\begin{aligned}
&=\chi(\delta(g*r,r'),g,r)
\\
&=\bigl((g,s(r)),(\partial(\delta(g*r,r'))g,t(r)),
(r,[g,u(s(r))],\delta(g*r,r')*(g*r))\bigr)
\\
&=\bigl((g,s(r)),(\partial(\delta(g*r,r'))g,t(r)),
(r,[g,u(s(r))],r')\bigr).
\end{aligned}
\end{multline*}
It remains to show that $\partial(\delta(g*r,r'))g=g'$.  We have
$\partial(\delta(g*r,r'))g\cdot t(r) = t(r')$
by~\eqref{equation;natural1}.  And $g'\cdot t(r)=t(r')$ by the
definition of $M$.  Since $G$ acts freely on $R_0$, we have
$\partial(\delta(g*r,r'))g=g'$.  So $\zeta=\chi\n$, and $\tau_\bu$ is
fully faithful.
\end{proof}

\begin{proof}[Proof of Theorem~\ref{theorem;reduction}]
If $H$ acts freely on $R_1$, Propositions~\ref{proposition;principal}
and~\ref{proposition;principal-groupoid-stack} show that a symplectic
reduction of $\X$ exists.  Conversely, suppose that $H$ does not act
freely on $R_1$.  It follows
from~\cite[Theorem~5.2]{bursztyn-noseda-zhu;principal-stacky-groupoids}
that $\B R_\bu\to\stack{S}$ is a principal $\B G_\bu$-bundle over some
stack $\stack{S}$ if and only if the Lie groupoid
\[
Y_\bu:=(R_0\times R_0)\times_{R_\bu\times R_\bu}^{(w)}(G_\bu\times
R_\bu)
\]
is Morita equivalent to a manifold.  Here the weak fibred product is
taken over the canonical map $R_0\times R_0\to R_\bu\times R_\bu$ and
the projection-action map ${\pr}_2\times a_\bu\colon G_\bu\times
R_\bu\to R_\bu\times R_\bu$.  We will show that $Y_\bu$ has
non-trivial isotropy groups and so cannot be Morita equivalent to a
manifold.  Choose $f\in R_1$ and $1\ne h\in H$ so that $h*f=f$.  Let
$x=s(f)$ and $y=t(f)$. Consider the point
\[
z=((x,x),(f,f),(1,y))\in (R_0\times R_0)\times_{R_0\times R_0}
(R_1\times R_1)\times_{R_0\times R_0}(G_0\times R_0).
\]
Let $k\in Y_1$ be the arrow
\[
k=((x,x),(f,f),((h,1),u(y)))\in(R_0\times R_0)\times_{R_0\times
  R_0}(R_1\times R_1)\times_{R_0\times R_0}(G_1\times R_1),
\]
where we identify $G_1=H\rtimes_\alpha G$.  Then $s(k)=z$ and
\[
t(k)=((x,x),(u(y),h*u(y))\circ(f,f),(\partial(h),y))=
((x,x),(f,f),(\partial(h),y)).
\]
But since $R_\bu$ is a regular form and $h*f=f$, we have
$\partial(h)=1$.  So $t(k)=z$.  Therefore the isotropy group of $z$ is
nontrivial.
\end{proof}

\subsection{Symplectic reduction by a Lie group}
\label{subsection;group-reduction}

Theorem~\ref{theorem;reduction} gives new information even in the
context of symplectic reduction by an ordinary Lie group.

\begin{corollary}\label{corollary;reduction}
Let $G$ be a Lie group and $(X,\omega)$ a symplectic manifold on which
$G$ acts in a Hamiltonian fashion with moment map $\mu\colon
X\to\g^*$.  If $0$ is a regular value of $\mu$, then $\mu^{-1}(0)/G$
is a symplectic stack.
\end{corollary}

An \emph{orbifold} is a separated \'etale stack.  If in
Corollary~\ref{corollary;reduction} we make the extra assumption that
$G$ acts properly on $\mu^{-1}(0)$, then the stack $\mu^{-1}(0)/G$ is
separated, so we obtain the following familiar Meyer-Marsden-Weinstein
reduction theorem~\cite{meyer;symmetries-integrals},
\cite{marsden-weinstein;reduction-symplectic-manifolds-symmetry},~%
\cite{marsden-weinstein;comments-symplectic-reduction}.

\begin{corollary}\label{corollary;orbifold}
Let $G$ be a Lie group and $(X,\omega)$ a symplectic manifold on which
$G$ acts in a Hamiltonian fashion with moment map $\mu\colon
X\to\g^*$.  If $0$ is a regular value of $\mu$ and if $G$ acts
properly on $\mu^{-1}(0)$, then $\mu^{-1}(0)/G$ is a symplectic
orbifold.  If in addition the action of $G$ on $\mu^{-1}(0)$ is free,
then $\mu^{-1}(0)/G$ is a symplectic manifold.
\end{corollary}

Here are three instances of Corollary~\ref{corollary;reduction}.

\begin{example}[space of geodesics]\label{example;geodesic}
Let $X=T^*M$ be the cotangent bundle of a complete Riemannian manifold
$M$.  The space of (non-parametrized) geodesics of $M$ is the
symplectic quotient of $X$ by the Hamiltonian action of $G=\R$
generated by the kinetic energy (norm square) function.
Corollary~\ref{corollary;reduction} says that the space of geodesics
can be interpreted as a symplectic stack.
\end{example}

\begin{example}[space of Reeb orbits]\label{example;contact}
Let $M$ be a manifold equipped with a contact form $\alpha$ and let
$X$ be the symplectization of $M$, i.e.\ the manifold
$X=(0,\infty)\times M$ equipped with the symplectic form
$d\alpha+dt\wedge\alpha$.  Suppose the Reeb vector field of $M$ is
complete.  Then the space of Reeb orbits of $M$ is the symplectic
quotient of $X$ by the Hamiltonian action of $G=\R$ generated by the
function $(t,m)\mapsto t$.  Corollary~\ref{corollary;reduction} says
that the space of Reeb orbits can be interpreted as a symplectic
stack.
\end{example}

\begin{example}[toric quasi-folds]\label{example;quasitoric}
Consider $\C^n$ with its standard symplectic form $\omega$ and the
standard action of the torus $G=\T^n$ as in Example~\ref{example1}.
Let $N\subseteq G$ an immersed Lie subgroup and
$\iota^*\circ\mu\colon\C^n\to\lie{n}^*$ the $N$-moment map with zero
fibre $X_0$.  Let $\tilde{N}\to N$ a covering homomorphism as in
Example~\ref{example4}.  Let $(X_\bu,\omega_\bu)$ be the
$0$-symplectic groupoid of Example~\ref{example8} and let
$(\X,\bomega)$ be the associated symplectic stack.  There is an
obvious morphism of groupoids $p_\bu$ from (the identity groupoid of)
$X_0$ to the action groupoid $X_\bu=\tilde{N}\ltimes X_0$.  The
associated morphism of stacks $\B p_\bu \colon X_0\to\X$ is a
principal $\tilde{N}$-bundle in the sense of
Definition~\ref{definition;principal-stack}.  By definition the
pullback of $\bomega$ to $X_0$ is equal to the presymplectic form
$\omega_0$ on $X_0$.  We conclude that the toric quasifold
$(\X,\bomega)$ is the symplectic reduction of $\C^n$ with respect to
$\tilde{N}$.  The symplectic stack $(\X,\bomega)$ is a symplectic
orbifold if and only if $N$ is a closed subgroup of $G$ and the
covering $\tilde{N}\to N$ is finite.  The action of $G=\T^n$ on $\C^n$
descends to an action of the quotient Lie group stack $\G=G/\tilde{N}$
on $\X$, which is nothing other than the $\G$-action defined in
Example~\ref{example8}.
\end{example}


%

\section{The Duistermaat-Heckman theorem}\label{section;dh}

\numberwithin{equation}{section}

In this section we prove an analogue of the Duistermaat-Heckman
theorem for Hamiltonian $\G$-stacks, where $\G$ is a stacky torus.
The Duistermaat-Heckman theorem has two parts: (1)~the variation of
the reduced symplectic form is linear, and (2)~the moment map image of
the Liouville measure is piecewise polynomial.  It is only the first
part that we generalize here; when the Hamiltonian stack is not
proper, it is unclear how to integrate the Liouville measure along
fibres of the moment map in a canonical, Morita-invariant fashion.
(See~\cite{crainic-mestre;measures-stacks} for a treatment of measures
and densities on differentiable stacks.)

The following version of the Duistermaat-Heckman theorem was obtained
by Guillemin and Sternberg by applying the coisotropic embedding
theorem to the zero fibre of the moment map.  Our approach is to
generalize this formulation to Hamiltonian $\G$-stacks.  We focus our
attention on the situation when we have a presentation of our
$\G$-stack $\X$ by a Hamiltonian groupoid with a leafwise transitive
action (Theorem~\ref{theorem;DH}).

\begin{theorem}%
[Duistermaat-Heckman~\cite{duistermaat-heckman;variation},
  Guillemin-Sternberg~\cite{guillemin-sternberg;birational}]
\label{theorem;DHoriginal}
Let $G$ be a torus and let $(M,\omega,G,\mu)$ be a connected
Hamiltonian $G$-manifold, where $\mu$ is a proper map.  Let $U$ be an
open neighborhood of $0\in \g^*$ which consists of regular values of
$\mu$, and let $Z=\mu\n(0)$.  Let $\theta\in\Omega^1(Z)\otimes\g$ be a
connection form for the locally free action of $G$ on $Z$, and define
the $1$-form $\gamma\in\Omega^1(Z\times\g^*)$ by
\begin{equation}\label{eqn;littlegamma}
\gamma(v|_p,w|_\beta)=\langle\beta,\theta|_p
(v|_p)\rangle.\end{equation} Then:
\begin{enumerate}
\item\label{item;DH1original} After possibly shrinking $U$, there is
  an isomorphism of Hamiltonian $G$-manifolds
\[
(\mu\n(U),\omega|_{\mu\n(U)},G,\mu|_{\mu\n(U)})\cong (Z\times U,
\omega|_Z + d\gamma,G,\pr_2),
\]
where the $G$-action on $Z\times U$ is $t\cdot(z,u)=(t\cdot z, u)$.
\item
Let $u\in U$. If the action of $G$ on $Z$ is free, there is a
symplectic isomorphism of reduced spaces
$(\mu\n(u)/G,\omega^{\red(u)})\cong (\mu\n(0)/G,\omega^{\red(0)}
+\langle u,\Gamma\rangle)$, where $\omega^{\red(u)},\omega^{\red(0)}$
denote the symplectic forms on the reduced spaces at $u$ and $0$, and
$\Gamma\in \Omega^2(Z/G)\otimes\g$ is the curvature $2$-form for the
principal $G$-bundle $Z\to Z/G$.
\item\label{item;DH3original}
The de Rham cohomology class $[\langle u,\Gamma\rangle]$ varies
linearly with $u$ and does not depend on the choice of connection
$\theta$ or the choices involved in constructing the isomorphism
in~\eqref{item;DH1original}.
\end{enumerate}
\end{theorem}

We expand on the second part of~\eqref{item;DH3original}.  In
constructing the isomorphism in~\eqref{item;DH1original}, one views
$\mu\colon\mu\n(U)\to U$ as a $G$-invariant locally trivial fibre
bundle.  To construct a trivialization $\mu\n(U)\cong Z\times U$, one
chooses a $G$-invariant connection on this fibre bundle and a smooth
contraction of $U$ to $0$.  The horizontal lifts of this contraction
gives a $G$-equivariant diffeomorphism from $\mu\n (0)\times Z$ to
$\mu\n(U)$.  The isotopy class of this diffeomorphism is independent
of the contraction and of the connection on $\mu\n (U)$.

\begin{theorem}\label{theorem;DH}
Let $\G$ be a stacky torus, and let $\B G_\bu\simeq \G$ be a
presentation by a Lie $2$-group with the property that $G_0$ is a
torus.  Let $(\X,\bomega,\G,\bmu)$ be a Hamiltonian $\G$-stack
presented by $(X_\bu,\omega_\bu,G_\bu,\mu_\bu)$.  Denote by
$\partial\colon H\to G$ the crossed module of $G_\bu$.  Assume the
following:
\begin{enumerate}\renewcommand{\theenumi}{\alph{enumi}}
\renewcommand{\labelenumi}{(\theenumi)}
%
\item\label{hypa}
$X_1$ is Hausdorff and $0\in \Lie(\G)^*$ is a regular value of $\mu$;
\item\label{hypb}
The action of $H$ on $X_0$ is locally free;
\item\label{hypc}
The action of $G_\bu$ on $X_\bu$ is leafwise transitive;
\item\label{hypd}
The moment map $\mu_0\colon X_0\to\ann(\lie{h})$ is proper, and
$X_0=X$ is connected.
\end{enumerate}
Then:
\begin{enumerate}
\item\label{item;DH1}
There is an open neighborhood $U$ of $0$ and an isomorphism of
Hamiltonian $\G$-stacks
\[
(\bmu\n(U),\bomega|_{\mu\n(U)},\G,\bmu)\cong (\bmu\n(0)\times U,
\bomega|_{\bmu\n(0)} + d\gamma|_{Z\times U},\G,\pr_2).
\]
where $d\gamma|_{Z\times U}$ is depends on a choice and is described
in~\eqref{eqn;DHgpd} below.
\item\label{item;DH2}
If the symplectic reduction of $\X$ exists at $0$, then it exists at
all points of $U$.  For all $u\in U$, there is an equivalence of
symplectic stacks of the reduced spaces
\[
(\bmu\n(u)/\G,\bomega^{\red(u)})\cong (\bmu\n(0)/\G,\bomega^{\red(0)}
+\Gamma)
\]
where $\bomega^{\red(u)},\bomega^{\red(0)}$ denote the symplectic
forms on the reduced spaces at $u$ and $0$, respectively, and
$\Gamma\in \Omega^2(\bmu\n(0)/\G)$ is described
before~\eqref{eqn;GammaForm} below.
\item\label{item;DH3}
The form $\Gamma$ varies linearly with $u\in \Lie(\G)^*$, and, after
fixing the presentation $X_\bu$ of $\X$, its cohomology with respect
to the complex $\Omega^\bu (\bmu\n(0)/\G)$ does not depend on the
choices involved.
\end{enumerate}
\end{theorem}

We will prove the theorem after some initial results.  For simplicity
we will assume that $H=N=N(\F)$, the null subgroup of $\F$; the proof
carries to other $H$ with minor changes.  Then $\h=\nnn=\nnn(\F)$, the
null ideal, and $\mu_0(X_0)\subseteq\ann(\nnn)$.

\begin{lemma}\label{lemma;DHbasicfacts} 
Assume we are in the context of Theorem~\ref{theorem;DH}.
\begin{enumerate}
\item\label{item;form}
There is an isomorphism of Lie groupoids $X_\bu\cong (N/L\ltimes
X)\rightrightarrows X$, where $L$ is a discrete subgroup of $N$, and
where the action of $G_\bu$ is given as in
Proposition~\ref{proposition;leafwisetrans}.  Under this isomorphism,
the $0$-symplectic form $\omega=(\omega_0,\omega_1)$ can be written
$(\omega_0,0\oplus\omega_0)$.
\item\label{item;regular-open}
There is an open set $U$ containing $0$ which consists of regular
values of $\mu_0$.
\item\label{item;locally-free}
The action of $G$ is then locally free on $V:=\mu_0\n(U)$.
\end{enumerate}
\end{lemma}

\begin{proof}
Item~\eqref{item;form} follows immediately from
Proposition~\ref{proposition;leafwisetrans} and the condition
$\omega_1=s^*\omega_0$.  Item~\eqref{item;regular-open} holds because
$\mu_0$ is proper.  Item~\eqref{item;locally-free} follows
from~\eqref{item;regular-open} because the action of $N$ is locally
free on $X_0$.
\end{proof}

Recall the notion of the \emph{symplectization} of a presymplectic
manifold; see for instance~\cite{lin-sjamaar;presymplectic}.  In the
context of Theorem~\ref{theorem;DH}, for the presymplectic manifold
$(X,\omega_0)$, let $T^*\F$ be the vector bundle dual to $T\F$, and
let $\pr\colon T^*\F\to X$ be the bundle projection.  By choosing a
$G$-invariant metric on $X$, one can embed $j\colon
T^*\F\hookrightarrow T^*X$.  Let $\tilde{\omega}$ be the standard
symplectic form on the cotangent bundle $T^*X$, and let
$\Omega=\pr^*\omega_0+j^*\tilde{\omega}$ be the $2$-form on $T^*\F$.
Then $\Omega$ is symplectic near the zero section $X\to T^*\F$, which
is a coisotropic embedding of $X$.  There is a moment map
$\Psi=\pr^*\mu_0 + j^*\tilde{\mu}$ for the $G$-action on $T^*\F$,
where $\tilde{\mu}\colon T^*X\to\g^*$ is the standard moment map for
the $G$-action on $T^*X$ given by
\[
\tilde{\mu}^\xi(y) =\langle y,\xi_X|_x\rangle,\qquad\text{for } y\in
T^*_x X.
\]
The germ at $X$ of $T^*\F$ is called a symplectization of $X$.
In~\cite{lin-sjamaar;presymplectic} it is shown that in the leafwise
transitive case, fibres of the moment map $\mu_0\colon X\to
\ann(\h)\subseteq\g^*$ are fibres of $\Psi\colon T^*\F\to\g^*$.

\begin{lemma} \label{lemma;propermoment}
Assume we are in the context of Theorem~\ref{theorem;DH}. Let $(T^*\F,
\Omega,G,\Psi)$ be the symplectization of $(X,\omega_0,G_0,\mu_0)$
just described, and let $U$ and $V$ be sets as in
Lemma~\ref{lemma;DHbasicfacts}.  The restriction of the moment map
$\Psi\colon T^*\F|_{V}\to\g^*$ to $V=\mu_0\n(U)$ is proper.
\end{lemma}

\begin{proof}
Choose a splitting $\g=\kk\oplus\nnn$.  Because the action of $G$ is
locally free on $V$, the tangent bundle can be split $TV\cong
\kk\oplus\nnn\oplus(TU/\g)$, and because the action is leafwise
transitive, $T\F|_{V}$ is isomorphic to the trivial bundle with fibre
$\nnn$.  Let us choose the $G$-invariant metric on $X$ so that $\kk$
is orthogonal to $\nnn$, as follows. Fix a metric on $TV$, by (1)
choosing a basis $\xi_i$ of $\kk$ and a basis $\nu_j$ of $\nnn$ and
declaring the corresponding sections of $TV$ to be orthonormal, and by
(2) extending to get a metric on $TV$ by choosing an arbitrary smooth
metric on $TV/\g$.  Averaging this metric over $G$ gives a
$G$-invariant metric on $V$ where $\kk$ is orthogonal to $\nnn$, as
desired.

Now consider the moment map $\Psi$.  Under the splitting
$\g^*=\kk^*\oplus\nnn^*$, the description of $\mu_0$ and the choice of
splitting gives
\[
\pr^*\mu_0(T\F^*|_{V})=U\oplus 0,\qquad j^*\tilde{\mu}(T^*\F|_{V}) =
0\oplus\nnn^*.
\]
Because the action of $N$ is locally free, the restriction of
$j^*\tilde{\mu}$ to a fibre $T^*\F|_x$ is a linear isomorphism to
$\nnn^*$.

To show that the restriction of $\Psi$ is proper, it is enough to show
that the preimage of $D\times E\subseteq\kk^*\oplus\nnn^*$ in
$T^*\F|_U$ is compact, where $D\subseteq\kk^*$ and $E\subseteq\nnn^*$
are compact.  Indeed, from the description of $\Psi$ above, the
preimage of $D\times E$ is homeomorphic to a fibre bundle over
$\mu_0\n(D)$ with fibre $E$.  Since $\mu_0$ is proper this space is
compact.
\end{proof}

\begin{proof}[Proof of Theorem~\ref{theorem;DH}]
Let $U$ and $V$ be as in Lemma~\ref{lemma;DHbasicfacts}. By
Lemma~\ref{lemma;propermoment}, we can apply
Theorem~\ref{theorem;DHoriginal} to the symplectization $T^*\F|_V$ of
$V\subseteq X$.  Let $Z=\mu_0\n(0)=\Psi\n(0)$, and let
\begin{equation}\label{eqn;choosealpha}
\theta\in\Omega^1(Z)\otimes\g
\end{equation}
be a connection form for the action of $G$ on $Z$, and let $\gamma\in
\Omega^1(Z\times\g^*)$ be as in Theorem~\ref{theorem;DHoriginal}.
After possibly shrinking $U$, there is an open neighborhood $U'$ of
$0$ in $\g^*$ so that $U'\cap\ann(\nnn)=U$ and an isomorphism of
Hamiltonian $G$-manifolds
\[
(\Psi\n(U'),\Omega|_{\Psi\n(U')}, G,\Psi|_{\Psi\n(U')})\cong (Z\times
U',\omega_0|_Z+d\gamma, G,\pr_2).
\]
We require that we choose the contraction of $U'$ described after
Theorem~\ref{theorem;DHoriginal} so that it restricts to a contraction
of $U=U'\cap\ann(\nnn)$.  Restricting the isomorphism to
$\Psi\n(U)=\mu_0\n(U)=V$ gives an isomorphism of presymplectic
Hamiltonian $G$-manifolds
\begin{equation}\label{eqn;DHpresym}
F\colon (V,\omega_0|_V,G,\mu_0)\cong (Z\times
U,\omega_0|_Z+(d\gamma)|_{Z\times U}, G,\pr_2).
\end{equation}
By Lemma~\ref{lemma;DHbasicfacts}\eqref{item;form} we can lift this to
an isomorphism of Hamiltonian $G_\bu$-groupoids
\begin{equation}\label{eqn;DHgpd}
F_\bu\colon (X_\bu|_V,\omega_\bu,G_\bu,\mu_\bu)\cong(X_\bu|_Z\times
U,\omega_\bu|_Z+(d\gamma)|_{Z\times U},G_\bu,\pr_2).
\end{equation}
Note that, by~\eqref{eqn;DHpresym}, the form $(d\gamma)|_{Z\times
  U}\in\Omega^2(Z\times U)$ determines an element of
$\Omega^2_\bas(X_\bu|_Z\times U)\cong\Omega^2(\B X_\bu|_Z\times U)$
which we have also denoted by $(d\gamma)|_{Z\times U}$
in~\eqref{eqn;DHgpd}.  Applying
Theorem~\ref{theorem;stack-groupoid}\eqref{item;groupoid-stack}
proves~\eqref{item;DH1}.

The first statement of~\eqref{item;DH2} follows immediately.  For the
second, since $\omega|_{Z}$ and $\omega|_Z+(d\gamma)|_{Z\times\{u\}}$
descend to symplectic forms on the stack $\bmu\n(0)/\G$, it follows
that $(d\gamma)|_{Z\times\{u\}}$ descends to some
\begin{equation}\label{eqn;GammaForm}
\Gamma\in\Omega^2(\bmu\n(0)/\G).
\end{equation}
This proves~\eqref{item;DH2}.  Finally, that $\Gamma$ varies linearly
with $u\in\Lie(\G)^*$ is obvious from the definition.  It remains to
check that its cohomology class does not depend on the choice of
connection $1$-form $\theta\in \Omega^1(Z)\otimes\g$ for the action of
$G$ on $Z$ or on the choice of isomorphism $X_\bu|_V\cong
X_\bu|_Z\times U$.

Let $\theta'\in\Omega^1(Z)\otimes\g$ be another connection form.  Then
$\theta-\theta'\in\Omega^1(Z)\otimes\g$ is a $G$-invariant horizontal
$1$-form.  If $\gamma'$ is related to $\theta'$ as
in~\eqref{eqn;littlegamma} then $\gamma-\gamma'\in
\Omega^1(Z\times\g^*)$ is $G$-invariant.  By the description of
$X_\bu$ in Lemma~\ref{lemma;DHbasicfacts} we consider
$(\gamma-\gamma')|_{Z\times U}$ as a $G$-invariant element of
$\Omega^1_\bas(X_\bu|_Z\times U)$.  So from
Lemma~\ref{lemma;exactbasic} below, the restriction of
$\gamma-\gamma'$ to $X_\bu|_Z\times\{u\}$ descends to an element of
$\Omega^1(\bmu\n(0)/\G\times\{u\})$.  The form $\Gamma'\in
\Omega^2(\bmu\n(0)/\G\times\{u\})$ determined by $\gamma'$ then
differs from $\Gamma$ by an exact form.

Let us now assume we choose a different isomorphism $F'\colon V\cong
Z\times U$ from~\eqref{eqn;DHpresym}.  Then $F'\circ F\n\colon Z\times
U\to Z\times U$ is $G$-equivariant and isotopic the identity.  So
$(F'\circ F\n)^* (\omega_0|_Z+(d\gamma)|_{Z\times U})
-(\omega_0|_Z+(d\gamma)|_{Z\times U})=d\eta$, where $\eta\in
\Omega^1(Z\times U)$ is $G$-invariant.
We view $\eta$ as a $G$-invariant element of
$\Omega_\bas^1(X_\bu|_Z\times U)$ as before and apply
Lemma~\ref{lemma;exactbasic}.  This proves~\eqref{item;DH3}.
\end{proof}

\begin{lemma}\label{lemma;exactbasic} 
Assume we are in the context of Theorem~\ref{theorem;DH}, and let $U$
and $V$ be as in Lemma~\ref{lemma;DHbasicfacts}.  Let
$(R_\bu,\phi_\bu)$ be a regular form of the zero fibre
$Z_\bu=\mu_\bu^{-1}(0)$ as in
Proposition~\ref{proposition;invarianttransversal} and let
$G\times^HR_\bu$ be the $0$-symplectic groupoid of
Lemma~\ref{lemma;reduction}.  Identify $X_\bu|_V\cong Z_\bu\times U$
as in~\eqref{eqn;DHgpd}.  If $\eta_\bu\in\Omega_\bas^k(Z_\bu\times U)$
is $G$-invariant, then $\phi_\bu^*\eta_\bu$ descends to an element of
$\Omega_\bas^k(G\times^HR_\bu)\cong\Omega^k(\bmu\n(0)/\G)$.
\end{lemma}

\begin{proof}
Consider the form
$\phi_\bu^*\eta_\bu=(\phi_0^*\eta_0,\phi_1^*\eta_1)\in
\Omega_\bas^k(R_\bu)$.  We will show that
$(\phi_0^*\eta_0,0\oplus\phi_1^*\eta_1)\in
\Omega_\bas^k(G\times^HR_1\rightrightarrows R_0)$.  Consider the maps
\[
\pr_2\circ({\id}\times s)\colon G\times R_1\to R_0,\qquad
a_0\circ({\id}\times t)\colon G\times R_1\to R_0,
\]
which under the projection $G\times R_1\to G\times^{H} R_1$ descend to
the source and target maps of the groupoid $G\times^{H}R_1$.  It
suffices to show that
\[
(\pr_2\circ({\id}\times s))^*(\phi_0^*\eta_0)=0\oplus\phi_1^*\eta_1=
(a\circ({\id}\times t))^*(\phi_0^*\eta_0).
\]
Since $\eta$ is basic, the first equality is obvious, and the second
follows immediately from the fact that $\eta_0$ is $G$-invariant and
that $\phi$ is $G_\bu$-equivariant.
\end{proof}

\begin{remark}
The cohomology class $[\Gamma]$ of
Theorem~\ref{theorem;DH}\eqref{item;DH3} is independent of the choice
of presentation of $\X$.  We omit the verification of this fact.
\end{remark}

\begin{example}\label{example10}
Consider the previous theorem in the context of
Examples~\ref{example;quasi} and~\ref{example;quasitoric}.  We have
$X_0=(\iota^*\circ\mu)\n(0)$, $X_1=\tilde{N}\ltimes X_0$, and
$G_\bu=\tilde{N}\ltimes X_0$, with moment map $\mu_\bu\colon
X_\bu\to\ann(\lie{n})$.  Let $U\subseteq\ann(\lie{n})$ be a small
neighborhood of $0$.  Note that the action of $G_\bu$ on $X_\bu$
satisfies the assumptions of Theorem~\ref{theorem;DH}.  The reduced
spaces over points of $U$ are then equivalent as smooth stacks to the
reduced space $\B\mu\n(0)/\B G_\bu$, which here is just equivalent to
a point.
\end{example}

\appendix

\section{Groups and actions in
  \texorpdfstring{$2$-categories}{2-categories}}
\label{appendix;groupobjects}

This is a review of group objects in $2$-categories and their actions.
The notion of a $2$-group internal to a $2$-category was developed by
Baez and Lauda~\cite{baez-lauda;2-groups}.  We modify their
terminology in two ways: our ``weak $2$-groups'' (resp.\ ``weak
homomorphisms'') are called ``coherent $2$-groups''
(resp.\ ``homomorphisms'') in~\cite{baez-lauda;2-groups}.

Throughout this appendix we denote by $\ca{C}$ a $2$-category with
finite products and with terminal object~$\star$.

\begin{definition}
\label{definition;internalgp}
A \emph{weak $2$-group in $\ca{C}$} consists of
\begin{enumerate}
\item
an object $G$ of $\ca{C}$;
\item
three $1$-morphisms: \emph{multiplication} $m\colon G\times G \to G$,
the \emph{group unit} $1\colon\star\to G$, and the \emph{inverse}
$(\cdot)\n\colon G\to G$;
\item
five $2$-isomorphisms:
\begin{enumerate}
\item
the \emph{associator} $\alpha\colon m\circ(m\times\id_G)\Rightarrow
m\circ(\id_G\times m)$,
\item
the \emph{left unit law} $\lambda\colon
m\circ(1\times\id_G)\Rightarrow\id_G$,
\item
the \emph{right unit law} $\rho\colon
m\circ(\id_G\times1)\Rightarrow\id_G$,
\item
the \emph{unit adjunction law} $d\colon1\circ\tau_G\Rightarrow
m\circ(\id_G\times(\cdot)\n)\circ \Delta_G$,
\item
the \emph{counit adjunction law} $e\colon
m\circ((\cdot)\n\times\id_G)\circ\Delta_G\Rightarrow1\circ\tau_G$.
\end{enumerate}
Here $\tau_G$ is the unique $1$-morphism $G\to\star$ and
$\Delta_G\colon G\to G\times G$ is the diagonal.
\end{enumerate}
The $2$-morphisms are subject to coherence conditions, which are that
the diagrams~\eqref{equation;pentagon}--\eqref{equation;secondzigzag}
below commute.  We write $\id=\id_G$ and $T\colon G \to G\times G
\times G$ for the three-way diagonal.
\begin{equation}\label{equation;pentagon}
\begin{tikzcd}[/tikz/column 1/.append style={anchor=base
      east},/tikz/column 3/.append style={anchor=base west}]
&m\circ(m\times m)
\ar[dr,Rightarrow,end anchor=north
  west,"\alpha\circ(\id\times\id\times m)"]&\\
m\circ(m\times\id)\circ(m\times\id\times\id) 
\ar[ur,Rightarrow,start anchor=north
  east,"\alpha\circ(m\times\id\times\id)"]
\ar[d,Rightarrow,start anchor=south east,end
  anchor=north east,shift right=8,"m\circ(\alpha\times\id)"']&&
m\circ(\id\times m)\circ(\id\times\id\times m)\\
m\circ(m\times\id)\circ(\id\times m\times\id)
\ar[rr,Rightarrow,"\alpha\circ(\id\times m\times\id)"]&&
m\circ(\id\times m)\circ(\id\times m\times\id)
\ar[u,Rightarrow,start anchor=north west,end anchor=south west,shift
  right=8,"m\circ(\id\times\alpha)"']
\end{tikzcd}
\end{equation}
\begin{equation}\label{equation;triangle}
\begin{tikzcd}[row sep=large,column sep=scriptsize]
m\circ(m\times\id)\circ(\id\times1\times\id)
\ar[rr,Rightarrow,"\alpha\circ(\id\times1\times\id)"]
\ar[dr,Rightarrow,start anchor=-10,"m\circ(\rho\times\id)"']&&
m\circ(\id\times m)\circ(\id\times1\times\id)
\ar[dl,Rightarrow,start anchor=-170,"m\circ(\id\times\lambda)"]\\
&m&
\end{tikzcd}
\end{equation}
\begin{equation}\label{equation;firstzigzag}
\begin{tikzcd}[cramped,row sep=large,/tikz/column 1/.append
    style={anchor=base east},/tikz/column 3/.append style={anchor=base
      west}]
m\circ(m\times\id)\circ(\id\times(\cdot)\n\times\id)\circ T
\ar[rr,Rightarrow,"\alpha\circ(\id\times(\cdot)\n\times\id)\circ T"]&&
m\circ(\id\times m)\circ(\id\times(\cdot)\n\times\id)\circ T\\
m\circ(1\times\id)
\ar[u,Rightarrow,"m\circ(d\times\id)",near end]
\ar[r,Rightarrow,"\lambda"]&\id
\ar[r,Rightarrow,"\rho\n"]&
m\circ(\id\times1)
\ar[u,Leftarrow,"m\circ(\id\times e)"',near end]&
\end{tikzcd}
\end{equation}
\begin{equation}\label{equation;secondzigzag}
\begin{tikzcd}[row sep=large,/tikz/column 1/.append style={anchor=base
      east},/tikz/column 3/.append style={anchor=base west}]
\text{\small$m\circ(\id\times m)\circ((\cdot)\n\times\id\times
  (\cdot)\n)\circ T$}
\ar[rr,Rightarrow,"\alpha\n\circ((\cdot)\n\times\id\times(\cdot)\n)\circ
  T"]&&
\text{\small$m\circ(m\times\id)\circ((\cdot)\n\times\id\times(\cdot)\n)\circ
  T$}
\ar[d,Rightarrow,"m\circ(e\times(\cdot)\n)",near start]\\
m\circ((\cdot)\n\times1)\ar[r,Rightarrow,"\rho"]
\ar[u,Rightarrow,"m\circ((\cdot)\n\times d)",near end]&
(\cdot)\n\ar[r,Rightarrow,"\lambda\n"]&m\circ(1\times(\cdot)\n)
\end{tikzcd}
\end{equation}
A \emph{strict $2$-group in $\ca{C}$} is a weak $2$-group in $\ca{C}$
for which the $2$-morphisms $\alpha$, $\lambda$, $\rho$, $d$,~$e$ are
all identity $2$-morphisms.  In other words, a strict $2$-group is a
group object in the underlying $1$-category of $\ca{C}$.
\end{definition}

\begin{definition}
\label{definition;hominternal2group}
Let $G$ and $G'$ be weak $2$-groups in $\ca{C}$.  A \emph{weak
  homomorphism} consists of a $1$-morphism $\phi_{(1)}=\phi\colon G\to
G'$ and two $2$-isomorphisms
\[
\phi_{(0)}\colon\phi\circ1\Longto1'\qquad\text{and}\qquad
\phi_{(2)}\colon m'\circ\phi\times\phi\Longto\phi\circ m.
\]
These are required to satisfy certain coherence conditions, namely
that the diagrams~\eqref{equation;homo1}--\eqref{equation;homo3} below
commute.
\begin{equation}\label{equation;homo1}
\begin{tikzcd}[column sep=huge,/tikz/column 1/.append
    style={anchor=base east},/tikz/column 2/.append style={anchor=base
      west}]
m'\circ(m'\times\id_{G'})\circ(\phi\times\phi\times\phi)
\ar[r,Rightarrow,"m'\circ(\phi_{(2)}\times\phi)"] 
\ar[d,Rightarrow,start anchor=south east,end anchor=north east,shift
  right=8,"\alpha\circ(\phi\times\phi\times\phi)"']&
m'\circ(\phi\times\phi)\circ(m\times\id_G)
\ar[d,Rightarrow,start anchor=south west,end anchor=north west,shift
  left=8,"\phi_{(2)}\circ(m\times\id_G)"]\\
m'\circ(\id_{G'}\times m')\circ(\phi\times\phi\times \phi)
\ar[d,Rightarrow,start anchor=south east,end anchor=north east,shift
  right=8,"m'\circ(\phi\times \phi_{(2)})"']&
\phi\circ m\circ(m\times\id_G)
\ar[d,Rightarrow,start anchor=south west,end anchor=north west,shift
  left=8,"\phi\circ\alpha"]\\
m'\circ(\phi\times\phi)\circ(\id_G\times m)
\ar[r, Rightarrow,"\phi_{(2)}\circ(\id_G\times m)"]&
\phi\circ m\circ(\id_G\times m)
\end{tikzcd}
\end{equation}
\begin{equation}\label{equation;homo2}
\begin{tikzcd}[column sep=huge,row sep=large,/tikz/column 1/.append
    style={anchor=base east},/tikz/column 2/.append style={anchor=base
      west}]
m'\circ(1'\times\phi)
\ar[r,Rightarrow,"\lambda'\circ\phi"]
\ar[d,Rightarrow,start anchor=south east,end anchor=north east,shift
  right=8,"m'\circ(\phi_{(0)}\times\phi)"']&
\phi\\
m'\circ(\phi\times\phi)\circ(1\times\id_G)
\ar[r,Rightarrow,"\phi_{(2)}\circ(1\times\id_G)"]&
\phi\circ m\circ(1\times\id_G)
\ar[u,Rightarrow,start anchor=north west,end anchor=south west,shift
  right=2,"\phi\circ\lambda"']
\end{tikzcd}
\end{equation}
\begin{equation}\label{equation;homo3}
\begin{tikzcd}[column sep=huge,row sep=large,/tikz/column 1/.append
    style={anchor=base east},/tikz/column 2/.append style={anchor=base
      west}]
m'\circ(\phi\times1')
\ar[r,Rightarrow,"\rho'\circ\phi"]
\ar[d,Rightarrow,start anchor=south east,end anchor=north east,shift
  right=8,"m'\circ(\phi\times\phi_{(0)})"']&
\phi\\
m'\circ(\phi\times\phi)\circ(\id_G\times1)
\ar[r,Rightarrow,"\phi_{(2)}\circ(\id_G\times1)"]&
\phi\circ m\circ(\id_G\times1)
\ar[u,Rightarrow,start anchor=north west,end anchor=south west,shift
  right=2,"\phi\circ \rho"']
\end{tikzcd}
\end{equation}
A \emph{strict homomorphism} is a weak homomorphism $\phi$ for which
$\phi_{(0)}$ and $\phi_{(2)}$ are both identities.
\end{definition}

\begin{definition}
Let $\phi$, $\psi\colon G\to G'$ be weak homomorphisms of weak
$2$-groups in $\ca{C}$.  A \emph{$2$-homomorphism} is a $2$-morphism
$\theta\colon\phi\Rightarrow\psi$ so that the following diagrams
commute:
\[
\begin{tikzcd}[sep=large,/tikz/column 1/.append style={anchor=base
      east},/tikz/column 2/.append style={anchor=base west}]
m'\circ(\phi\times\phi)\ar[r,Rightarrow,"m'\circ(\theta\times\theta)"]
\ar[d,Rightarrow,start anchor=south east,end anchor=north east,shift
  right=6,"\phi_{(2)}"']&m'\circ(\psi\times\psi)
\ar[d,Rightarrow,start anchor=south west,end anchor=north west,shift
  left=6,"\psi_{(2)}"]\\
\phi\circ m\ar[r,Rightarrow,"\theta\circ m"]&\psi\circ m
\end{tikzcd}
\qquad
\begin{tikzcd}[row sep=large]
1'\ar[r,Leftarrow,"\psi_{(0)}"]\ar[d,Leftarrow,"\phi_{(0)}"']&
\psi\circ1\\
\phi\circ1\ar[ur,Rightarrow,"\theta\circ1"']&
\end{tikzcd}
\]
\end{definition}

\begin{definition}
\label{definition;equivalence}
A weak homomorphism $\phi\colon G\to G'$ of weak $2$-groups is an
\emph{equivalence} if there exists a weak homomorphism $\psi\colon
G'\to G$, together with $2$-homomorphisms
$\psi\circ\phi\Rightarrow\id_G$ and
$\phi\circ\psi\Rightarrow\id_{G'}$.  In this case $\psi$ is called a
\emph{weak inverse} to $\phi$ and we write $\psi=\phi\n$.
\end{definition}

The following definition is modeled
on~\cite[Part~A]{romagny;group-actions-stacks}; see
also~\cite[\S\,3.2]{bursztyn-noseda-zhu;principal-stacky-groupoids}.

\begin{definition}
\label{definition;actioninternal}
Let $G$ be a strict $2$-group in $\ca{C}$ and let $X$ be an object of
$\ca{C}$.  A \emph{weak action of $G$ on $X$} consists of the
following data: an \emph{action morphism} $a\colon G\times X\to X$,
and two $2$-isomorphisms
\[
\beta\colon a\circ(m\times\id_X)\Rightarrow a\circ(\id_G\times a),
\qquad\epsilon\colon a\circ(1\times\id_X)\Rightarrow\id_X.
\]
The $2$-morphisms $\beta$ and $\epsilon$ are subject to the condition
that the
diagrams~\eqref{equation;actioncoh1}--\eqref{equation;actioncoh1}
below commute.  For simplicity, we omit notation for horizontal
composition of $1$-morphisms and $2$-morphisms, and display only the
$2$-morphism being used.
\begin{equation}\label{equation;actioncoh1}
\begin{tikzcd}[row sep=large,/tikz/column
2/.append style={anchor=base west}]
a\circ\bigl((m\circ(\id_G\times1))\times{\id_X}\bigr)
\ar[r,Rightarrow,"\beta"]
\ar[dr,equal,start anchor=south east,end anchor=north west,shift
  right]&
a\circ\bigl({\id_G}\times(a\circ(1\times{\id_X}))\bigr)
\ar[d,Rightarrow,start anchor=south west,end anchor=north west,shift
  left=8,"\epsilon"]\\
&a\circ(\pr_G\times\id_X)
\end{tikzcd}
\end{equation}
\begin{equation}\label{equation;actioncoh2}
\begin{tikzcd}[row sep=large,/tikz/column
2/.append style={anchor=base west}]
a\circ\bigl((m\circ(1\times\id_G))\times{\id_X}\bigr)
\ar[r,Rightarrow,"\beta"]
\ar[dr,equal,start anchor=south east,end anchor=north west,shift
  right]&
a\circ(1\times a)
\ar[d,Rightarrow,start anchor=south west,end anchor=north west,shift
  left=8,"\epsilon"]\\
&a\circ(\pr_G\times\id_X)
\end{tikzcd}
\end{equation}
\begin{equation}\label{equation;actioncoh3}
\begin{tikzcd}[row sep=large,column sep=small,/tikz/column 1/.append
    style={anchor=base east},/tikz/column 3/.append style={anchor=base
      west}]
a\circ\bigl((m\circ(m\times\id_G))\times{\id_X}\bigr)
\ar[r,Rightarrow,"\beta"] 
\ar[d,equal,start anchor=south east,end anchor=north east,shift
  right=8]&
a\circ(m\times a)
\ar[r,Rightarrow,"\beta"]&
a\circ\bigl(\id_G\times(a\circ(\id_G\times a))\bigr)\\
a\circ\bigl((m\circ(\id_G\times m))\times{\id_X}\bigr)
\ar[rr,Rightarrow,"\beta"]&&
a\circ\bigl(\id_G\times(a\circ(m\times{\id_X}))\bigr).
\ar[u,Rightarrow,start anchor=north west,end anchor=south west,shift
  right=8,"\beta"']
\end{tikzcd}
\end{equation}
A \emph{strict action} is a weak action $a$ for which the
$2$-morphisms $\beta$ and $\epsilon$ are both identities.
\end{definition}

\begin{definition}
\label{definition;equivariantinternal}
Let $G$ be a strict $2$-group in $\ca{C}$.  Given weak actions
\[a\colon G\times X\to X,\qquad a'\colon G\times X'\to X'\]
on objects $X$ and $X'$ of $\ca{C}$, a \emph{weakly $G$-equivariant
  map} consists of a $1$-morphism $F\colon X\to X'$ and a
$2$-isomorphism
\[\delta\colon a'\circ(\id_G\times F)\To F\circ a,\]
which are subject to the condition that the
diagrams~\eqref{equation;equivariant1}--\eqref{equation;equivariant2}
commute.  We omit horizontal composition of $1$-morphisms, as above.
\begin{equation}\label{equation;equivariant1}
\begin{tikzcd}[row sep=large,column sep=small,/tikz/column 1/.append
    style={anchor=base east},/tikz/column 3/.append style={anchor=base
      west}]
a'\circ(m\times F)
\ar[r,Rightarrow,"\beta'"]
\ar[d,Rightarrow,start anchor=south east,end anchor=north east,shift
  right=8,"\delta"']&
a'\circ(\id_G\times a')\circ(\id_G\times\id_G\times F)
\ar[r,Rightarrow,"\delta"]&
a'\circ(\id_G\times F)\circ(\id_G\times a)
\ar[d,Rightarrow,start anchor=south west,end anchor=north west,shift
  left=8,"\delta"]\\
F\circ a\circ(m\times\id_X)
\ar[rr,Rightarrow,"\beta"]&&
F\circ a\circ(\id_G\times a)
\end{tikzcd}
\end{equation}
\begin{equation}\label{equation;equivariant2}
\begin{tikzcd}[row sep=large,/tikz/column 1/.append style={anchor=base
      east}]
a'\circ(1\times F)
\ar[r,Rightarrow,"\epsilon'"]
\ar[d,Rightarrow,start anchor=south east,end anchor=north east,shift
  right=8,"\delta"']&F\\
F\circ a\circ(1\times\id_X)
\ar[ur,Rightarrow,start anchor=north east,"\epsilon"']
\end{tikzcd}
\end{equation}
If $\delta$ is the identity $2$-morphism, then $F$ is \emph{strictly
  equivariant}.
\end{definition}

\section{Strictification of stacky actions}\label{appendix;strict}

This appendix contains the proof of the following result, which is
Theorem~\ref{theorem;strictaction} in the main text.

\begin{theorem}\label{theorem;appendixAthm}
Let $\G$ be a connected (strict) Lie group stack acting weakly on an
\'etale differentiable stack $\X$.  Suppose that $\G$ admits a
presentation $\B G_\bu\simeq\G$ by a base-connected Lie $2$-group
$G_\bu$.  For every such presentation $\B G_\bu\simeq\G$ there exists
a presentation $\B X_\bu\simeq\X$ of $\X$ by a Lie groupoid $X_\bu$ so
that
\begin{enumerate}
\item\label{item;strict-action}
$G_\bu$ acts strictly on $X_\bu$;
\item\label{item;weak-equivariant}
identifying $\B G_\bu=\G$, the equivalence $\B X_\bu\simeq\X$ is
weakly $\G$-equivariant.
\end{enumerate}
\end{theorem}

We first prove the following.

\begin{proposition}\label{proposition;strictatlas}
Let $G_\bu$ be a base-connected Lie $2$-group, let $\G=\B G_\bu$, and
let
\[\ba\colon\G\times\X\to\X\]
be a weak action of $\G$ on a differentiable stack $\X$.  Let
$\bp\colon Y_0\to\X$ be an atlas, let $X_0=G_0\times Y_0$, and
consider the map
\[
\begin{tikzcd}
\bo{b}\colon X_0=G_0\times Y_0\ar[r]&\G\times\X\ar[r,"\ba"]&\X.
\end{tikzcd}
\]
The map $\bo{b}$ is a representable epimorphism and a submersion.  So
$\bo{b}\colon X_0\to\X$ is an atlas and hence determines a canonical
equivalence $\B X_\bu\simeq\X$, where $X_\bu$ is the Lie groupoid
$X_0\times_\X X_0\rightrightarrows X_0$.
\end{proposition}

\begin{proof}
Consider the $2$-cartesian square
\[
\begin{tikzcd}[row sep=large]
(G_0\times Y_0)\times_{\X}Y_0\ar[r,"q"]\ar[d,"{\pr_1}"']
  &Y_0\ar[d,"{\bp}"]\\
X_0=G_0\times Y_0\ar[r,"{\bo{b}}"]&\X
\end{tikzcd}
\]
where $\pr_1$ and $q$ are the canonical (up to $2$-isomorphism) maps
out of the fibred product.  Then $(G_0\times Y_0)\times_{\X}Y_0$ is
(equivalent to) a manifold, since $\bp$ is representable.  By Lemmas
2.2 and 2.3 of~\cite{behrend-xu;stacks-gerbes}, to show that $\bo{b}$
is a representable submersive epimorphism it is enough to show that
$q$ is a surjective submersion.  To see this, consider any $g\in G_0$
and form the following diagram composed of two $2$-cartesian squares:
\begin{equation}\label{equation;double-square}
\begin{tikzcd}[row sep=large]
(\{g\}\times Y_0)\times_{\X}Y_0\ar[r,hook]\ar[d]&(G_0\times
  Y_0)\times_{\X}Y_0\ar[r,"q"]\ar[d]&Y_0\ar[d,"\bp"]\\
\{g\}\times Y_0\ar[r,hook]&G_0\times Y_0\ar[r,"\bo{b}"]&\X
\end{tikzcd}
\end{equation}
We assert that the bottom row
\begin{equation}\label{eqn;twistedatlas}
\{g\}\times Y_0\longinj G_0\times Y_0\longto\X
\end{equation}
is an atlas for $\X$.  Indeed, applying the functor $\B$ gives us a
categorical point $\bo{g}\colon\star\to\G$, which determines a
morphism of stacks
\[
\begin{tikzcd}
\bo{L}_g\colon\X\simeq\star\times\X
\ar[r,"{\bo{g}\times{\id}}"]&\G\times\X\ar[r,"\ba"]&\X.
\end{tikzcd}
\]
By the axioms for the weak $\G$-action on $\X$, the morphism
$\bo{L}_g$ is an equivalence with weak inverse $\bo{L}_{g\n}$.  The
composition~\eqref{eqn;twistedatlas} is naturally isomorphic to
$\bo{L}_g\circ\bp\colon X_0\to\X$.  Since $\bo{L}_g$ is an equivalence
and $\bp$ is an atlas, we have that $\eqref{eqn;twistedatlas}$ is an
atlas.  Since being an epimorphism and being a submersion are stable
under pullbacks, this implies that the top row
of~\eqref{equation;double-square}
\[
(\{g\}\times Y_0)\times_\X Y_0\longinj(G_0\times
Y_0)\times_{\X}Y_0\longto Y_0
\]
is a surjective submersion.  It follows immediately that $q$ is a
surjective submersion.
\end{proof}

\begin{proof}[Proof of Theorem~\ref{theorem;appendixAthm}]
Let $Y_0\to \X$ be any atlas and let $X_\bu$ be as in
Proposition~\ref{proposition;strictatlas}.  Each square of the
following diagram $2$-commutes:
\begin{equation}\label{equation;strictification}
\begin{tikzcd}
G_0\times G_0\times Y_0 \ar[r,"{m_0\times{\id}}"]\ar[d]&G_0\times
Y_0\ar[d]\\
\G\times\G \times\X\ar[r,"{\bm\times{\id}}"]
\ar[d,"{{\id}\times\ba}"']&\G\times\X\ar[d,"\ba"]\\
\G\times\X\ar[r,"\ba"]&\X.
\end{tikzcd}
\end{equation}
There is a Lie groupoid morphism $A_\bu\colon G_\bu\times X_\bu\to
X_\bu$, where $A_0=m_0 \times{\id}$ and $\B A_\bu \cong\ba$.  Then
$A_0$ determines an action of $G_0$ on $X_0$.  We write
\[
A_0(g,x)=g\cdot x,\qquad A_1(h,y)=h\cdot y
\]
for $g\in G_0$, $x\in X_0$, $h\in G_1$, and $y\in X_1$.  We will show
that $A_\bu$ is naturally isomorphic to a strict action
$\tilde{A}_\bu\colon G_\bu\times X_\bu\to X_\bu$ with
$\tilde{A}_0=A_0$.  Note that despite our notation, $A_1$ is not in
general an action.

There are $2$-isomorphisms which come with the action $\ba$ as in
Definition~\ref{definition;actioninternal}, and we must write these in
terms of $A_\bu$.  To do this, we make use of the following fact: let
$F_\bu$, $F_\bu'$ be morphisms of Lie groupoids and let
$\balpha\colon\B F\Rightarrow\B F'$ be a $2$-arrow in the category of
stacks.  Then there is a unique natural isomorphism $\alpha\colon
F\Rightarrow F'$ in the category of Lie groupoids so that
$\balpha=\B\alpha$.  This assertion follows from the equivalence of
$\Diffstack$ and $\Liegpd[\ca{M}\n]$ and the description of $2$-cells
in \cite[\S\,2.3]{pronk;etendues-fractions}.  It was also shown
in~\cite[Theorem~2.5]{berwick-evans-lerman;lie-2-algebras-vector}.

The $2$-isomorphisms which come with $\ba$ then translate as
follows. There are $2$-arrows
\begin{align*}
\beta&\colon A_\bu\circ(m_\bu\times{\id})\Longrightarrow
A_\bu\circ({\id}\times A_\bu);\\
\beta&\colon G_0\times G_0\times X_0\longto X_1;\\
\epsilon&\colon A_\bu\circ(1\times{\id})\Longrightarrow{\id};\\
\epsilon&\colon X_0\longto X_1,
\end{align*}
where $1\colon\star\to G_\bu$ is the $2$-group unit, and we write
\[
1\times{\id}\colon X_\bu\cong\star\times X_\bu\longto G_\bu\times
X_\bu.
\]
We also write the coherence conditions on $\epsilon$ and $\beta$.  For
$k$, $h$, $g\in G_0$ and $x\in X_0$, they are as follows.
\begin{gather}
\label{coherenceA}
(u(g)\cdot\epsilon(x))\circ\beta(g,1,x)=u(g\cdot x);\\
\label{coherenceB}
\epsilon(g\cdot x)\circ\beta(1,g,x)=u(g\cdot x);\\
\label{coherenceC}
\beta(k,h,g\cdot
x)\circ\beta(kh,g,x)=(u(k)\cdot\beta(h,g,x))\circ\beta(k,hg,x).
\end{gather}
Since $A_0$ is an action, we have that $\epsilon(x)$ is an arrow from
$x$ to $x$ and $\beta(h,g,x)$ is an arrow from $(hg)\cdot x$ to
$(hg)\cdot x$.

%
Now fix $x\in X_0$. Define smooth maps $\gamma_1,\gamma_2\colon
G_0\times G_0 \to s\n(x)\cap t\n(x)$ as follows.
\begin{align*}
\gamma_1(k,j) & = u( (kj)\n) \cdot \beta(k,1,j\cdot x)\\
\gamma_2(k,j) & = u( (kj)\n ) \cdot \beta(k,j,x).
\end{align*}
Then, since $G_0$ is connected and $X_\bu$ has discrete isotropy
groups, the maps $\gamma_1$ and $\gamma_2$ are constant.  Moreover,
$\gamma_1(1,1)=1\cdot\beta(1,1,x)=\gamma_2(1,1)$, so
$\gamma_1=\gamma_2$ and thus
\[u((kj)\n)\cdot \beta(k,1,j\cdot x) = u((kj)\n) \cdot \beta(k,j,x)\]
for all $k$, $j\in G_0$.  Therefore,
\[
u(kj)\cdot (u((kj)\n)\cdot \beta(k,1,j\cdot x) )=u(kj)\cdot( u((kj)\n)
\cdot \beta(k,j,x))
\]
Applying the natural transformation $\beta$ to both sides, one finds
\begin{multline*}
\beta(kj,(kj)\n,kj\cdot x)\circ\left(1\cdot\beta(k,1,j\cdot
x)\right)\circ\beta(kj,(kj)\n,kj\cdot x)\n\\
=\beta(kj,(kj)\n,kj\cdot x)\circ\left(1\cdot\beta(k,j,
x)\right)\circ\beta(kj,(kj)\n,kj\cdot x)\n,
\end{multline*}
and so
\[1\cdot \beta(k,1,j\cdot x) = 1\cdot \beta(k,j,x).\]
Applying the natural transformation $\epsilon$ to both sides, one has
\[
\epsilon(kj\cdot x)\n \circ \beta(k,1,j\cdot x) \circ \epsilon(kj\cdot
x) = \epsilon(kj\cdot x)\n \circ \beta(k,j,\cdot x) \circ
\epsilon(kj\cdot x),
\]
and therefore
\begin{equation}\label{equation;betanice}
\beta(k,1,j\cdot x)=\beta(k,j,x).
\end{equation}
As a consequence, by applying the coherence
condition~\eqref{coherenceA} we have
\begin{equation}\label{equation;beta}
u(k)\cdot\epsilon(j\cdot x)=\beta(k,1,j\cdot x)\n=\beta(k,j,x)\n.
\end{equation}

Now define $\tilde{A}_\bu\colon G_\bu\times X_\bu\to X_\bu$ by
\begin{gather*}
\tilde{A}_0(g,x)=A_0(g,x)=g\cdot x\\
\tilde{A}_1(k,f)=\epsilon(t(k)\cdot t(f))\circ(k\cdot
f)\circ\epsilon(s(k)\cdot s(f))\n.
\end{gather*}
We write $\tilde{A}_1(k,f)=k\odot f$.  That $\tilde{A}_\bu$ is a Lie
groupoid morphism is easy to check.  The smooth map
\[\epsilon\circ A_0\colon G_0\times X_0\to X_1\]
is a natural isomorphism from $A_\bu$ to $\tilde{A}_\bu$.  It remains
to check that $\tilde{A}_1\colon G_1\times X_1\to X_1$ is an action.
First,
\[
1\odot f=\epsilon(t(f))\circ(1\cdot f)\circ\epsilon(s(f))\n=f,
\]
since $\epsilon$ is a natural isomorphism
$A_\bu\circ(1\times{\id})\Rightarrow{\id}$.  Next,
\begin{align*}
k\odot(j\odot f)&=k\odot\left(\epsilon(t(j)\cdot t(f))\circ(j\cdot
f)\circ\epsilon(s(j)\cdot s(f))\n\right)\\
&=\epsilon(t(kj)\cdot t(f))\circ F\circ\epsilon(s(kj)\cdot s(f))\n,
\end{align*}
where $F=k\cdot\left(\epsilon(t(j)\cdot t(f))\circ(j\cdot
f)\circ\epsilon(s(j)\cdot s(f))\n\right)$.  On the other hand,
\[
kj\odot f=\epsilon(t(kj)\cdot t(f))\circ\left[kj\cdot
  f\right]\circ\epsilon(s(kj)\cdot s(f))\n.
\]
We must then show that
\begin{equation}\label{equation;strictgoal}
F=kj\cdot f.
\end{equation}
We compute
\begin{equation}\label{equation;longcomp}
\begin{aligned}
F&=\left[ u(t(k)) \circ k \circ u(s(k))\right] \cdot \left[
  \epsilon(t(j)\cdot t(f)) \circ (j\cdot f) \circ \epsilon(s(j)\cdot
  s(f)) \n \right]\\
&=\left[u(t(k))\cdot\epsilon(t(j)\cdot
  t(f))\right]\circ\left[k\cdot(j\cdot
  f)\right]\circ\left[u(s(k))\cdot\epsilon(s(j)\cdot s(f))\n\right],
\end{aligned}
\end{equation}
since $A_\bu$ is a groupoid morphism.  From~\eqref{equation;beta}, one
has that~\eqref{equation;longcomp} is equal to
\[
\beta(t(k),t(j),t(f))\n \circ \left[k\cdot(j\cdot f)\right] \circ
\beta(s(k),s(j),s(f)).
\]
But since $\beta$ is a natural transformation
$A_\bu\circ(m_\bu\times{\id})\Rightarrow A_\bu\circ({\id}\times
A_\bu)$ this is just $kj\cdot f$.  So we have
found~\eqref{equation;strictgoal}, and
proved~\eqref{item;strict-action}.

To prove~\eqref{item;weak-equivariant}, assume $\G=\B G_\bu$.  It
suffices to show that the identity $1$-morphism $X_\bu\to X_\bu$
together with the $2$-isomorphism $\delta=\epsilon\circ A_0\colon
G_0\times X_0\to X_1$ are the data of a weakly $G_\bu$-equivariant map
intertwining the strict action $\tilde{A}_\bu$ with the weak action
$A_\bu$.  Indeed, the coherence
condition~\eqref{equation;equivariant1} becomes the
requirement that for any $g,h\in G_0$ and $x\in X_0$,
\[
\epsilon(gh\cdot x)\circ\bigl(u(g)\cdot\epsilon(h\cdot
x)\bigr)\circ\beta(g,h,x)=\epsilon(gh\cdot x).
\]
This follows from the coherence condition~\eqref{coherenceA} above
together with~\eqref{equation;betanice}.  The coherence
condition~\eqref{equation;equivariant2} is automatic from the
definition of $\delta$.
\end{proof}

\section{Weak fibred products of Lie group stacks (by C. Zhu)}
\label{appendix;zhu}

This appendix is devoted to the proof of the following result, which
is Theorem~\ref{theorem;pullbackgroups} in the main text.

\begin{theorem}
Let $\G\to\HH$ and $\G'\to\HH$ be weak homomorphisms of strict Lie
group stacks, and assume that the fibred product of stacks
$\KK=\G\times_{\HH}\G'$ is a differentiable stack.  Then $\KK$ is
naturally a weak Lie group stack, and the projections $\KK\to\G$ and
$\KK\to\G'$ are strict homomorphisms.
\end{theorem}

\begin{proof}
It suffices to prove the statement object-wise.  That is, for an
object $U\in\group{Diff}$, the groupoid
$\KK(U)=\G(U)\times^{(w)}_{\HH(U)}\G'(U)$ is a weak $2$-group.
Because $\KK$ is assumed to be a differentiable stack, it will then
automatically be a Lie group stack.

Let us denote
\begin{align*}
G=(G_1\rightrightarrows G_0)&:=\G(U)\\
G'=(G'_1\rightrightarrows G'_0)&:=\G'(U)\\
H=(H_1\rightrightarrows H_0)&:=\HH(U)
\end{align*}
and the maps between them $\phi\colon G\to H$ and $\phi'\colon G'\to
H$.  We now define the group structure maps on $G\times^{(w)}_H G'$.

\emph{Multiplication}.  The multiplication
\[
\tilde{m}\colon (G\times^{(w)}_H G')\times (G\times^{(w)}_H G') =
G^{\times 2}\times^{(w)}_{H^{\times 2}} {G'}^{\times 2}\longto
(G\times^{(w)}_H G')
\]
is essentially given by $m\times m'$, where $m$ and $m'$ are the group
multiplication maps for $G$ and $G'$ respectively.  We denote by
$\cdot$ the multiplications $m$, $m'$, $\tilde{m}$, etc.  Then on the
level of objects of $G\times^{(w)}_H G'$, define:
\[
(g^1_0, h^1_1, {g'}^1_0)\cdot (g^2_0, h^2_1, {g'}^2_0) := (g^1_0\cdot
g^2_0, \phi_{(2)}'\circ (h^1_1\cdot h^2_1)\circ\phi_{(2)}\n, {g'}^1_0
\cdot {g'}^2_0 ),
\]
where $\phi_{(2)}\colon\phi_0(g^1_0)\cdot \phi_0(g^2_0) \to
\phi_0(g^1_0\cdot g^2_0 )$ comes from the $2$-isomorphism data of the
weak homomorphism $\phi$; and similarly for $\phi_{(2)}'$.  The middle
element in $H_1$ is then obtained by the following composition
\[
\phi_0(g^1_0\cdot g^2_0 )\xrightarrow{\phi_{(2)}\n}\phi_0(g^1_0)\cdot
\phi_0(g^2_0)\xrightarrow{h^1_1\cdot h^2_1}\phi_0({g'}^1_0)\cdot
\phi_0({g'}^2_0)\xrightarrow{\phi_{(2)}'}\phi_0(g'^1_0\cdot g'^2_0).
\]
On the level of morphisms, define
\[
(g^1_1, h^1_1, {g'}^1_1)\cdot (g^2_1, h^2_1, {g'}^2_1):=(g^1_1\cdot
g^2_1, {\phi_{(2)}'}\circ (h^1_1\cdot h^2_1)\circ\phi_{(2)}\n,
{g'}^1_1\cdot {g'}^2_1)
\]
where $s(g^1_1,h^1_1,{g'}^1_1)=(g^1_0,h^1_1,{g'}^1_0)$ and
$s(g^2_1,h^2_1,{g'}^2_1)=(g^2_0,h^2_1,{g'}^2_0)$.

\emph{Associator}.  The associator $\tilde{\alpha}$ for $\tilde{m}$ is
a natural transformation, whose value at the point
\[
\bigl((g_0^1, g_0^2,g_0^3), (h_1^1, h_1^2,h_1^3), ({g'}_0^1,
     {g'}_0^2,{g'}_0^3)\bigr)\in((G\times_H^{(w)} G')_0)^3
\]
is given by the following element in $\bigl(G\times^{(w)}_H
G'\bigr)_1$,
\begin{multline*}
\bigl((g_0^1\cdot g_0^2)\cdot g_0^3,\overline{\phi}_{(2)}'\circ
     [(h_1^1\cdot h_1^2)\cdot h_1^3]\circ\overline{\phi}_{(2)}\n,
     ({g'}_0^1 \cdot {g'}_0^2)\cdot {g'}_0^3\bigr)\\
\xrightarrow{\bigl(\alpha(g_0^1,g_0^2,g_0^3),
  \overline{\phi}_{(2)}'\circ[(h_1^1\cdot h_1^2)\cdot h_1^3]
  \circ\overline{\phi}_{(2)}\n,\alpha'({g'}_0^1,{g'}_0^2,{g'}_0^3)\bigr)}\\
\bigl(g_0^1\cdot (g_0^2\cdot
g_0^3),(\hat{\phi}_{(2)}')\circ[h_1^1\cdot(h_1^2\cdot
  h_1^3)]\circ\hat{\phi}_{(2)}\n,{g'}_0^1\cdot({g'}_0^2\cdot{g'}_0^3)\bigr),
\end{multline*}
where the $1$-morphisms
\begin{align*}
\overline{\phi}_{(2)}\colon & (\phi_0(g_0^1)\cdot \phi_0(g_0^2)) \cdot
\phi_0(g_0^3) \to \phi_0((g_0^1\cdot g_0^2)\cdot g_0^3)
\\ \hat{\phi}_{(2)} \colon & \phi_0(g_0^1) \cdot (\phi_0(g_0^2)
\cdot\phi_0(g_0^3))\to \phi_0(g_0^1\cdot(g_0^2\cdot g_0^3))
\end{align*}
are defined using the $2$-isomorphism data $\phi_{(2)}$, and similarly
for $\overline{\phi}_{(2)}'$ and $\hat{\phi}_{(2)}'$.  It is an arrow
between these two objects thanks to the compatibility of $\phi_{(2)}$
and the associator $\alpha$, and that of $\phi_{(2)}'$ and $\alpha'$.

Naturality of $\tilde{\alpha}$ comes from that of $\alpha$ and
$\alpha'$.  The pentagon identity for $\tilde{\alpha}$ is implied by
that of $\alpha$ and $\alpha'$.

\emph{Identity}.  Similarly, we define the identity map
$\tilde{1}\colon *\to G\times^{(w)}_H G' $ to be essentially $(1,
1')$.  More precisely,
$\tilde{1}_0:=(1_0,\phi_{(0)}'^{-1}\circ\phi_{(0)}, 1'_0)\in
(G\times^{(w)}_HG')_0$, where we have
\[
\phi_0(1_0)\xrightarrow{\phi_{(0)}}
1_0^H\xrightarrow{\phi_{(0)}'^{-1}} \phi_0(1'_0).
\]
Here $\phi_{(0)}$ and $\phi_{(0)}'$ come from the $2$-isomorphism data
in the definition of $\phi$ and $\phi'$.  Then
$\tilde{1}_1:=(1_1,\phi_{(0)}'^{-1} \circ\phi_{(0)}, 1'_1)$.  The
$2$-morphisms involving $\tilde{1}$ for $\G\times_{\HH}\G'$ are
similarly defined by those for $1$ and $1'$, and the coherence
conditions for these $2$-morphisms are similarly implied by that for
$\G$ and $\G'$.

\emph{Inverse}. The inverse map $\tilde{\imath}\colon G\times^{(w)}_H
G'\to G\times^{(w)}_H G'$ is essentially $(i, i')$, where $i$ is the
inverse map for $G$ and $i'$ that for $G'$. More precisely, on objects
\[
\tilde{\imath}_0(g_0, h_1, g'_0):=(i_0(g_0),\gamma'^{-1}\circ
i_1(h_1)\circ\gamma, i'_0(g'_0)),
\]
where
\[
\phi_0(i_0(g_0))\xrightarrow{\gamma}i_0(\phi_0(g_0))
\xrightarrow{i_1(h_1)} i_0(\phi_0(g'_0))\xrightarrow{\gamma'^{-1}}
\phi_0.
\]
And on arrows
\[
\tilde{\imath}_1(g_1, h_1, g'_1):=(i_1(g_1),\gamma'^{-1}\circ h_1^{-1}
\circ\gamma, i'_1(g'_1)).
\]
Here $\gamma$ is uniquely determined by the $2$-isomorphism data of
$\phi$ and satisfies certain coherence laws;
see~\cite[\S\,6]{baez-lauda;2-groups}.  The $2$-morphisms involving
$\tilde{\imath}$ for $\G\times_{\HH}\G'$ are similarly defined by
those for $i $ and $i'$, and the coherence conditions for these
$2$-morphisms are similarly implied by that for $\G$ and $\G'$.

\emph{Projections}.  From the construction of the group structure maps
$\tilde{m},\tilde{1},\tilde{\imath}$ for $\KK$, we see that the
projections to $\G$ and $\G'$ are naturally strict homomorphisms of
Lie group stacks.
\end{proof}

\input{hamstack.gls}


\bibliographystyle{amsplain}

\bibliography{hamilton}


\end{document}